\documentclass[12pt]{article}

    \usepackage{amsmath}
    \usepackage{amsthm}
    \usepackage{amssymb}
    \usepackage{amsfonts}
    \usepackage{color}
\usepackage{hyperref}

\usepackage{graphicx}
\usepackage{wrapfig}

\newcommand{\XX}{\mathbb{X}}
\newcommand{\Prob}{\mathbb{P}}

\renewcommand{\H}{{\cal H}}

\newcommand{\X}{{\cal X}}
\newcommand{\0}{{\bf 0}}

    \newtheorem{theo}{Theorem}\numberwithin{theo}{section}
    \newtheorem{coro}[theo]{Corollary}

    \newtheorem{lemm}[theo]{Lemma}

        \def\N{\mathbb{N}}

\numberwithin{equation}{section}

    \def\E{\mathbb{E}}
    
    \def\0{{\bf 0}}
    
    \def\R{\mathbb{R}}

    \renewcommand{\E}{\mathbb E \,}

    \newcommand{\MM}{\mathbb{M}}

\newcommand{\M}{{\cal M}}
    \newcommand{\Q}{{\cal Q}}

    \renewcommand{\P}{{{\cal P}}}

    \newcommand{\A}{{\cal A}}

    \newcommand{\Var}{\operatorname{Var}}

    \newcommand{\conv}{{\rm conv}}

    \newcommand{\Vol}{{\rm Vol}}
    \newcommand{\diam}{{\rm diam}}

\newcommand{\equlaw}{\stackrel{(d)}{=}}

    \def\bdm{\begin{displaymath}}
    \newcommand{\edm}{\end{displaymath}}
    \def\benu{\begin{enumerate}}
    \def\eenu{\end{enumerate}}
    \def\beqn{\begin{equation}}
    \def\eeqn{\end{equation}}
    \def\be{\begin{equation}}
    \def\ee{\end{equation}}
    \def\bea{\begin{eqnarray}}
    \def\eea{\end{eqnarray}}
    \newcommand{\bean}{\begin{eqnarray*}}
    \newcommand{\eean}{\end{eqnarray*}}
    \newcommand{\bear}{\begin{eqnarray}}
    \newcommand{\eear}{\end{eqnarray}}

    \def\R{\mathbb{R}}
    \def\Q{\mathbb{Q}}
    \def\D{{\mathbf{D}}}
    \def\d{\operatorname{d}}

    \def\ka{{\kappa}}

    \def\qed{\hfill\hbox{${\vcenter{\vbox{
        \hrule height 0.4pt\hbox{\vrule width 0.4pt height 6pt
        \kern5pt\vrule width 0.4pt}\hrule height 0.4pt}}}$}}

\def\dint{\textup{d}}

\newcommand{\at}{\makeatletter @\makeatother}

\begin{document}

\title{\bf {Normal approximation for stabilizing functionals}}

    \author{Raphael Lachi\`eze-Rey\footnotemark[1]{}, Matthias Schulte\footnotemark[2]{}, J. E. Yukich\footnotemark[3]}

    \date{\today}
    \maketitle
		
\footnotetext[1]{Universit\'e Paris Descartes, Sorbonne Paris Cit\'e, raphael.lachieze-rey{\at}parisdescartes.fr}
\footnotetext[2]{Universit\"at Bern, matthias.schulte{\at}stat.unibe.ch}
\footnotetext[3]{Lehigh University, jey0{\at}lehigh.edu}

\footnotetext{Research supported in part by  NSF grant  DMS-1406410 (JY). This research was initiated during the workshop `New Directions in Stein's Method' at the Institute for Mathematical Sciences in Singapore held during May 2015. The authors are very grateful to the IMS for their hospitality and support.}

 \begin{abstract}

We establish presumably optimal rates of normal convergence with respect to the Kolmogorov distance for a large class of geometric functionals of marked Poisson and binomial point processes on general metric spaces. The rates  are valid whenever the geometric functional is expressible as a sum of  exponentially stabilizing score functions satisfying a moment condition.  By incorporating stabilization methods into the Malliavin-Stein theory,  we obtain rates of normal approximation for sums of stabilizing score functions which either improve upon existing rates or are the first of their kind.

Our general rates hold for functionals of marked input on spaces more general than full-dimensional subsets of $\R^d$, including $m$-dimensional Riemannian manifolds, $m\leq d$. We use the general results to deduce improved and new rates of normal convergence for several functionals in stochastic geometry, including those whose variances re-scale as the volume or the surface area of an underlying set. In particular, we improve upon rates of normal convergence for the $k$-face and $i$th intrinsic volume functionals of the convex hull of Poisson and binomial random samples in a smooth convex body in  dimension $d\geq 2$.
We also provide improved rates of normal convergence for statistics of nearest neighbors graphs and high-dimensional data sets, the number of maximal points in a random sample, estimators of surface area and volume arising in set approximation via Voronoi tessellations, and clique counts in generalized random geometric graphs.



\end{abstract}

\section{Introduction}

Let $(\XX,\mathcal{F})$ be a measurable space equipped with a $\sigma$-finite measure $\Q$ and a measurable semi-metric $\d: \XX\times\XX\to [0,\infty)$.  For all $s\geq 1$ let $\P_s$ be a Poisson point process with intensity measure $s  \Q$.  When $\Q$ is
a probability measure, we let $\X_n$ be a binomial point process of $n$ points which are i.i.d. according to $\Q$. Consider the statistics
\be \label{Poissonstat}
H_s:= h_s(\P_s):=\sum_{x\in\P_s} \xi_s(x,\P_s), \quad s\geq 1,
\ee
and
\be \label{Binomialstat}
H'_n := h_n(\X_n) :=\sum_{x\in\X_n} \xi_n(x,\X_n), \quad n\in\N,
\ee
where, roughly speaking,  the scores $ \xi_s(x,\P_s)$  and  $\xi_n(x,\X_n)$ represent the local
contributions to the global statistics $H_s $ and $H_n'$,  respectively. Functionals such as $H_s$ and $H_n'$, which are in some sense locally defined, are called stabilizing functionals. The concept of stabilization and the systematic investigation of stabilizing functionals go back to the papers \cite{PY1,PY4}. In the following we are interested in quantitative central limit theorems for stabilizing functionals, whereas laws of large numbers are shown in \cite{Pe,PY4} and moderate deviations are considered in \cite{ERS}. For a survey on limit theorems in stochastic geometry with a particular focus on stabilization we refer to \cite{SchreiberSurvey}.
 Statistics $H_s$ and $H_n'$ typically describe a global property of a random geometric structure on
$\XX$ in terms of local contributions exhibiting spatial interaction and dependence. Functionals in stochastic geometry
which may be cast in the form of \eqref{Poissonstat} and \eqref{Binomialstat} include total edge length and clique counts in random
graphs, statistics of Voronoi set approximation, the $k$-face and volume functional of convex hulls of random point samples, as well as statistics of RSA packing models and spatial birth growth models.

In the following we allow that the underlying point processes $\P_s$ and $\X_n$ are marked, i.e., that an i.i.d.\ random mark is attached to each of their points.

Throughout this paper we denote by $N$ a standard Gaussian random variable and by
\begin{equation}\label{eqn:KolmogorovDistance}
d_K(Y,Z) := \sup_{t \in \R} |\Prob(Y \leq t) - \Prob(Z \leq t)|
\end{equation}
the Kolmogorov distance of two random variables $Y$ and $Z$. For a sum $S_n=\sum_{i=1}^n Y_i$ of $n$ i.i.d.\ random variables $Y_1,\hdots,Y_n$ such that $\E |Y_1|^3<\infty$ it is known from the classical Berry-Esseen theorem that
\begin{equation}\label{eqn:ClaissicalCLT}
d_K\bigg(\frac{S_n-\E S_n}{\sqrt{\Var S_n}},N\bigg) \leq \frac{C \E|Y_1-\E Y_1|^3}{\Var Y_1}\frac{1}{\sqrt{\Var S_n}}, \quad n\in\N,
\end{equation}
with $C\in(0,\infty)$ a universal constant. By considering special choices for $Y_1,\hdots,Y_n$, one can show that the rate $1/\sqrt{\Var S_n}$ in \eqref{eqn:ClaissicalCLT} is optimal. The main contribution of this paper is to show that exponentially stabilizing functionals $H_s$ and $H_n'$ satisfy bounds resembling that at \eqref{eqn:ClaissicalCLT}, with rates $1/\sqrt{\Var H_s}$ and $1/\sqrt{\Var H_n'}$, respectively. Here the scores $(\xi_s)_{s \geq 1}$ and $(\xi_n)_{n  \geq 1}$ have uniformly bounded $(4 + p)$th moments for some $p>0$, similar to the assumption   $\E |Y_1|^3<\infty$ at \eqref{eqn:ClaissicalCLT}.
In contrast to the summands of $S_n$, the summands of $H_s$ and $H_n'$ are dependent in general, but nevertheless by comparison with the classical Berry-Esseen theorem, one can expect the rates  $1/\sqrt{\Var H_s}$ and $1/\sqrt{\Var H_n'}$ to be optimal.

In stochastic geometry, it is frequently the case that $(H_s - \E H_s)/\sqrt{ \Var H_s }$ converges to the
standard normal, and likewise for $(H_n'-\E H_n')/\sqrt{\Var H_n'}$.  However up to now there has been no systematic treatment which establishes presumably optimal rates of convergence to the normal. For example, in \cite{BB} a central limit theorem for functionals of nearest neighbor graphs is derived, but no rate of convergence is given. Dependency graph methods are used in \cite{AB} to show asymptotic normality of the total edge length of the nearest neighbor graph as well as of the Voronoi and Delaunay tessellations, but lead to suboptimal rates of convergence. Anticipating stabilization methods, the authors of \cite{KL} proved asymptotic normality for the total edge length of the Euclidean minimal spanning tree, though they did not obtain a rate of convergence. In the papers \cite{BY05,Pe07,PY1} abstract central limit theorems for stabilizing functionals are derived and applied to several problems from stochastic geometry. Quantitative bounds for the normal approximation of stabilizing functionals of an underlying Poisson point process are given in \cite{BX1,PenroseRosoman,PY5,PY6,Yu}.
 These results yield rates of convergence for the Kolmogorov distance  of the order
  $1/\sqrt{\Var H_s}$ times some extraneous logarithmic factors. For stabilizing functionals of an underlying binomial point process we are unaware of analogous results. The paper \cite{Chat} uses Stein's method to provide rates of normal convergence for functionals on binomial input satisfying a type of local dependence, though these rates are in the Wasserstein distance.

Recent work  \cite{LPS} shows that the Malliavin calculus, combined with Stein's method of normal approximation, yields rates of
normal approximation for general Poisson functionals.  The rates are in the Kolmogorov distance, they are presumably optimal, and the authors
use their general results to deduce rates of normal convergence (cf. Proposition 1.4 and Theorem 6.1 of \cite{LPS}) for Poisson functionals satisfying a type of stabilization.
That paper states that `the new connection between the Stein-Malliavin approach and the theory of stabilization has a great potential for further
generalisations and applications',  though it stops short of linking these two fertile research areas.

The first main goal of this paper is to fully
develop this connection, showing that the theory of stabilization neatly dovetails with Malliavin-Stein methods, giving presumably optimal rates of
normal convergence. Malliavin-Stein rates of normal convergence, expressed in terms of moments of first and second order difference operators \cite{LPS}, seemingly
consist of unwieldy terms.  However, if $\xi_s$ is exponentially stabilizing and satisfies a moment condition, then our first main goal is to show that the Malliavin-Stein bounds remarkably simplify, showing that
\be \label{Poissonrate}
d_K\bigg( \frac{H_s-\E H_s}{\sqrt{\Var H_s}},N \bigg)  \leq  \frac{ \tilde{C}}  { \sqrt{ \Var H_s} },  \quad s\geq 1,
\ee
as explained in Corollary \ref{coroLSY}.  These rates, presumed optimal, remove extraneous logarithmic factors appearing in \cite{BX1, PenroseRosoman, PY5,PY6,Yu}.

Our second main goal is to show that \eqref{Poissonrate} holds when $H_s$ is replaced by $H'_n$, thus giving analogous rates of normal convergence
when Poisson input is replaced by binomial input. Recall that the paper \cite{LRP} (see  Theorem 5.1 there) uses Stein's method and difference operators to
establish rates of normal convergence in the Kolmogorov distance for general functionals of binomial point processes.  Though \cite{LRP} deduces
rates of normal convergence for some statistics of binomial input in geometric probability, it too stops short of systematically developing the connection
between stabilization, Stein's method, and difference operators.  Our second goal is to explicitly and fully develop
this connection. As a by-product, we show that the ostensibly unmanageable bounds in the Kolmogorov distance may be re-cast
into bounds which collapse into a single term $1/ \sqrt{\Var H_n'}$.  In other words, when $\xi_n$ has a a radius of stabilization (with respect to
binomial input $\X_n$) which decays exponentially fast, then subject to a moment condition on $\xi_n$, Corollary \ref{coroLSY} shows
\be \label{binomialrate}
d_K\bigg( \frac{H_n'-\E H_n'}{\sqrt{\Var H'_n}},N \bigg)  \leq  \frac{ \tilde{C'}}  { \sqrt{ \Var H'_n} }, \quad n\geq 9.
\ee

The main finding of this paper, culminating much research related to stabilizing score functionals and captured by the rate results  \eqref{Poissonrate} and \eqref{binomialrate}, is this: \  Statistics \eqref{Poissonstat} and \eqref{Binomialstat} enjoy presumably optimal rates of normal convergence once the scores $\xi_s$ and $\xi_n$ satisfy exponential stabilization and a moment condition.  In problems of interest, the verification of these conditions is sometimes a straightforward exercise, as seen in Section 5, the applications section. On the other hand, for statistics involving convex hulls of random point samples in a smooth compact convex set, the verification of these conditions involves a judicious choice of the underlying metric space, one which allows us to express complicated spatial dependencies in a relatively simple fashion.  This is all illustrated in Subsection \ref{conhull}, where it is shown
for both the intrinsic volumes of the convex hull and for the count of its lower dimensional faces, that the convergence rates \eqref{Poissonrate} and \eqref{binomialrate} are either the first of their kind or that they significantly improve upon existing rates of convergence in the literature, for both Poisson and binomial input in all dimensions $d\geq 2$.

Our third and final goal is to broaden the scope of existing central limit theory in such a way that:

(i) The presumably optimal rates \eqref{Poissonrate} and \eqref{binomialrate} are applicable both in the context of volume order and of surface area order
scaling of the variance of the functional. By this we mean that the variance of $H_s$ (resp. $H'_n$) is of order $s$ (resp. $n$) or $s^{1-1/d}$ (resp. $n^{1-1/d}$), after renormalising so that the score of an arbitrary point is of constant order. The notions volume order scaling and surface area order scaling come from a different (but for many problems equivalent) formulation where the intensity of the underlying point process is kept fixed and a set carrying the input is dilated instead. In this set-up the variance may be asymptotically proportional to the volume or  surface area of the carrying set. Surface order scaling of the variance typically arises when the scores are non-vanishing only for points close to a $(d-1)$-dimensional subset of $\R^d$.
As shown in Theorems \ref{maxpts} and \ref{Vor}, this generality yields improved rates of normal convergence for the number of maximal points in a random sample and for statistics arising in Voronoi set approximation, respectively.

(ii)  The methods are sufficiently general so that they bring within their purview score functions of data on spaces $(\XX, \d)$, with $\d$
an {\em arbitrary  semi-metric}. We illustrate the power of our general approach by establishing a self-contained, relatively short proof of the asymptotic normality of statistics of convex hulls of random  point samples as discussed earlier in this introduction.
Our methods also deliver rates of convergence for statistics of $k$-nearest neighbors graphs and clique counts on both
Poisson and binomial input on general metric spaces $(\XX, \d)$, as seen in Theorems \ref{NNG} and \ref{cliquecount}.

We anticipate that the generality of the methods here will lead to further non-trivial applications in the central limit theory for
functionals in stochastic geometry.

This paper is organized as follows. In Section \ref{sec:MainResults} we give abstract bounds for the normal approximation of stabilizing functionals with respect to Poisson or binomial input, which are our main findings. These are proven in Section \ref{section4}, which we prepare by recalling and rewriting some existing Malliavin-Stein bounds in Section \ref{sec:MalliavinStein}. In Section \ref{sec:Applications} we demonstrate the power of our general bounds by applying them to several problems from stochastic geometry.

\section{Main results}\label{sec:MainResults}

In this section we present our main results in detail. We first spell out assumptions on the measurable space $(\XX,\mathcal{F})$, the $\sigma$-finite measure $\Q$ and the measurable semi-metric $\d: \XX\times\XX\to[0,\infty)$. By $B(x,r)$ we denote the ball of radius $r>0$ around $x\in\XX$, i.e.\ $B(x,r):=\{y\in\XX: \d(x,y)\leq r\}$. In the standard set-up for stabilizing functionals, $\XX$ is a subset of $\R^d$ and $\Q$ has a bounded density with respect to the Lebesgue measure (see, for example, \cite{Pe07, PY5, Yu}). To handle more general $\XX$ and $\Q$, we replace this standard assumption by the following growth condition on the $\Q$-surface area of spheres: \ There are constants $\gamma,\kappa>0$ such that
\begin{equation}\label{eqn:SurfaceBall}
\limsup_{\varepsilon\to 0} \frac{\Q(B(x,r+\varepsilon))-\Q(B(x,r))}{\varepsilon}\leq \kappa\gamma r^{\gamma -1}, \quad r\geq 0, x\in\XX.
\end{equation}

Two examples for measure spaces $(\XX,\mathcal{F},\Q)$ and semi-metrics $\d$ satisfying the assumption \eqref{eqn:SurfaceBall} are the following:
\begin{itemize}
\item Example 1. Let $\XX$ be a full-dimensional subset of $\R^d$ equipped with the induced Borel-$\sigma$-field $\mathcal{F}$ and the usual Euclidean distance $\d$, assume that $\Q$ is a measure on $\XX$ with a density $g$ with respect to the Lebesgue measure, and put $\gamma:=d$. Then condition \eqref{eqn:SurfaceBall} reduces to the standard assumption that $g$ is bounded. Indeed, if $\|g\|_\infty:=\sup_{x\in\XX}|g(x)|<\infty$, then \eqref{eqn:SurfaceBall} is obviously satisfied with $\kappa:=\|g\|_\infty \kappa_d$, where $\kappa_d:= \pi^{d/2}/ \Gamma(d/2 + 1)$ is the volume of the $d$-dimensional unit ball in $\R^d$.  On the other hand, if \eqref{eqn:SurfaceBall} holds, then $\Q(B(x,r)) \leq \ka r^{d}$ as seen by Lemma \ref{Lem4.1}(a) below.
This gives an upper bound of $\ka/\kappa_d$ for $g$ since, by Lebesgue's differentiation theorem,  Lebesgue almost all points $x$ in $\R^d$ are Lebesgue points, that is to say
$$
g(x) = \lim_{r \to 0} (\kappa_d r^{d})^{-1} \int_{y \in B(x,r)} g(y) \, \dint y = \lim_{r \to 0} (\kappa_d r^{d})^{-1} \Q(B(x,r)) \leq \ka/\kappa_d.
$$

\item Example 2. Let $\XX \subset \R^d$ be a smooth $m$-dimensional subset of $\R^d$, $m \leq d$,  equipped with a semi-metric
$\d$, and a measure $\Q$ on $\XX$ with a bounded density $g$ with respect to the uniform surface  measure ${\Vol}_{m}$ on $\XX$.  We assume that the ${\Vol}_{m-1}$ measure of the sphere $\partial (B(x,r))$ is bounded by the surface area of
the Euclidean sphere $\mathbb{S}^{m-1}(0,r)$ of the same radius, that is to say
\be \label{SAbd}
{\Vol}_{m-1}( \partial B(x,r)) \leq m \ka_{m} r^{m-1}, \ x \in \XX, \ r > 0.
\ee
When $\XX$ is an $m$-dimensional affine space and $\d$ is the usual Euclidean metric on $\R^d$, \eqref{SAbd} holds with equality, naturally. However \eqref{SAbd} holds in more general
situations.  For example, by Bishop's comparison theorem (Theorem 1.2 of \cite{SYau}, along with (1.15) there),  \eqref{SAbd}
holds for Riemannian manifolds $\XX$ with non-negative Ricci curvature, with $\d$ the geodesic distance. Given the bound \eqref{SAbd}, one obtains \eqref{eqn:SurfaceBall} with $\kappa=\|g\|_{\infty} \kappa_m$ and $\gamma=m$. This example includes the case $\XX=\mathbb{S}^{m}$, the unit sphere
in $\R^{m+1}$ equipped with the geodesic distance.
\end{itemize}

In order to deal with marked point processes, let $(\MM,\mathcal{F}_\MM,\Q_\MM)$ be a probability space. In the following $\MM$ shall be the space of marks and $\Q_\MM$ the underlying probability measure of the marks. Let $\widehat{\XX}:=\XX\times \MM$, put $\widehat{\mathcal{F}}$ to be the product $\sigma$-field of $\mathcal{F}$ and $\mathcal{F}_\MM$,  and let $\widehat{\Q}$ be the product measure of $\Q$ and $\Q_{\MM}$. When $(\MM,\mathcal{F}_{\MM},\Q_{\MM})$ is a singleton endowed with a Dirac point mass, $\widehat{\XX}$ reduces to $\XX$ and the `hat' superscript can be removed in all occurrences.

Let $\mathbf{N}$ be the set of $\sigma$-finite counting measures on $\widehat{\XX}$, which can be interpreted as point configurations in $\widehat{\XX}$. Thus, we treat the elements from $\mathbf{N}$ as sets in our notation. The set $\mathbf{N}$ is equipped with the smallest $\sigma$-field $\mathcal{N}$ such that the maps $m_A: \mathbf{N}\to \N\cup\{0,\infty\}, \mathcal{M}\mapsto \mathcal{M}(A)$ are measurable for all $A\in\widehat{\mathcal{F}}$. A point process is now a random element in $\mathbf{N}$. In this paper we consider two different classes of point processes, namely Poisson and binomial point processes. For $s\geq 1$, update the notation $\P_s$ to represent a Poisson point process with intensity measure $s\widehat{\Q}$. This means that the numbers of points of $\P_s$ in disjoint sets $A_1,\hdots,A_m\in\widehat{\mathcal{F}}$, $m\in\N$, are independent and that the number of points of $\P_s$ in a set $A\in\widehat{\mathcal{F}}$ follows a Poisson distribution with mean $s\widehat{\Q}(A)$. In case  $\Q$ is a
probability measure, we denote similarly by $\X_n$ a binomial point process of $n\in\N$ points that are independently distributed according to $\widehat{\Q}$. Whenever we state a result involving the binomial point process $\X_n$, we implicitly assume that $\Q$, and hence $\widehat{\Q}$, are  probability measures.

As mentioned in the first section, we seek central limit theorems for $H_s$  and  $H_n'$ defined at \eqref{Poissonstat}  and \eqref{Binomialstat},  respectively.  We assume that the scores $(\xi_s)_{s\geq 1}$ are measurable functions from $\widehat{\XX}\times\mathbf{N}$ to $\R$.  To derive central limit theorems for $H_s$ and $H_n'$, we impose several conditions on the scores. For $s\geq 1$ a measurable map $R_s: \widehat{\XX}\times\mathbf{N}\to\R$ is called a {\em radius of stabilization}
for $\xi_s$ if for all $\hat{x}:=(x,m_x)\in\widehat{\XX}$, $\M\in\mathbf{N}$ and finite $\widehat{\A} \subset \widehat{\XX}$ with $|\widehat{\A}|\leq 7$ we have
\begin{equation}\label{eqn:RadiusOfStabilization}
\xi_s(\hat{x},(\M\cup\{\hat{x}\} \cup \widehat{\A} ) \cap \widehat{B}(x,R_s(\hat{x},\M\cup\{\hat{x}\})) )   = \xi_s(\hat{x}, \M\cup\{\hat{x}\} \cup \widehat{\A} ),
\end{equation}
where $\widehat{B}(y,r):=B(y,r)\times\MM$ for $y\in\XX$ and $r>0$.

For a given point $x\in\XX$ we denote by $M_x$ the corresponding random mark, which is distributed according to $\Q_\MM$ and is independent of everything else. Say that  $(\xi_s)_{s \geq 1}$ (resp.\ $(\xi_n)_{n\in\N}$) are  {\em exponentially stabilizing} if there are radii of stabilization $(R_s)_{s\geq 1}$ (resp.\ $(R_n)_{n\in\N}$) and constants $C_{stab},c_{stab},\alpha_{stab} \in (0,\infty)$ such that, for $x\in \XX$, $r\geq 0$ and $s\geq 1$,
\begin{equation} \label{eqn:expstabPoisson}
\Prob(R_s((x,M_x), \P_s\cup\{(x,M_x)\}) \geq r)  \leq C_{stab} \exp(-c_{stab} (s^{1/\gamma} r)^{\alpha_{stab}}),
\end{equation}
resp.\ for $x\in \XX$, $r\geq 0$ and $n\geq 9$,
\begin{equation} \label{eqn:expstabBinomial}
\Prob(R_n((x,M_x), \X_{n-8}\cup\{(x,M_x)\}) \geq r)  \leq C_{stab} \exp(-c_{stab} (n^{1/\gamma} r)^{\alpha_{stab}}),
\end{equation}
where $\gamma$ is the constant from \eqref{eqn:SurfaceBall}.

For a finite set $\mathcal{A}\subset\XX$ we denote by $(\mathcal{A},M_\mathcal{A})$ the random set obtained by equipping each point of $\mathcal{A}$ with a random mark distributed according to $\Q_\MM$ and independent of everything else. Given $p \in [0, \infty)$, say that $(\xi_s)_{s\geq 1}$ or $(\xi_n)_{n\in\N}$ satisfy a $(4+p)$th moment condition if there is a constant $C_p \in (0, \infty)$ such that
for all $\A \subset \XX$ with $|\A|\leq 7$,
\begin{equation} \label{eqn:momPoisson}
\sup_{s \in [1, \infty)}  \sup_{x \in \XX}  \E |\xi_s((x,M_x), \P_s\cup\{(x,M_x)\}  \cup  (\A,M_\A) ) |^{4+p} \leq C_{p}
\end{equation}
or
\begin{equation} \label{eqn:momBinomial}
\sup_{n\in\N, n\geq 9}  \sup_{x \in \XX}  \E |\xi_n((x,M_x), \X_{n-8}\cup\{(x,M_x)\}  \cup  (\A,M_\A) ) |^{4+p} \leq C_{p}.
\end{equation}

Let $K$ be a measurable subset of $\XX$. By $\d(z,K):=\inf_{y\in K} \d(z,y)$ we denote the distance between a point $z\in\XX$ and $K$. Moreover, we use the abbreviation $\d_s(\cdot,\cdot):=s^{1/\gamma} \d(\cdot,\cdot)$, $s\geq 1$. We introduce another notion relevant for functionals whose variances exhibit surface area order scaling. Say that $(\xi_s)_{s \geq 1}$, resp.\ $(\xi_n)_{n\in\N}$, {\em decay exponentially fast with the distance to $K$} if there are constants $C_{K},c_{K},\alpha_{K}\in(0,\infty)$ such that for all $\A \subset \XX$ with $|\A|\leq 7$ we have
\begin{equation} \label{eqn:expfastPoisson}
\Prob( \xi_s((x,M_x), \P_s\cup\{(x,M_x)\}  \cup (\A,M_\A)) \neq 0) \leq C_{K}\exp( -c_{K} \d_s(x, K)^{\alpha_{K}})
\end{equation}
for $x\in\XX$ and $s\geq 1$ resp.\
\begin{equation} \label{eqn:expfastBinomial}
\Prob( \xi_n((x,M_x), \X_{n-8} \cup\{(x,M_x)\} \cup (\A,M_\A)) \neq 0) \leq C_{K}\exp( - c_{K} \d_n(x, K)^{\alpha_{K}})
\end{equation}
for $x\in\XX$ and $n\geq 9$.
For functionals whose variances have volume order we will make the choice $K = \XX$, in which case \eqref{eqn:expfastPoisson} and \eqref{eqn:expfastBinomial} are obviously satisfied with $C_K=1$ and arbitrary $c_K,\alpha_K\in(0,\infty)$. Later we will have  that $\XX$ is $\R^d$ or a compact convex subset of $\R^d$ such as the unit cube and that $K$ is a $(d-1)$-dimensional subset of $\R^d$. This situation arises, for example, in statistics of convex hulls of random samples and Voronoi set approximation. Moreover, problems with surface order scaling of the variance are typically of this form.

The following general theorem provides rates of normal convergence for $H_s$ and $H_n'$ in terms of the Kolmogorov distance defined at \eqref{eqn:KolmogorovDistance}. This theorem is a consequence of general theorems from \cite{LPS} and \cite{LRP} giving Malliavin-Stein bounds for functionals of Poisson and binomial point processes (see Theorems \ref{thm:LPS} and \ref{thm:general-binomial} below). Let $\alpha:=\min\{\alpha_{stab},\alpha_K\}$ and
\begin{equation}\label{eqn:DefinitionIKn}
I_{K,s}:= s \int_{\XX} \exp\bigg(-\frac{  \min\{c_{stab},c_K\}p\d_s(x,K)^{\alpha}}{36\cdot 4^{\alpha+1}}\bigg) \, \Q(\dint x), \quad s\geq 1.
\end{equation}
Throughout this paper $N$ always denotes a standard Gaussian random variable. The proofs of the following results are postponed to Section \ref{section4}.

\begin{theo}\label{theoLSY}
\begin{itemize}
\item[(a)] Assume that the score functions $(\xi_s)_{s\geq 1}$ are exponentially stabilizing \eqref{eqn:expstabPoisson}, satisfy the moment condition \eqref{eqn:momPoisson} for some $p \in (0, 1]$, and decay exponentially fast with the distance to a measurable set $K\subset\XX$, as
at \eqref{eqn:expfastPoisson}.
Then there is a constant $\tilde{C}\in(0,\infty)$ only depending on the constants in \eqref{eqn:SurfaceBall}, \eqref{eqn:expstabPoisson}, \eqref{eqn:momPoisson} and \eqref{eqn:expfastPoisson} such that
\begin{equation} \label{eqn:KolmogorovPoisson}
d_K\bigg( \frac{H_s-\E H_s}{\sqrt{\Var H_s}},N \bigg)  \leq  \tilde{C}  \bigg(\frac{\sqrt{ I_{K,s}}}{\Var H_s}+\frac{I_{K,s}}{(\Var H_s)^{3/2}}+\frac{I_{K,s}^{5/4}+I_{K,s}^{3/2}}{(\Var H_s)^{2}}\bigg), \quad s\geq 1.
\end{equation}
\item[(b)] Assume that the score functions $(\xi_n)_{n\in\N}$ are exponentially stabilizing \eqref{eqn:expstabBinomial}, satisfy the moment condition \eqref{eqn:momBinomial} for some $p \in (0,1]$, and decay exponentially fast with the distance to a measurable set $K\subset\XX$, as
at \eqref{eqn:expfastBinomial}. Let $(I_{K,n})_{n\in\N}$ be as in \eqref{eqn:DefinitionIKn}. Then there is a constant $\tilde{C}\in(0,\infty)$ only depending on the constants in \eqref{eqn:SurfaceBall}, \eqref{eqn:expstabBinomial}, \eqref{eqn:momBinomial} and \eqref{eqn:expfastBinomial} such that
\begin{equation} \label{eqn:KolmogorovBinomial}
d_K\bigg( \frac{H'_n-\E H'_n}{\sqrt{\Var H'_n}},N \bigg)  \leq \tilde{C} \bigg(\frac{\sqrt{ I_{K,n} }} {\Var H'_n} + \frac{I_{K,n}}{(\Var H'_n)^{3/2}}+\frac{I_{K,n}+ I_{K,n}^{3/2}}{(\Var H'_n)^2}\bigg), \quad n\geq 9.
\end{equation}
\end{itemize}
\end{theo}

Notice that if $K=\XX$, we have
\begin{equation}\label{eqn:IXlinear}
I_{\XX,s}=s\Q(\XX), \quad s\geq 1, \quad \text{ and } \quad I_{\XX,n}=n\Q(\XX), \quad n\in\N.
\end{equation}
Assuming growth bounds  on $I_{K,s}/ \Var H_s$ and $I_{K,n}/ \Var H'_n$, the rates \eqref{eqn:KolmogorovPoisson} and \eqref{eqn:KolmogorovBinomial}
nicely simplify into presumably optimal rates, ready for off-the-shelf use in applications.

\begin{coro}\label{coroLSY}
\begin{itemize}
\item [(a)] Let the conditions of Theorem \ref{theoLSY}(a) prevail. Assume further that there is a $C\in(0,\infty)$ such that
$\sup_{s \geq 1} I_{K,s}/ \Var H_s \leq C.$  Then there is a $\tilde{C'} \in (0,\infty)$ only depending on $C$ and the constants in \eqref{eqn:SurfaceBall}, \eqref{eqn:expstabPoisson}, \eqref{eqn:momPoisson} and \eqref{eqn:expfastPoisson} such that
\begin{equation} \label{eqn:KolmogorovPoissoncoro}
d_K\bigg( \frac{H_s-\E H_s}{\sqrt{\Var H_s}},N \bigg)  \leq  \frac{ \tilde{C'}}  { \sqrt{ \Var H_s} },  \quad s\geq 1.
\end{equation}
\item [(b)] Let the conditions of Theorem \ref{theoLSY}(b) prevail. If there is a $C\in(0,\infty)$ such that $\sup_{n \geq 1} I_{K,n}/ \Var H'_n \leq C$, then there is a $\tilde{C'} \in (0,\infty)$ only depending on $C$ and the constants in \eqref{eqn:SurfaceBall}, \eqref{eqn:expstabBinomial}, \eqref{eqn:momBinomial} and \eqref{eqn:expfastBinomial} such that
\begin{equation} \label{eqn:KolmogorovBinomialcoro}
d_K\bigg( \frac{H'_n-\E H'_n}{\sqrt{\Var H'_n}},N \bigg)  \leq  \frac{ \tilde{C'}}  { \sqrt{ \Var H'_n} }, \quad n\geq 9.
\end{equation}
\end{itemize}
\end{coro}
This corollary is applied in the context of the convex hull of a random sample of points in a smooth convex set in Subsection \ref{conhull}. In this case, the  variance is of order  $s^{\frac{d-1}{d+1}}$ ($n^{\frac{d-1}{d+1}}$ in the binomial setting), and we obtain rates of normal convergence  of  order $(\Var H_{s})^{-1/2}=\Theta (s^{-(d-1)/(2(d+1)))}$ (resp. $(\Var H_{n}')^{-1/2}=\Theta (n^{-(d-1)/(2(d+1)))}$),  which   improves upon rates obtained via other methods.

In the setting $\XX \subset \R^d$, our results admit further simplification, which goes as follows.
For $K \subset \XX \subset \R^d$ and $r \in (0, \infty)$, let
$K_r := \{ y \in \R^d: \ \ \d(y, K) \leq r \}$ denote the $r$-parallel set of $K$. Recall that the $(d-1)$-dimensional upper Minkowski content
of $K$ is given by
\begin{equation}\label{eqn:UpperMinkowski}
\overline{\cal M}^{d-1}(K) := \limsup_{r \to 0} \frac{ \Vol_d(K_r) } {2 r}.
\end{equation}
If $K$ is a closed $(d-1)$-rectifiable set in $\R^d$ (i.e., the Lipschitz image of a bounded set in $\R^{d-1}$), then $\overline{\cal M}^{d-1}(K)$
exists and coincides with a scalar multiple of $\H^{d-1}(K)$, the $(d-1)$-dimensional Hausdorff measure  of $K$. Given an unbounded set $I\subset (0, \infty)$ and two families of real numbers $(a_{i})_{i\in I},(b_{i})_{i\in I}$, we use the Landau notation $a_{i}=O(b_{i})$
to indicate that $\limsup_{i\in I, i\to\infty}|a_{i}|/|b_{i}|<\infty$. If $b_i=O(a_i)$ we write $a_i=\Omega(b_i)$, whereas if  $a_i=O(b_i)$ and $b_i=O(a_i)$ we write $a_i=\Theta(b_i)$.

\begin{theo}\label{2coroLSY}
Let $\XX\subset\R^d$ be full-dimensional, let $\Q$ have a bounded density with respect to Lebesgue measure and let the conditions of Theorem \ref{theoLSY} prevail with $\gamma:=d$. \vskip.1cm
\noindent (a)  Let $K$ be a full-dimensional  compact subset of $\XX$  with $\overline{\cal M}^{d-1}( \partial K) < \infty$. If $\Var H_s = \Omega(s)$, resp.\ $\Var H'_n = \Omega(n)$, then there is a constant $c \in (0, \infty)$ such that
\be \label{arates}
d_K\bigg( \frac{H_s-\E H_s}{ \sqrt{\Var H_s} }, N \bigg)  \leq  \frac{c} { \sqrt{s}}, \quad s \geq 1, \ \text{ resp. } \
d_K\bigg( \frac{H'_n-\E H'_n} {\sqrt{\Var H'_n}} ,N \bigg)  \leq  \frac{c} { \sqrt{n}},  \quad n\geq 9.
\ee
\noindent (b) Let $K$ be a $(d-1)$-dimensional compact subset of $\XX$ with $\overline{\cal M}^{d-1}(K) < \infty$.
If $\Var H_s = \Omega(s^{(d-1)/d})$, resp.\ $\Var H'_n =\Omega(n^{(d-1)/d})$,
then there is a constant $c \in (0, \infty)$ such that
\be \label{aratessurface}
d_K\bigg( \frac{H_s-\E H_s}{\sqrt{\Var H_s}},N \bigg)  \leq  \frac{c}{s^{ \frac{1}{2} - \frac{1}{2d}}   }, \quad s \geq 1, \ \text{ resp.\ }
\ d_K\bigg( \frac{H'_n-\E H'_n}{\sqrt{\Var H'_n}}, N \bigg)  \leq   \frac{c} {n^{ \frac{1}{2} - \frac{1}{2d}}   } , \quad n\geq 9.
\ee
\end{theo}

\noindent{\em Remarks.} (i) {\em Comparing \eqref{arates} with existing results.}  The results at \eqref{arates} are applicable in the setting
of volume order scaling of the variances, i.e., when the variances of $H_s$ and $H_n'$ exhibit scaling proportional to $s$ and $n$.
The rate for Poisson input in \eqref{arates} significantly improves upon the rate given by Theorem 2.1 of \cite{PY5} (see also Lemma 4.4 of \cite{Pe07}),  Corollary 3.1 of \cite{BX1},  and Theorem 2.3 in \cite{PenroseRosoman},
which all contain extraneous logarithmic factors and which rely on dependency graph methods. The rate in \eqref{arates} for binomial input is new, as up to now there are no explicit general rates of normal convergence for sums of stabilizing score functions $\xi_n$ of binomial input.

\vskip.3cm
\noindent (ii) {\em Comparing \eqref{aratessurface} with existing results.}
The rates at  \eqref{aratessurface} are relevant for statistics with surface area rescaling of the variances, i.e., when the variance of $H_s$ (resp. $H_n'$) exhibits scaling proportional to $s^{1-1/d}$ (resp. $n^{1-1/d}$). These rates  both improve and extend upon the rates given in the main result (Theorem 1.3) in \cite{Yu}.  First, in the case of Poisson input,
the rates remove the logarithmic factors present in Theorem 1.3 of \cite{Yu}. Second,  we obtain rates of normal convergence for binomial input, whereas \cite{Yu} does not treat this situation.
\vskip.3cm
\noindent (iii) {\em Extensions to random measures}.
Up to a constant factor, the rates of normal convergence in Theorem \ref{theoLSY}, Corollary \ref{coroLSY}, and Theorem \ref{2coroLSY} hold for the non-linear statistics
$H_{s}(f)= \sum_{x \in \P_s} f(x)\xi_s(x, \P_s)$ and $H'_{n}(f)= \sum_{x \in \X_n}f(x) \xi_n(x, \X_n)$, obtained by integrating the random measures
$\sum_{x \in \P_s} \xi_s(x, \P_s) \delta_x$ and $\sum_{x \in \X_n} \xi_n(x, \X_n) \delta_x$ with a bounded measurable test function $f$ on $\XX$.
For example, if $K = \XX$, $ \Var( H_{s}(f)) = \Omega(s)$, and $\Var (H'_{n}(f)) = \Omega(n)$, then there is a constant $c\in(0,\infty)$ such that
\be \label{ratesfPoisson}
d_K\bigg( \frac{H_{s}(f)-\E H_{s}(f) }{ \sqrt{\Var H_{s}(f)} }, N \bigg)  \leq  \frac{c} { \sqrt{s}}, \quad s \geq 1,
\ee
and
\be \label{ratesfBinomial}
d_K\bigg( \frac{H_{n}'(f) -\E H_{n}'(f)  }{ \sqrt{\Var H_{n}'(f) } }, N \bigg)  \leq  \frac{c} { \sqrt{n}}, \quad n \geq 9.
\ee
The rate \eqref{ratesfPoisson} improves upon the main result (Theorem 2.1) of \cite{PY5} whereas the rate \eqref{ratesfBinomial} is new.
\vskip.3cm
\noindent (iv) {\em Extensions to the Wasserstein distance}.
All quantitative bounds presented in this section also hold for the Wasserstein distance (see also the discussion at the end of Section \ref{sec:MalliavinStein}). The Wasserstein distance between random variables $Y$ and $Z$ with $\E|Y|,\E|Z|<\infty$ is given by
\begin{equation}\label{eqn:DefinitiondW}
d_W(Y,Z):=\sup_{h\in\operatorname{Lip}(1)} |\E h(Y)-\E h(Z)|,
\end{equation}
where $\operatorname{Lip}(1)$ stands for the set of all functions $h: \R\to\R$ whose Lipschitz constant is at most one. Since we believe that the Kolmogorov distance $d_K$  is  more prominent than the Wasserstein distance, the applications in Section \ref{sec:Applications} are formulated only for $d_K$.
\vskip.3cm
\noindent (v) {\em Subsets without influence}. Assume that there is a measurable set $\tilde{\XX}\subset\XX$ such that the scores satisfy
$$
\xi_s(x,\M) =  \mathbf{1}_{\{x\in\tilde{\XX}\}} \xi_s(x,\M\cap\tilde{\XX}), \quad \M\in \mathbf{N}, x\in\M, s\geq 1,
$$
where $\M\cap\tilde{\XX}$ stands for the restriction of the point configuration $\M$ to $\tilde{X}$. In other words, the sum of scores $\sum_{x\in\M} \xi_s(x,\M)$ only depends on the points of $\M$ which belong to $\tilde{\XX}$. In this case our previous results are still valid if the assumptions \eqref{eqn:SurfaceBall}-\eqref{eqn:expfastBinomial} hold for all $x\in\tilde{\XX}$.
\vskip.3cm
\noindent (vi) {\em Null sets}. In our assumptions \eqref{eqn:SurfaceBall}-\eqref{eqn:expfastBinomial} we require, for simplicity, that some inequalities are satisfied for all $x\in\XX$. In case that these only hold for $\Q$-a.e.\ $x\in\XX$, our results are still true. This also applies to comment (v).

\section{Malliavin-Stein bounds}\label{sec:MalliavinStein}

For any measurable $f: \mathbf{N}\to\R$ and $\mathcal{M}\in\mathbf{N}$ we define
$$
D_{\hat{x}}f(\mathcal{M})=f(\mathcal{M}\cup\{\hat{x}\})-f(\mathcal{M}), \quad \hat{x}\in\widehat{\XX},
$$
and
$$
D^2_{\hat{x}_1,\hat{x}_2}f(\mathcal{M})  =f(\mathcal{M}\cup\{\hat{x}_1,\hat{x}_2\})-f(\mathcal{M}\cup\{\hat{x}_1\}) -f(\mathcal{M}\cup\{\hat{x}_2\})+f(\mathcal{M}), \quad \hat{x}_1,\hat{x}_2\in\widehat{\XX}.
$$
Our key tool for the proof of the bound \eqref{eqn:KolmogorovPoisson} is the following marked version of a result from \cite{LPS} (see Proposition 1.4 and Theorem 6.1 in \cite{LPS}) for square integrable Poisson functionals.

\begin{theo}\label{thm:LPS}
Let $s>0$ and let $f:\mathbf{N}\to\R$ be measurable with $\E f(\P_s)^2<\infty$. Assume there are constants $c,p\in(0,\infty)$ such that
\begin{equation} \label{eqn:ConditionLPS}
\E|D_{(x,M_x)}f(\P_s\cup\{(\A,M_\A)\})|^{4+p} \leq c, \quad \Q\text{-a.e. } x\in\XX, \A\subset\XX, |\A|\leq 1.
\end{equation}
Let $F :=f(\P_s)$. Then there is a constant $C:= C(c,p) \in (0,\infty)$ such that
\begin{align}
\label{eq:bound-Po}
  d_K\bigg(\frac{F-\E F}{\sqrt{\Var F}},N \bigg)
 \leq C ( S_{1}  +S_{2} + S_{3} ),
\end{align}
with
\begin{align*}
\Gamma_s:=&s\int_{\XX} \Prob(D_{(x,M_x)}f(\P_s)\neq 0)^{\frac{p}{8+2p}} \, \Q(\dint x),\\
\psi _{s}(x_{1},x_{2}):=& \Prob(D^2_{(x_1,M_{x_1}),(x_2,M_{x_2})}f(\P_s)\neq 0)^{\frac{p}{16+4p}},\\
S_{1}:=& \frac{s}{\Var F} \sqrt{ \int_{\XX^2}\psi _{s}(x_{1},x_{2})^{2} \, \Q^2(\dint(x_1,x_2)) },\\
S_{2}:=& \frac{s^{3/2} }{\Var F} \sqrt{ \int_{\XX} \bigg( \int_{\XX}\psi_{s} (x_{1},x_{2}) \, \Q(\dint x_2)\bigg)^2  \, \Q(\dint x_1)},\\
S_{3}:=& \frac{\sqrt{\Gamma_s}}{\Var F}+\frac{2 \Gamma_s}{(\Var F)^{3/2}}+\frac{ \Gamma_s^{5/4}+2 \Gamma_s^{3/2}}{(\Var F)^2}.
\end{align*}
\end{theo}
\begin{proof}
In case that there are no marks, this is Theorem 6.1 in \cite{LPS}. The marked version can be obtained in the following way: In Theorem 1.2 in \cite{LPS} one can use the product form of $\widehat{\Q}$ and H\"older's inequality to bring the marks under the expectations. Evaluating this new bound along the lines of the proof of Theorem 6.1 in \cite{LPS} yields \eqref{eq:bound-Po}.
\end{proof}

For the  case of binomial input, we do not have the same ready-made  bounds at our disposal. We fill this lacuna with the following analogous bound, bringing \cite{LRP} and \cite{LPS} into a satisfying alignment.

\begin{theo} \label{thm:general-binomial}
Let $n\geq 3$ and let $f:\mathbf{N}\to\R$ be measurable with $\E f(\X_n)^2<\infty$.   Assume that there are constants $c,p\in(0,\infty)$ such that
  \begin{equation} \label{eqn:Conditiongeneral-binomial}
\E|D_{(x,M_x)} f(\X _{n-1- | \A | }\cup \{(\A,M_\A)\})|^{4+p } \leq c, \quad \Q\text{-a.e. }x\in \XX, \A\subset \XX,| \A | \leq 2.
\end{equation}
Let $F:=f(\X_n)$. Then there is a constant $C:= C(c,p) \in(0,\infty)$  such that
\begin{align}
\label{eq:bound-Bn}
d_{K}\left(
\frac{F-\E F}{\sqrt{\Var F}},N\right)&\leq
C(S_{1}'+S_{2}'+S_{3}'),
\end{align}
with
\begin{align*}
\Gamma'_n:= & n\int_{\XX} \Prob(D_{(x,M_x)}f(\X  _{n-1})\neq 0)^{\frac{p}{8+2p }} \, \Q (\dint x), \\
\psi'_n(x,x'):= & \sup_{\A\subset \XX: | \A | \leq 1} \Prob(D^{2}_{(x,M_x),(x',M_{x'})}f(\X  _{n-2- | \A | }\cup (\A,M_\A))\neq 0)^{\frac{p}{8+2p} },  \\
S_{1}':=&\frac{n}{\Var F}
  \sqrt{\int_{\XX^{2}}\psi'_n(x,x') \, \Q^2(\dint (x,x'))},\\
S_{2}':= &\frac{n^{3/2}}{\Var F}\sqrt{\int_{\XX}\left(
\int_{\XX}\psi'_n(x,x') \, \Q(\dint x')
\right)^{2} \Q(\dint x)},\\
S_{3}':= &\frac{\sqrt{\Gamma'_n}}{\Var F}+\frac{ \Gamma'_n}{\sqrt{\Var F}^{3}}+\frac{ \sqrt{\Gamma '_n}^{ 3}+ \Gamma _n'}{(\Var F)^2} .
\end{align*}
\end{theo}

Before proving Theorem \ref{thm:general-binomial} we require two auxiliary results, the first of which involves some additional notation. For a measurable $f:\mathbf{N}\to\R$ extend the notation $f(x_{1},\dots ,x_{q}):=f(\{x_{1},\dots ,x_{q}\})$ for $x_1,\ldots,x_q\in \XX$.

For a fixed $n\geq 1$ let $X:=(X_{1},\dots ,X_{n})$, where $X_1,\hdots,X_n$ are independent random elements in $\widehat{\XX}$ distributed according to $\widehat{\Q}$.
  Let $X',\widetilde{X}$ be independent copies of $X$.  {We write} {  $U\stackrel{a.s.}{=}V$ if two variables $U$ and $V$ satisfy $\Prob(U=V)=1$}. In the vocabulary of \cite{LRP}, a random vector $Y:=(Y_{1},\dots ,Y_{n})$ is a recombination of $\{X,X',\widetilde X\}$ if for each $1\leq i\leq n$, either {  $Y_{i}\stackrel{a.s.}{=}X_{i},Y_{i}\stackrel{a.s.}{=}X'_{i}$ or $Y_{i}\stackrel{a.s.}{=}\widetilde X_{i}$}.
For a vector $x=(x_{1},\dots ,x_{p})\in \widehat{\XX}^p$, and indices $I:=\{i_{1},\dots ,i_{q}\}\subset [p]:= \{1,2,...,p\}$, define $x^{i_{1},\dots ,i_{q}}: =(x_{j},j\notin I)$, the vector $x$ with the components indexed by $I$ removed. For $i,j \in [n]$, introduce the index derivatives
 \begin{align*}
\D_{i}f(X)&:=f(X)-f(X^i) \\
\D_{i,j}^2f(X)&:=f(X)-f(X^i)-f(X^j)+f(X^{i,j})=\D^2_{j,i}f(X).
\end{align*}
We note that the derivatives $D$ and $\D$ obey the relation $ \D_{i}f(X) = D_{X_i}f(\X_n^{i})$.

We introduce, for $n$-dimensional random vectors $Y,Y'$ and $Z$,
\begin{align*}
\gamma _{Y,Z}(f)&:=\E \left[ \mathbf{1}_{\{\D^2_{1,2}f(Y)\neq 0\}}\D_{2}f(Z)^{4} \right] \allowdisplaybreaks\\
\gamma '_{Y,Y',Z}(f)&:=\E\left[ \mathbf{1}_{\{\D^2_{1,2}f(Y)\neq 0,\;\D^2_{1,3}f(Y')\neq 0\}} \D_{2}f(Z)^{4}\right] \allowdisplaybreaks\\
B_{n}(f)&:=\sup\{\gamma _{Y,Z}(f);\;{Y,Z\text{ recombinations of }\{X,X',\widetilde X\}}\} \allowdisplaybreaks\\
B'_{n}(f)&:=\sup\{\gamma'_{Y,Y',Z}(f);\;{Y,Y',Z\text{ recombinations of }\{X,X',\widetilde X\}}\}.
\end{align*}
Theorem 5.1  of \cite{LRP}, simplified by \cite[Remark 5.2]{LRP} and \cite[Proposition 5.3]{LRP}, gives the following:

\begin{theo}\label{thm:LRP}
Let $n\geq 2$,  $f: \mathbf{N}\to\R$ measurable with $\E f(\X_n)^2<\infty$, and  $F:=f(\X_n)$.
Then there is a constant  $c_{0} \in (0, \infty)$, depending neither on $n$ nor $f$, such that
\begin{align}
\label{eq:LRP}
d_{K}\bigg(\frac{F-\E F}{\sqrt{\Var F}},N\bigg)& \leq
c_{0}
\left[\frac{\sqrt{n}}{\Var F} \left( \sqrt{nB_{n}(f)  } +\sqrt{n^{2}B'_{n}(f)}+ \sqrt{\E \D_{1}f(X)^{4}}\right) \right.\\
\notag &\left. + \sup_{Y}{ \frac{n}{(\Var F)^2}}  {\E | (f(X)-\E F) (\D_{1}f(Y))^3  |}+ \frac{n}{ (\Var F)^{\frac{3}{2}}}\E  |   \D_{1}f(X)|^{3} \right] ,
 \end{align}where the $\sup_{Y}$ runs over recombinations $Y$ of $\{X,X',\tilde X\}$.
\end{theo}

To control the fourth centered moment of $F:=f(\X_n)$, we use the following bound. For a similar bound for Poisson functionals we refer to \cite[Lemma 4.2]{LPS}.

\begin{lemm}\label{lem:BoundFourthMoment}
For a measurable $f: \mathbf{N}\to\R$, $n\in\N$  and $F:=f(\X_n)$ assume that $\Var F=1$. Then
$$
 \E (F-\E F)^4  \leq 9\max\bigg\{\bigg(32 n \int_{\XX} \sqrt{\E (D_{(x,M_x)}f(\X_{n-1}))^4} \, {\Q}(\dint x) \bigg)^2,4n\E (\D_1f(\X_n))^4+1\bigg\}.
$$
\end{lemm}

\begin{proof}
The Efron-Stein inequality implies that for  measurable $g: \mathbf{N}\to\R$ and $n\in\N$ such that $\E g(\X_n)^2 <\infty$,
$$
\Var g(\X_n) \leq 2n \E (\D_1g(\X_n))^2.
$$
Using $\Var F=1$ and the Efron-Stein bound in this order gives
$$
\E (F-\E F)^4 = \Var \big((f(\X_n)-\E F)^2\big) + 1 \leq 2n \E \big(\D_1((f(\X_n)- \E F)^2)\big)^2 + 1.
$$
Combining the identity
\begin{align*}
\D_1(g(\mathcal{X}_n)^2) & = g(\mathcal{X}_n)^2 - g(\mathcal{X}_n^1)^2  = (g(\mathcal{X}_n^1)+\D_1g(\mathcal{X}_n))^2 - g(\mathcal{X}_n^1)^2\\
& =2 g(\mathcal{X}_n^1) \D_1g(\mathcal{X}_n) + (\D_1g(\mathcal{X}_n))^2
\end{align*}
with Jensen's inequality, we obtain
\begin{align*}
\E (F-\E F)^4 & \leq 2n \E\left[
 (2 \D_1f(\X_n)(f(\X^1_n)- \E F)+ (\D_1f(\X_n))^2\big)^2
\right]+1 \\
& \leq 4n \E \left[4(\D_1f(\X_n))^2(f(\X_n^1)-\E F)^2+(\D_1f(\X_n))^4\right]  +1.
\end{align*}
H\"older's inequality and a combination of the triangle inequality and Jensen's inequality imply that
\begin{align*}
& \E (\D_1f(\X_n))^2(f(\X_n^1)-\E F)^2\\
 & \leq  \int_{\widehat{\XX}} \sqrt{\E(f(\X_n^1\cup\{y\})-f(\X_n^1))^4} \, \widehat{\Q}(\dint y)\sqrt{\E (f(\X_n^1)-\E F)^4}\\
& \leq  \int_{\XX} \sqrt{\E(D_{(x,M_x)}f(\X_{n-1}))^4} \, \Q(\dint x) \
2(\sqrt{\E (f(\X_n)-\E F)^4}+\sqrt{\E (\D_1f(\X_n))^4}).
\end{align*}
Combining the above estimates we arrive at
\begin{align*}
 \E (F-\E F)^4
& \leq 32n \int_{\XX} \sqrt{\E(D_{(x,M_x)}f(\X_{n-1}))^4} \, \Q(\dint x)
(\sqrt{\E (F-\E F)^4}+\sqrt{\E (\D_1f(\X_n))^4})\\
& \quad +4n \mathbb{E}(\D_1f(\X_n))^4  +1,
\end{align*}
which implies the asserted inequality.
\end{proof}

Given  Lemma \ref{lem:BoundFourthMoment}, we deduce Theorem \ref{thm:general-binomial} from Theorem \ref{thm:LRP}  as follows.

\vskip.3cm

\begin{proof}[Proof of Theorem \ref{thm:general-binomial}]
It suffices to show that each of the five terms in \eqref{eq:LRP} is bounded by a scalar multiple of $S_{1}'$, $S_{2}'$, or  $S_{3}'$. We first show that the terms in \eqref{eq:LRP} involving $B_{n}(f)$ and $B_{n}'(f)$ are bounded resp.\ by scalar multiples of $S_{1}'$ and $S_{2}'$.
Let us estimate first $B_{n}(f)$. By $\widehat{\Q}^{Y_1,Y_2,Z_1,Z_2}$ we denote the joint probability measure of $Y_1,Y_2,Z_1,Z_2$ and by $\Q^{Y_1,Y_2,Z_1,Z_2}$ the joint probability measure of $Y_1,Y_2,Z_1,Z_2$ without marks. By H\"older's inequality, the fact that $\widehat{\Q}^{Y_1,Y_2,Z_1,Z_2}$ factorizes into $\Q^{Y_1,Y_2,Z_1,Z_2}$ and a part controlling the marks, the independence of $Y_1,Y_2$, and \eqref{eqn:Conditiongeneral-binomial}, we obtain that
\begin{align*}
\gamma_{Y,Z}(f) & = \E[\mathbf{1}_{\{\D^2_{1,2}f(Y)\neq 0\}} (\D_2f(Z))^4 ]\\
& = \int_{\widehat{\XX}^4} \E[\mathbf{1}_{\{D^2_{\hat{y}_1,\hat{y}_2}f(Y^{1,2})\neq 0\}} (D_{\hat{z}_2}f(Z^{1,2}\cup\{\hat{z}_1\}))^4 ] \, \widehat{\Q}^{Y_1,Y_2,Z_1,Z_2}(\dint(\hat{y}_1,\hat{y}_2,\hat{z}_1,\hat{z}_2)) \\
& \leq \int_{\widehat{\XX}^4} \Prob(D^2_{\hat{y}_1,\hat{y}_2}f(Y^{1,2})\neq 0)^{\frac{p}{4+p}} \, \E[|D_{\hat{z}_2}f(Z^{1,2}\cup\{\hat{z}_1\}))|^{4+p} ]^{\frac{4}{4+p}} \\
& \hskip 5cm \widehat{\Q}^{Y_1,Y_2,Z_1,Z_2}(\dint(\hat{y}_1,\hat{y}_2,\hat{z}_1,\hat{z}_2)) \allowdisplaybreaks\\
& \leq \int_{\XX^4} \Prob(D^2_{(y_1,M_{y_1}),(y_2,M_{y_2})}f(Y^{1,2})\neq 0)^{\frac{p}{4+p}} \, \E[|D_{(z_2,M_{z_2})}f(Z^{1,2}\cup\{(z_1,M_{z_1})\}))|^{4+p} ]^{\frac{4}{4+p}}\\
& \hskip 5cm \Q^{Y_1,Y_2,Z_1,Z_2}(\dint(y_1,y_2,z_1,z_2))\\
& \leq c^{\frac{4}{4+p}} \int_{\XX^2} \Prob(D^2_{(y_1,M_{y_1}),(y_2,M_{y_2})}f(\X_{n-2})\neq 0)^{\frac{p}{4+p}} \, \Q^{2}(\dint(y_1,y_2)).
\end{align*}
This implies that
$$
\gamma_{Y,Z}(f) \leq c^{\frac{4}{4+p}}\int_{\XX^2} \psi_{n}'(y_{1},y_{2}) \, \Q^2(\dint (y_1,y_2)),
$$
which gives the desired bound
$$
\frac{\sqrt{n}}{\Var F}\sqrt{nB_{n}(f)} \leq c^{\frac{2}{4+p}}\frac{n}{\Var F}\sqrt{\int_{\XX^{2}}\psi_{n}'(x,x') \, \Q^2(\dint(x,x'))} \leq C(c,p) S'_{1}.
$$

To estimate $B_n'(f)$, let $\widehat{\Q}^{(Y_1,\hdots,Z_3)}$ be the joint probability measure of
$$
(Y_1,\hdots,Y_3,Y_1',\hdots,Y_3',Z_1,\hdots,Z_3)
$$
and let $\Q^{(Y_1,\hdots,Z_3)}$ be the corresponding probability measure without marks. By similar arguments as above, we obtain that
\begin{align*}
\gamma'_{Y,Y',Z}(f) & = \E[\mathbf{1}_{\{\D^2_{1,2}f(Y)\neq 0,\D^2_{1,3}f(Y')\neq 0\}} (\D_2f(Z))^4 ]\\
& = \int_{\widehat{\XX}^9} \E[\mathbf{1}_{\{D^2_{\hat{y}_1,\hat{y}_2}f(Y^{1,2,3}\cup\{\hat{y}_3\})\neq 0\}} \mathbf{1}_{\{D^2_{\hat{y}'_1,\hat{y}'_3}f(Y'^{1,2,3}\cup\{\hat{y}'_2\})\neq 0\}} (D_{\hat{z}_2}f(Z^{1,2,3}\cup\{\hat{z}_1,\hat{z}_3\}))^4 ] \\
& \hskip 2cm \widehat{\Q}^{Y_1,\hdots,Z_3}(\dint(\hat{y}_1,\hdots,\hat{z}_3)) \allowdisplaybreaks\\
& \leq \int_{\widehat{\XX}^9} \Prob(D^2_{\hat{y}_1,\hat{y}_2}f(Y^{1,2,3}\cup\{\hat{y}_3\})\neq 0)^{\frac{p}{8+2p}} \Prob(D^2_{\hat{y}'_1,\hat{y}'_3}f(Y'^{1,2,3}\cup\{\hat{y}'_2\})\neq 0)^{\frac{p}{8+2p}} \\
& \hskip 2cm \E[|D_{\hat{z}_2}f(Z^{1,2,3}\cup\{\hat{z}_1,\hat{z}_3\})|^{4+p} ]^{\frac{4}{4+p}} \, \widehat{\Q}^{Y_1,\hdots,Z_3}(\dint(\hat{y}_1,\hdots,\hat{z}_3)) \allowdisplaybreaks\\
& \leq \int_{\XX^9} \Prob(D^2_{(y_1,M_{y_1}),(y_2,M_{y_2})}f(\X_{n-3}\cup\{(y_3,M_{y_3})\})\neq 0)^{\frac{p}{8+2p}}\\
& \hskip 1.5cm \Prob(D^2_{(y'_1,M_{y'_1}),(y'_3,M_{y'_3})}f(\X_{n-3}\cup\{(y'_2,M_{y'_2})\})\neq 0)^{\frac{p}{8+2p}} \\
& \hskip 1.5cm \E[|D_{(z_2,M_{z_2})}f(\X_{n-3}\cup\{(z_1,M_{z_1}),(z_3,M_{z_3})\})|^{4+p} ]^{\frac{4}{4+p}} \, \Q^{Y_1,\hdots,Z_3}(\dint(y_1,\hdots,z_3))\\
& \leq c^{\frac{4}{4+p}} \int_{\XX^9} \psi_n'(y_1,y_2) \, \psi_n'(y'_1,y'_3) \, \Q^{Y_1,\hdots,Z_3}(\dint(y_1,\hdots,z_3)).
\end{align*}
If $Y_1\stackrel{a.s.}{=}Y_1'$, this simplifies to
$$
\gamma'_{Y,Y',Z}(f) \leq c^{\frac{4}{4+p}} \int_\XX \bigg( \int_\XX \psi_n'(x,x') \, \Q(\dint x')\bigg)^2 \, \Q(\dint x).
$$
If $Y_{1}$ and $Y'_{1}$ are independent, the Cauchy-Schwarz inequality leads to
$$
\gamma '_{Y,Y',Z}(f) \leq c^{\frac{4}{4+p}}\bigg( \int_{\XX^{2}}\psi_n'(x,x') \, \Q^2(\dint (x,x')) \bigg)^{2}
 \leq c^{\frac{4}{4+p}} \int_{\XX}\bigg(\int_{\XX}\psi_{n}'(x,x')\, \Q(\dint x')\bigg)^{2}\, \Q(\dint x).
$$
Thus, we obtain the desired bound
$$
\frac{\sqrt{n}}{\Var F}\sqrt{n^{2}\gamma '_{Y,Y',Z}(f)}\leq c^{\frac{2}{4+p}}\frac{n^{\frac{3}{2}}}{\Var F}\sqrt{\int_{\XX}\left(
\int_{\XX}\psi_{n}'(x,x') \, \Q(\dint x')
\right)^{2} \, \Q(\dint x)} \leq C(c,p) S_{2}'.
$$

We now show that the remaining terms in  \eqref{eq:LRP} are bounded by a scalar multiple of $S_{3}'$.
 For $1\leq m\leq 4$ and $\Q$-a.e.\ $x\in\XX$,  H\"older's inequality and \eqref{eqn:Conditiongeneral-binomial}
 lead to
\begin{equation}\label{eqn:MomentDxFBinomial}
\begin{split}
\E |D_{(x,M_x)} f(\X_{n-1})|^m & \leq \E[|D_{(x,M_x)}f(\X_{n-1})|^{4+p}]^{\frac{m}{4+p}} \, \Prob( D_{(x,M_x)}f(\X_{n-1})\neq 0)^{\frac{4+p-m}{4+p}}\\
& \leq c^{\frac{m}{4+p}} \Prob( D_{(x,M_x)}f(\X_{n-1})\neq 0)^{\frac{p}{4+p}},
\end{split}
\end{equation}
where we have also used that $\frac{4+p-m}{4+p}\geq \frac{p}{4+p}$. For $1\leq m \leq 4$ and $u\in [1/2,1]$ we derive from \eqref{eqn:MomentDxFBinomial} that
\begin{equation}\label{eqn:IntegratedMomentDxFBinomial}
\int_{\XX} \E\big[|D_{(x,M_{x})}f(\X_{n-1})|^{m}\big]^{u} \, \Q(\dint x) \leq c^{\frac{mu}{4+p}} \int_\XX   \Prob( D_{(x,M_x)}f(\X_{n-1})\neq 0)^{\frac{up}{4+p}} \, \Q(\dint x) \leq c^{\frac{mu}{4+p}}  \frac{\Gamma_n'}{n}.
\end{equation}
This implies immediately that, for $1\leq m\leq 4$,
$$
\E | \D_{1}f(X) |^{m} \leq c^{\frac{m}{4+p}} \int_{\XX} \Prob( D_{(x,M_x)}f(\X_{n-1})\neq 0)^{\frac{p}{4+p}} \, \Q(\dint x) \leq c^{\frac{m}{4+p}} \frac{\Gamma_n'}{n}.
$$
This gives for $m=4$ and $m=3$ that the third and fifth terms in \eqref{eq:LRP} are bounded by
\begin{align*}
\frac{\sqrt{n}\sqrt{\E\D_{1}f(X)^{4}}}{\Var F}+\frac{n\E | \D_{1}f(X) | ^{3}}{(\Var F)^{\frac{3}{2}}}
\leq \frac{c^{\frac{2}{4+p}}\sqrt{\Gamma_n'}}{\Var F}+\frac{c^{\frac{3}{4+p}}\Gamma_n'}{(\Var F)^{\frac{3}{2}}} \leq C(c,p) S_{3}'.
\end{align*}

Lastly, we bound the fourth term in \eqref{eq:LRP} by a scalar multiple of $S_{3}'$.
Let $Y$ be a recombination of $\{X,X',\tilde X\}$. Noting that $Y\equlaw X$, let us estimate
\begin{align*}
& \E  | (f(X)-\E F) (\D _{1}f(Y))^{3}|\\
 &=\E \left|\left( f(X^{1})-\E F+\D_{1}f(X)\right) (\D_{1}f(Y))^{3} \right|\\
&\leq\int_{\XX}\E\left[|f(X^{1})-\E F| \, |D_{(y_{1},M_{y_1})}f(Y^{1})|^{3} \right] \, \Q (\dint y_{1})+\E \left[|\D_{1}f(X)| \, |\D_{1}f(Y)|^{3}\right] \allowdisplaybreaks\\
&\leq \E[(f(\X_{n}^{1})-\E F)^4]^{\frac{1}{4}} \int_{\XX} \E\left[(D_{(x,M_x)}f(\X_{n-1}))^{4}\right]^{\frac{3}{4}} \, \Q(\dint x)+\E (\D_{1}f(X))^{4}\\
&\leq \big(\E[(f(\X_n)-\E F)^4]^{\frac{1}{4}}+ \E[(\D_1f(X))^4]^{\frac{1}{4}} \big) \int_{\XX} \E\left[(D_{(x,M_x)}f(\X_{n-1}))^{4}\right]^{\frac{3}{4}} \, \Q(\dint x)
+ c^{\frac{4}{4+p}}\frac{\Gamma_n'}{n}.
\end{align*}
By  \eqref{eqn:IntegratedMomentDxFBinomial} we have
$$
\int_{\XX} \E\left[(D_{(x,M_x)}f(\X_{n-1}))^{4}\right]^{\frac{3}{4}} \, \Q(\dint x) \leq c^{\frac{3}{4+p}}  \frac{\Gamma_n'}{n}.
$$
From Lemma \ref{lem:BoundFourthMoment} and \eqref{eqn:IntegratedMomentDxFBinomial} it follows that
\begin{align*}
\frac{\E (F-\E F)^4}{(\Var F)^2} & \leq 9 \max\bigg\{\bigg(\frac{32 n}{\Var F} \int_\XX \sqrt{\E (D_{(y,M_y)}f(\X_{n-1}))^4} \, \Q(\dint y) \bigg)^2,\\
& \hskip 2cm 4n \frac{\E (\D_1f(\X_n))^4}{(\Var F)^2} +1 \bigg\}\\
& \leq 9\max\bigg\{\frac{1024c^{\frac{4}{4+p}} \, (\Gamma'_n)^2}{(\Var F)^2}, \frac{4c^{\frac{4}{4+p}} \, \Gamma_n'}{(\Var F)^2}+1\bigg\}.
\end{align*}
All together, the fourth term in \eqref{eq:LRP} satisfies the bound
\begin{align*}
& \frac{n \E  | (f(X)-\E F) (\D_{1}f(Y))^{3}|}{(\Var F)^2}\\
& \leq
  \bigg(\sqrt{3}\max\bigg\{ \frac{4\cdot\sqrt{2} c^{\frac{1}{4+p}} \, \sqrt{\Gamma_n'} }{\sqrt{\Var F}}  , \frac{\sqrt{2}c^{\frac{1}{4+p}} \, \Gamma_n'^{\frac{1}{4}}}{\sqrt{\Var F}}+1\bigg\}+\frac{c^{\frac{1}{4+p}} \, \Gamma_n'^{\frac{1}{4}}}{n^{\frac{1}{4}} \sqrt{\Var F}}  \bigg) \frac{ c^{\frac{3}{4+p}} \, \Gamma_n'}{(\Var F)^{\frac{3}{2}}} +\frac{c^{\frac{4}{4+p}}\Gamma_n'}{(\Var F)^2}\\
 & \leq C(c,p) S_{3}',
\end{align*}
which completes the proof.
\end{proof}

\noindent{\em Remark.} The bounds in Theorem \ref{thm:LPS} and Theorem \ref{thm:general-binomial} are still valid for the Wasserstein distance given in \eqref{eqn:DefinitiondW}. This follows from the fact that the underlying bounds in Theorem 6.1 in \cite{LPS} and Theorem \ref{thm:LRP} (see also Remark 4.3 in \cite{LRP}) are true for the Wasserstein distance as well.

\section{Proofs of Theorem \ref{theoLSY}  and  Theorem \ref{2coroLSY}}  \label{section4}

The bounds in Theorems \ref{thm:LPS} and \ref{thm:general-binomial} are admittedly unwieldy. However when $F$ is a sum of stabilizing score functions, as in
\eqref{Poissonstat} and \eqref{Binomialstat}, then the terms on the right-hand side of \eqref{eq:bound-Po} and \eqref{eq:bound-Bn} conveniently collapse into
the more manageable bounds \eqref{eqn:KolmogorovPoisson} and \eqref{eqn:KolmogorovBinomial}, respectively.

We first provide  several lemmas giving moment and probability bounds for the first and second order difference operators.  Throughout we assume that the hypotheses of Theorem \ref{theoLSY} are in force.  We can assume without loss of generality that $C_{stab}=C_{K}=:C$, $c_{stab}=c_{K}=:c$ and $\alpha_{stab}=\alpha_K=:\alpha$.

\begin{lemm} \label{Lem4.1}
\begin{itemize}
\item []
\item [(a)] For any $x\in\XX$ and $r\geq 0$,
\begin{equation}\label{eqn:UpperBoundBall}
\Q(B(x,r))\leq \kappa r^{\gamma}.
\end{equation}
\item [(b)] For any $\nu>0$ there is a constant $C_\nu\in(0,\infty)$ such that
\begin{equation}\label{eqn:IntegrationFormula}
\int_{\XX\setminus B(x,r)} \exp(-(\beta^{1/\gamma} \d(x,y))^\nu) \, \Q(\dint y) \leq \frac{C_\nu}{\beta} \exp(-(\beta^{1/\gamma} r)^{\nu}/2)
\end{equation}
for all $\beta\geq 1$, $x\in\XX$ and $r\geq 0$.

\end{itemize}
\end{lemm}
\begin{proof}
We prove only (b) since (a) can be shown similarly. We first derive the inequality
\begin{equation}\label{eqn:differenceQ}
\Q(B(x,v))-\Q(B(x,u)) \leq \kappa\gamma \max\{u^{\gamma-1},v^{\gamma-1}\} (v-u)
\end{equation}
for $0<u<v<\infty$. Let $g(t):=\Q(B(x,t))$, $t>0$, and assume that there is a $c\in(0,\infty)$ such that $g(v)-g(u)\geq c(v-u)$. Then, one can construct sequences $(u_n)_{n\in\N}$ and $(v_n)_{n\in\N}$ such that $u_1=u$, $v_1=v$, $u_{n}\leq u_{n+1}<v_{n+1}\leq v_n$, $n\in\N$, $\lim_{n\to\infty} u_n =\lim_{n\to\infty} v_n=:w$, and $g(v_n)-g(u_n)\geq c(v_n-u_n)$, $n\in\N$. Consequently,
$$
c \leq \frac{v_n-w}{v_n-u_n} \frac{g(v_n)-g(w)}{v_n-w} + \frac{w-u_n}{v_n-u_n} \frac{g(w)-g(u_n)}{w-u_n} \leq \max\bigg\{\frac{g(v_n)-g(w)}{v_n-w},\frac{g(w)-g(u_n)}{w-u_n}\bigg\}
$$
and $n\to\infty$ and \eqref{eqn:SurfaceBall} lead to $c\leq \kappa \gamma w^{\gamma-1} \leq \kappa\gamma \max\{u^{\gamma-1},v^{\gamma-1}\}$.

It is sufficient to show \eqref{eqn:IntegrationFormula} for $r>0$ since the case $r=0$ then follows from $r\to0$. For any monotone sequence $(r_n)_{n\in\N}$ with $r_1>r=:r_0$ and $\lim_{n\to\infty}r_n=\infty$ we have
$$
\int_{\XX\setminus B(x,r)} \exp(-(\beta^{1/\gamma} \d(x,y))^\nu) \, \Q(\dint y) \leq \sum_{n=1}^\infty \exp(-(\beta^{1/\gamma} r_{n-1})^\nu) \, \Q(B(x,r_n)\setminus B(x,r_{n-1})).
$$
For $\sup_{n\in\N}|r_n-r_{n-1}|\to 0$ the inequality \eqref{eqn:differenceQ} and the properties of the Riemann integral imply that
\begin{align*}
\int_{\XX\setminus B(x,r)} \exp(-(\beta^{1/\gamma} \d(x,y))^\nu) \, \Q(\dint y) & \leq \int_r^\infty \exp(-(\beta^{1/\gamma} u)^\nu) \, \kappa\gamma u^{\gamma -1} \, \dint u\\
& = \frac{1}{\beta} \int_{\beta^{1/\gamma}r}^\infty \exp(-w^\nu) \, \kappa\gamma w^{\gamma -1} \, \dint w.
\end{align*}
Now a straightforward computation completes the proof of (b).
\end{proof}

Throughout our proofs we only make use of \eqref{eqn:UpperBoundBall} and \eqref{eqn:IntegrationFormula} and not of \eqref{eqn:SurfaceBall} so that one could replace the assumption \eqref{eqn:SurfaceBall} by \eqref{eqn:UpperBoundBall} and \eqref{eqn:IntegrationFormula}.

\begin{lemm}\label{lem:DH}
Let $\mathcal{M}\in\mathbf{N}$ and $\hat{y},\hat{y}_1,\hat{y}_2\in\widehat{\XX}$. Then, for $s\geq 1$,
\begin{align*}
D_{\hat{y}} h_s(\mathcal{M}) & = \xi_s(\hat{y},\mathcal{M}\cup\{\hat{y}\})+\sum_{x\in \mathcal{M}} D_{\hat{y}}\xi_s(x,\mathcal{M})\\
D^2_{\hat{y}_1,\hat{y}_2} h_s(\mathcal{M}) & = D_{\hat{y}_1}\xi_s(\hat{y}_2,\mathcal{M}\cup\{\hat{y}_2\}) +  D_{\hat{y}_2}\xi_s(\hat{y}_1,\mathcal{M}\cup\{\hat{y}_1\}) + \sum_{x\in\mathcal{M}} D^2_{\hat{y}_1,\hat{y}_2}\xi_s(x,\mathcal{M}).
\end{align*}

\end{lemm}

\begin{proof}
In the following let $h:=h_s$ and $\xi:=\xi_s$. By the definition of the difference operator we have that
\begin{align*}
D_{\hat{y}}h(\mathcal{M}) & = \sum_{x\in\mathcal{M}\cup\{\hat{y}\}} \xi(x,\mathcal{M}\cup\{\hat{y}\}) - \sum_{x\in\mathcal{M}} \xi(x,\mathcal{M}) \\
& = \xi(\hat{y},\mathcal{M}\cup\{\hat{y}\}) + \sum_{x\in\mathcal{M}} \big(\xi(x,\mathcal{M}\cup\{\hat{y}\}) - \xi(x,\mathcal{M})\big) \allowdisplaybreaks\\
& = \xi(\hat{y},\mathcal{M}\cup\{\hat{y}\}) + \sum_{x\in\mathcal{M}} D_{\hat{y}}\xi(x,\mathcal{M}).
\end{align*}
For the second-order difference operator this implies that
\begin{align*}
& D^2_{\hat{y}_1,\hat{y}_2} h(\mathcal{M})\\
& = \xi(\hat{y}_2,\mathcal{M}\cup\{\hat{y}_1,\hat{y}_2\}) + \sum_{x\in\mathcal{M}\cup\{\hat{y}_1\}} \!\!\!\!\! D_{\hat{y}_2}\xi(x,\mathcal{M}\cup\{\hat{y}_1\})
- \xi(\hat{y}_2,\mathcal{M}\cup\{\hat{y}_2\}) - \sum_{x\in\mathcal{M}} D_{\hat{y}_2}\xi(x,\mathcal{M}) \allowdisplaybreaks\\
& = D_{\hat{y}_1}\xi(\hat{y}_2,\mathcal{M}\cup\{\hat{y}_2\}) +  D_{\hat{y}_2}\xi(\hat{y}_1,\mathcal{M}\cup\{\hat{y}_1\}) + \sum_{x\in\mathcal{M}} \big(D_{\hat{y}_2}\xi(x,\mathcal{M}\cup\{\hat{y}_1\})-D_{\hat{y}_2}\xi(x,\mathcal{M})\big)\\
& = D_{\hat{y}_1}\xi(\hat{y}_2,\mathcal{M}\cup\{\hat{y}_2\}) +  D_{\hat{y}_2}\xi(\hat{y}_1,\mathcal{M}\cup\{\hat{y}_1\}) + \sum_{x\in\mathcal{M}} D^2_{\hat{y}_1,\hat{y}_2}\xi(x,\mathcal{M}),
\end{align*}
which completes the proof.
\end{proof}

If a point $\hat{y} \in\widehat{\XX}$ is inserted into $\mathcal{M}\in\mathbf{N}$ at a distance exceeding the stabilization radius at $\hat{x}\in\M$, then the difference operator $D_{\hat{y}}$ of the score at $\hat{x}$ vanishes, as seen by the next lemma.

\begin{lemm}\label{lem:DXi}
Let $\mathcal{M}\in\mathbf{N}$, $(x,m_x)\in\M$, $\widehat{\A} \subset \widehat{\XX}$ with $|\widehat{\A}|\leq 6$, $y,y_1,y_2\in\XX$ and $m_y,m_{y_1},m_{y_2}\in\MM$. Then, for $s\geq 1$,
$$
D_{(y,m_y)}\xi_s((x,m_x),\mathcal{M} \cup \widehat{\A}) = 0 \quad \text{ if } \quad R_s((x,m_x),\mathcal{M}) < \d(x,y)
$$
and
$$
D^2_{(y_{1},m_{y_1}),(y_{2},m_{y_2})}\xi_s((x,m_x),\mathcal{M}) = 0 \quad \text{ if } \quad R_s((x,m_x),\mathcal{M}) < \max\{\d(x,y_1),\d(x,y_2)\}.
$$ \end{lemm}
\begin{proof}
Note that $R:=R_s$ and $\xi:=\xi_s$. Moreover, we use the abbreviations $\hat{x}:=(x,m_x)$, $\hat{y}:=(y,m_y)$, $\hat{y}_1:=(y_1,m_{y_1})$ and $\hat{y}_2:=(y_2,m_{y_2})$. Recall that $\widehat{B}(z,r)$ stands for the cylinder $B(z,r)\times\MM$ for $z\in\XX$ and $r>0$. It follows from the definitions of the difference operator and of the radius of stabilization that
\begin{equation}\label{eqn:DyXi}
\begin{split}
D_{\hat{y}}\xi( \hat{x},\mathcal{M}\cup \widehat{\A})
& = \xi(\hat{x},\mathcal{M}\cup \widehat{\A} \cup\{ \hat{y}\}) -\xi( \hat{x},\mathcal{M}\cup \widehat{\A})\\
& = \xi( \hat{x},(\mathcal{M}\cup \widehat{\A} \cup\{ \hat{y}\})\cap \widehat{B}(x,R(\hat{x},\mathcal{M})))\\
& \quad -\xi(\hat{x},(\mathcal{M}\cup \widehat{\A})\cap \widehat{B}(x,R(\hat{x},\mathcal{M}))).
\end{split}
\end{equation}
If $R(\hat{x},\mathcal{M}) < \d(x,y)$, we have
$$
(\mathcal{M}\cup \widehat{\A}   \ \cup\{ \hat{y}\})\cap \widehat{B}(x,R(\hat{x},\mathcal{M}))=(\mathcal{M}\cup \widehat{\A}   \ )\cap \widehat{B}(x,R(\hat{x},\mathcal{M}))
$$
so that the terms on the right-hand side of \eqref{eqn:DyXi} cancel out. For the second order difference operator, we obtain that
\begin{equation}\label{eqn:Dy1y2Xi}
\begin{split}
D^2_{\hat{y}_1,\hat{y}_2}\xi(\hat{x},\mathcal{M})
& = \xi( \hat{x},(\mathcal{M}\cup \{\hat{y}_1,\hat{y}_2\})\cap \widehat{B}(x,R(\hat{x},\mathcal{M})))\\
& \quad - \xi(\hat{x},(\mathcal{M}\cup\{\hat{y}_1\})\cap \widehat{B}(x,R(\hat{x},\mathcal{M})))\\
& \quad - \xi(\hat{x},(\mathcal{M}\cup\{\hat{y}_2\})\cap \widehat{B}(x,R(\hat{x},\mathcal{M})))\\
& \quad + \xi(\hat{x},\mathcal{M}\cap \widehat{B}(x,R(\hat{x},\mathcal{M}))).
\end{split}
\end{equation}
Without loss of generality we can assume that $\d(x,y_1)\geq \d(x,y_2)$. If $R(\hat{x},\mathcal{M}) < \max\{\d(x,y_1),\d(x,y_2)\}=\d(x,y_1)$, we see that
$$
(\mathcal{M}\cup \{ {y}_1, {y}_2\})\cap \widehat{B}(x,R(\hat{x},\mathcal{M})) = (\mathcal{M}\cup \{ {y}_2\})\cap \widehat{B}(x,R(\hat{x},\mathcal{M}))
$$
and
$$
(\mathcal{M}\cup \{\hat{y}_1\})\cap \widehat{B}(x,R(\hat{x},\mathcal{M}))=\mathcal{M}\cap \widehat{B}(x,R(\hat{x},\mathcal{M})),
$$
whence the terms on the right-hand side of \eqref{eqn:Dy1y2Xi} cancel out.
\end{proof}

We recall that $M_x$, $x\in\XX$, always stands for a random mark distributed according to $\Q_\MM$ and associated with the point $x$. Moreover, we tacitly assume that $M_x$ is independent from everything else. For a finite set $\A\subset\XX$, $(\A,M_\A)$ is the shorthand notation for $\{(x,M_x): x\in\A\}$.
The next lemma shows that moments of difference operators of the scores are uniformly bounded. In the following $p\in(0,1]$ and $C_p>0$ come from the moment assumptions \eqref{eqn:momPoisson} and \eqref{eqn:momBinomial}, respectively.

\begin{lemm}\label{lem:BoundsEDXi}
\begin{itemize}
\item []
\item [(a)] For any $\varepsilon\in(0,p]$ and for all $s\geq 1$, $x,y\in\XX$ and $\A \subset \XX$ with $|\A|\leq 6$,
$$
\E |D_{(y,M_y)}\xi_s((x,M_x),\P_s\cup\{(x,M_x)\}\cup (\A,M_{\A})   \ )|^{4+\varepsilon} \leq 2^{4+\varepsilon} C_p^{\frac{4+\varepsilon}{4+p}}.
$$
\item [(b)] For any $\varepsilon\in(0,p]$ and for all $n\geq 9$, $x,y\in\XX$ and $\A \subset \XX$ with $|\A|\leq 6$,
$$
\E |D_{(y,M_y)}\xi_n((x,M_x),\X_{n-8}\cup\{(x,M_x)\}\cup (\A,M_{\A})   \ )|^{4+\varepsilon} \leq 2^{4+\varepsilon} C_p^{\frac{4+\varepsilon}{4+p}}.
$$
\end{itemize}
\end{lemm}
\begin{proof}
It follows from Jensen's inequality, H\"older's inequality and \eqref{eqn:momPoisson} that
\begin{align*}
& \E|D_{(y,M_y)}\xi_s((x,M_x),\P_s\cup (\A,M_{\A}) \cup \{(x,M_x)\}  ))|^{4+\varepsilon}\\
& \leq   2^{3+\varepsilon} \E(|\xi_s((x,M_x),\P_s\cup (\A,M_{\A})\cup \{(y,M_y) \}   \cup \{(x,M_x)\} )|^{4+\varepsilon} \\
&   \quad \quad \quad \quad +|\xi_s((x,M_x),\P_s\cup (\A,M_{\A}) \cup \{(x,M_x)\}  ))|^{4+\varepsilon} ) \\
& \leq 2^{4+\varepsilon} \ C_{p}^{\frac{4+\varepsilon}{4+p}},
\end{align*}
which proves (a). Part (b) follows in the same way from \eqref{eqn:momBinomial}.
\end{proof}

\begin{lemm}\label{lem:BoundsDPoisson}
For any $\varepsilon\in (0,p)$, there is a constant $C_\varepsilon\in(0,\infty)$ only depending on the constants in \eqref{eqn:SurfaceBall}, \eqref{eqn:expstabPoisson}, and \eqref{eqn:momPoisson} such that
$$
\E |D_{(y,M_y)}h_s(\P_s\cup(\A,M_\A))|^{4+\varepsilon} \leq C_\varepsilon
$$
for $y\in\XX$, $\A\subset\XX$ with $|\A|\leq 1$ and $s\geq 1$.
\end{lemm}

\begin{proof}
Fix $y\in \XX$. We start with the case $\A=\emptyset$. It follows from Lemma \ref{lem:DH} and Jensen's inequality that
\begin{align*}
& \E |D_{(y,M_y)}h_s(\P_s)|^{4+\varepsilon}\\
& = \E \bigg| \xi_s((y,M_y),\P_s\cup\{(y,M_y)\})+\sum_{x\in \P_s} D_{(y,M_y)}\xi_s(x,\P_s) \bigg|^{4+\varepsilon}\\
& \leq 2^{3+\varepsilon} \E |\xi_s((y,M_y),\P_s\cup\{(y,M_y)\})|^{4+\varepsilon} + 2^{3+\varepsilon} \E \bigg| \sum_{x\in \P_s} D_{(y,M_y)}\xi_s(x,\P_s) \bigg|^{4+\varepsilon}.
\end{align*}
Here, the first summand is bounded by $2^{3+\varepsilon} (C_{p}+1)$ by assumption \eqref{eqn:momPoisson}. The second summand is a sum of $Z:=\sum_{x\in \P_s} \mathbf{1}_{\{D_{(y,M_y)}\xi_s(x,\P_s)\neq 0\}}$ terms distinct from zero. A further application of Jensen's inequality to the function $x\mapsto x^{4+\varepsilon}$
leads to
\begin{align*}
\bigg| \sum_{x\in \P_s} D_{(y,M_y)}\xi_s(x,\P_s) \bigg|^{4+\varepsilon} & \leq Z^{4+\varepsilon }  \bigg| \sum_{x\in \P_s} Z^{-1}D_{(y,M_y)}\xi_s(x,\P_s) \bigg|^{4+\varepsilon}\\
& \leq Z^{4+\varepsilon }\sum_{x\in \P_s} Z^{-1} |D_{(y,M_y)}\xi_s(x,\P_s)|^{4+\varepsilon}\\
& \leq Z^{4} \sum_{x\in \P_s} |D_{(y,M_y)}\xi_s(x,\P_s)|^{4+\varepsilon}.
\end{align*}
By deciding whether points in different sums are identical or distinct, we obtain that
$$
\E Z^{4} \sum_{x\in \P_s} |D_{(y,M_y)}\xi_s(x,\P_s)|^{4+\varepsilon} = I_1+15I_2+25I_3+10I_4+I_5,
$$
where, for $i\in\{1,\hdots,5\}$,
$$
I_i=\E \sum_{(x_1,\hdots,x_i)\in\P_{s,\neq}^i} \mathbf{1}_{\{D_{(y,M_y)}\xi_s(x_j,\P_s)\neq 0, j=1,\hdots,i\}} \, |D_{(y,M_y)}\xi_s(x_1,\P_s)|^{4+\varepsilon}.
$$
By $\P_{s,\neq}^i$ we denote the set of $i$-tuples of distinct points of $\P_s$. It follows from the multivariate Mecke formula and H\"older's inequality that
\begin{align*}
I_i & = s^i\int_{\widehat{\XX}^i} \E \mathbf{1}_{\{D_{(y,M_y)}\xi_s(x_j,\P_s\cup\{x_1,\hdots,x_i\})\neq 0, j=1,\hdots,i\}} \\ & \hskip 3cm|D_{(y,M_y)}\xi_s(x_1,\P_s\cup\{x_1,\hdots,x_i\})|^{4+\varepsilon} \, \widehat{\Q}^i(\dint(x_1,\hdots,x_i))\\
& \leq s^i\int_{\XX^i} \prod_{j=1}^i \left[
\Prob(D_{(y,M_y)}\xi_s(x_j,\P_s\cup\{(x_1,M_{x_1}),\hdots,(x_i,M_{x_i})\})\neq 0)^{\frac{p-\varepsilon}{4i+pi}}
\right]\\
& \hskip 1.55cm (\E |D_{(y,M_y)}\xi_s(x_1,\P_s\cup\{(x_1,M_{x_1}),\hdots,(x_i,M_{x_i})\})|^{4+p})^{\frac{4+\varepsilon}{4+p}} \, \Q^i(\dint(x_1,\hdots,x_i)).
\end{align*}
Combining this with Lemma \ref{lem:DXi}, \eqref{eqn:expstabPoisson} and Lemma \ref{lem:BoundsEDXi}(a) leads to
\begin{align*}
I_i & \leq 2^{4+\varepsilon}C^{\frac{4+\varepsilon }{4+p}}_{p} s^i\int_{\XX^i} C^{\frac{p-\varepsilon}{4+p}} \prod_{j=1}^i \exp\bigg(-\frac{c (p-\varepsilon)}{4i+pi} \d_s(x_j,y)^\alpha\bigg)  \, \Q^i(\dint(x_1,\hdots,x_i))\\
& = 2^{4+\varepsilon} C^{\frac{4+\varepsilon }{4+p}}_{p} \bigg(s C^{\frac{p-\varepsilon}{4i+pi}} \int_{\XX} \exp\bigg(-\frac{c(p-\varepsilon)}{4i+pi} \d_s(x,y)^\alpha\bigg)  \, \Q(\dint x)\bigg)^i.
\end{align*}
Now \eqref{eqn:IntegrationFormula} with $r=0$ yields that the integrals on the right-hand side are uniformly bounded and thus the first
asserted moment bound holds.

Next we assume that $\A=\{z\}$ with $z\in\XX$. Lemma \ref{lem:DH} and a further application of Jensen's inequality show that
\begin{align*}
& \E |D_{(y,M_y)}h_s(\P_s\cup\{(z,M_z)\})|^{4+\varepsilon}\\
& = \E | \xi_s((y,M_y),\P_s\cup\{(y,M_y),(z,M_z)\})+D_{(y,M_y)}\xi_s((z,M_z),\P_s\cup\{(z,M_z)\})\\
& \quad \quad +\sum_{x\in \P_s} D_{(y,M_y)}\xi_s(x,\P_s\cup\{(z,M_z)\})|^{4+\varepsilon}\\
& \leq 3^{3+\varepsilon} \E |\xi_s((y,M_y),\P_s\cup\{(y,M_y),(z,M_z)\})|^{4+\varepsilon}\\
& \quad  + 3^{3+\varepsilon} \E |D_{(y,M_y)}\xi_s((z,M_z),\P_s\cup\{(z,M_z)\})|^{4+\varepsilon}\\
& \quad  + 3^{3+\varepsilon} \E \bigg| \sum_{x\in \P_s} D_{(y,M_y)}\xi_s(x,\P_s\cup\{(z,M_z)\}) \bigg|^{4+\varepsilon}.
\end{align*}
The last term on the right-hand side can be now bounded by exactly the same arguments as above since these still hold true if one adds an additional point. As the other terms are bounded by \eqref{eqn:momPoisson} and Lemma \ref{lem:BoundsEDXi}(a), this completes the proof.
\end{proof}

\begin{lemm}\label{lem:BoundDBinomial}
For any $\varepsilon\in (0,p)$, there is a constant $C_\varepsilon\in(0,\infty)$ only depending on the constants in \eqref{eqn:SurfaceBall}, \eqref{eqn:expstabBinomial} and \eqref{eqn:momBinomial} such that
$$
\E | D_{(y,M_y)}h_n(\X _{n-1- | \A | }\cup (\A,M_\A)) |^{4+\varepsilon} \leq C_\varepsilon
$$
for $y\in \XX$, $\A\subset \XX$ with $| \A | \leq 2$ and $n\geq 9$.
\end{lemm}

\begin{proof}
Let $\X_{n,\A}:=\X_{n-1-|\A|}\cup(\A,M_\A)$. It follows from Lemma \ref{lem:DH} and Jensen's inequality that
\begin{align*}
& \E |D_{(y,M_y)} h_n(\X_{n,\A})|^{4+\varepsilon}\\
& = \E \bigg| \xi_n((y,M_y),\X_{n,\A}\cup\{(y,M_y)\})+\sum_{x\in \X_{n-1- | \A | }\cup(\A,M_\A)} D_{(y,M_y)}\xi_n(x,\X_{n,\A})\bigg|^{4+\varepsilon} \allowdisplaybreaks\\
& \leq 4^{3+\varepsilon} \E |\xi_n((y,M_y),\X_{n,\A} \cup\{(y,M_y)\})|^{4+\varepsilon} + 4^{3+\varepsilon} \sum_{x\in \A} \E |D_{(y,M_y)}\xi_n((x,M_x),\X_{n,\A})|^{4+\varepsilon}\\
& \quad + 4^{3+\varepsilon} \E \bigg| \sum_{x\in \X_{n-1-|\A|}} D_{(y,M_y)}\xi_n(x,\X_{n,\A}) \bigg|^{4+\varepsilon}.
\end{align*}
On the right-hand side, the first summand is bounded by $4^{3+\varepsilon} (C_{p}+1)$ by assumption \eqref{eqn:momBinomial} (after conditioning on the points of $\X_{n-1- | \A | }\setminus \X_{n-8}$) and the second summand is bounded by $4^{3+\varepsilon}\cdot2\cdot 2^{4+\varepsilon} C_p^{\frac{4+\varepsilon}{4+p}}$ by Lemma \ref{lem:BoundsEDXi}(b). A further application of Jensen's inequality with $Z:=\sum_{x\in \X_{n-1-|\A|}} \mathbf{1}_{\{D_{(y,M_y)}\xi_n(x,\X_{n,\A})\neq 0\}}$ leads to
\begin{align*}
  \bigg| \sum_{x\in \X_{n-1-|\A|}} D_{(y,M_y)}\xi_n(x,\X_{n,\A}) \bigg|^{4+\varepsilon}
 & \leq Z^{3+\varepsilon} \sum_{x\in \X_{n-1-|\A|}} |D_{(y,M_y)}\xi_n(x,\X_{n,\A})|^{4+\varepsilon}\\
 & \leq Z^{4} \sum_{x\in\X_{n-1-|\A|}} |D_{(y,M_y)}\xi_n(x,\X_{n,\A})|^{4+\varepsilon}.
\end{align*}
By deciding whether points in different sums are identical or distinct, we obtain that
$$
\E Z^{4} \sum_{x\in \X_{n-1-|\A|}} |D_{(y,M_y)}\xi_n(x,\X_{n,\A})|^{4+\varepsilon}= I_1+15I_2+25I_3+10I_4+I_5,
$$
where, for $i\in\{1,\hdots,5\}$,
$$
I_i=\E \sum_{(x_1,\hdots,x_i)\in\X_{n-1-|\A|,\neq}^i} \mathbf{1}_{\{D_{(y,M_y)}\xi_n(x_j,\X_{n,\A})\neq 0, j=1,\hdots,i\}} \, |D_{(y,M_y)}\xi_n(x_1,\X_{n,\A})|^{4+\varepsilon}.
$$
It follows from H\"older's inequality that
\begin{align*}
I_i & = \frac{(n-1-|\A|)!}{(n-1-|\A|-i)!}\int_{\widehat{\XX}^i} \E \mathbf{1}_{\{D_{(y,M_y)}\xi_n(x_j,\X_{n-i,\A}\cup\{x_1,\hdots,x_i\})\neq 0, j=1,\hdots,i\}} \\ & \hskip 4.5cm|D_{(y,M_y)}\xi_n(x_1,\X_{n-i,\A}\cup\{x_1,\hdots,x_i\})|^{4+\varepsilon} \, \widehat{\Q}^i(\dint(x_1,\hdots,x_i))\\
& \leq n^i\int_{\XX^i} \prod_{j=1}^i \Prob(D_{(y,M_y)}\xi_n((x_j,M_{x_j}),\X_{n-i,\A}\cup\{(x_1,M_{x_1}),\hdots,(x_i,M_{x_i)}\})\neq 0)^{\frac{p-\varepsilon}{4i+pi}}\\
& \hskip 1.5cm (\E |D_{(y,M_y)}\xi_n((x_1,M_{x_1}),\X_{n-i,\A}\cup\{(x_1,M_{x_1}),\hdots,(x_i,M_{x_i})\})|^{4+p})^{\frac{4+\varepsilon}{4+p}} \\
& \hskip 1.5cm \Q^i(\dint(x_1,\hdots,x_i)).
\end{align*}
Combining this with Lemma \ref{lem:DXi}, \eqref{eqn:expstabBinomial} and Lemma \ref{lem:BoundsEDXi}(b) leads to
\begin{align*}
I_i & \leq 2^{4+\varepsilon} C_{p}^{\frac{4+\varepsilon}{4+p}} n^i\int_{\XX^i} C^{\frac{p-\varepsilon}{4+p}} \prod_{j=1}^i \exp\bigg(-\frac{c(p-\varepsilon)}{4i+pi} \d_n(x_j,y)^\alpha\bigg)  \, \Q^i(\dint(x_1,\hdots,x_i))\\
& = 2^{4+\varepsilon} C_{p}^{\frac{4+\varepsilon}{4+p}} \bigg(n C^{\frac{p-\varepsilon}{4i+pi}} \int_{\XX} \exp\bigg(-\frac{c(p-\varepsilon)}{4i+pi} \d_n(x,y)^\alpha\bigg)  \, \Q(\dint x)\bigg)^i.
\end{align*}
Now \eqref{eqn:IntegrationFormula} yields that the integrals on the right-hand side are uniformly bounded.
\end{proof}

\begin{lemm} \label{L2.5}
\begin{itemize}
\item []
\item [(a)] For $x,z\in\XX$ and $s\geq 1$,
$$
\Prob(D_{(x,M_x)} \xi_s((z,M_z), \P_s\cup\{(z,M_z)\}) \neq 0) \leq 2C \exp( - c \max \{\d_s(x,z), \d_s(z, K)\}^\alpha  ).
$$
\item [(b)] For $x_{1},x_{2},z\in\XX$ and $s\geq 1$,
\begin{align*}
& \Prob(D^2_{(x_1,M_{x_1}),(x_2,M_{x_2})} \xi_s((z,M_z), \P_s\cup\{(z,M_z)\}) \neq 0)\\
& \leq 4C \exp( - c \max \{\d_s(x_1,z), \d_s(x_2, z), \d_s(z, K)\}^\alpha ).
\end{align*}
\end{itemize}
\end{lemm}

\begin{proof}
We prove part (b). By Lemma \ref{lem:DXi} the event
$$
D_{(x_1,M_{x_1}), (x_2,M_{x_2})}^2 \xi_s((z,M_z), \P_s\cup\{(z,M_z)\}) \neq 0
$$
is a subset of the event
that the points $x_1, x_2$ belong to the ball centered at $z$ with radius $R_s((z,M_z), \P_s\cup\{(z,M_z)\})$, i.e., $R_s((z,M_z), \P_s\cup\{(z,M_z)\}) \geq \max\{ \d(x_1, z), \d(x_2, z) \}$.
By \eqref{eqn:expstabPoisson} this gives
$$
\Prob( D^2_{(x_1,M_{x_1}),(x_2,M_{x_2})}\xi_s((z,M_z), \P_s\cup\{(z,M_z)\}) \neq 0) \leq C \exp(-c \max\{ \d_s(x_1,z), \d_s(x_2,z)\}^{\alpha}).
$$
By \eqref{eqn:expfastPoisson} we also have
$$
\Prob( D^2_{(x_1,M_{x_1}),(x_2,M_{x_2})}\xi_s((z,M_z), \P_s\cup\{(z,M_z)\}) \neq 0) \leq 4 C \exp( -c\d_s(z,K)^{\alpha}).
$$
This gives the proof of part (b). Part (a) is proven in a similar way.
\end{proof}

\begin{lemm} \label{lem:PDxineq0Binomial}
\begin{itemize}
\item []
\item [(a)] For $x,z\in\XX$ and $n\geq 9$,
\begin{align*}
& \sup_{\A\subset\XX, |\A|\leq 1}\Prob(D_{(x,M_x)} \xi_n((z,M_z), \X_{n-2-|\A|}\cup\{(z,M_z)\}\cup(\A,M_\A)) \neq 0) \\
& \leq 2C \exp( -c \max \{\d_n(x,z), \d_n(z, K)\}^\alpha  ).
\end{align*}
\item [(b)] For $x_1,x_2,z\in\XX$ and $n\geq 9$,
\begin{align*}
& \sup_{\A\subset\XX, |\A|\leq 1}\Prob(D^2_{(x_1,M_{x_1}),(x_2,M_{x_2})} \xi_n((z,M_z), \X_{n-3-|\A|}\cup\{(z,M_z)\}\cup(\A,M_\A)) \neq 0)\\
& \leq 4C \exp( -c \max \{\d_n(x_1,z), \d_n(x_2, z), \d_n(z, K)\}^\alpha ).
\end{align*}
\end{itemize}
\end{lemm}

\begin{proof}
By Lemma \ref{lem:DXi} and \eqref{eqn:expstabBinomial} together with similar arguments as in the proof of Lemma \ref{L2.5}, we obtain
\begin{align*}
& \Prob(D^2_{(x_1,M_{x_1}),(x_2,M_{x_2})} \xi_n((z,M_z), \X_{n-3-|\A|}\cup\{(z,M_z)\}\cup(\A,M_\A)) \neq 0)\\
& \leq \Prob( R_n((z,M_z),\X_{n-8}\cup\{(z,M_z)\})\geq\max\{\d_s(x_1,z),\d_s(x_2,z)\} )\\
& \leq C \exp(-c \max\{\d_s(x_1,z),\d_s(x_2,z)\}^\alpha).
\end{align*}
It follows from \eqref{eqn:expfastBinomial} that
\begin{align*}
& \Prob(D^2_{(x_1,M_{x_1}),(x_2,M_{x_2})} \xi_n((z,M_z), \X_{n-3-|\A|}\cup\{(z,M_z)\}\cup(\A,M_\A)) \neq 0)\\
& \leq 4 C \exp( -c\d_s(z,K)^{\alpha}),
\end{align*}
which completes the proof of part (b). Part (a) is proven similarly.
\end{proof}

\begin{lemm} \label{L2.6}
\begin{itemize}
\item[]
\item [(a)] Let $\beta \in (0, \infty)$ be fixed. Then there is a constant $C_\beta\in(0,\infty)$ such that
$$
s \int_{\XX} \Prob(D_{(x_1,M_{x_1}), (x_2,M_{x_2})}^2 h_s(\P_s) \neq 0)^\beta  \, \Q(\dint x_2) \leq  C_\beta  \exp(-c\beta \d_s(x_1,K)/4^{\alpha+1})
$$
for all $x_1\in\XX$ and $s \in [1, \infty)$.
\item [(b)] Let $\beta \in (0, \infty)$ be fixed. Then there is a constant $C_\beta\in(0,\infty)$ such that
\begin{align*}
& n \int_{\XX} \sup_{\A\subset\XX, |\A|\leq 1} \Prob(D_{(x_1,M_{x_1}), (x_2,M_{x_2})}^2 h_n(\X_{n-2-|\A|}\cup(\A,M_\A)) \neq 0)^\beta  \, \Q(\dint x_2)\\
& \leq  C_\beta  \exp(-c\beta \d_s(x_1,K)/4^{\alpha+1})
\end{align*}
for all $x_1\in\XX$ and $n\geq 9$.
\end{itemize}
\end{lemm}

\begin{proof}
 By Lemma \ref{lem:DH}  we have
\begin{align*}
D^2_{(x_1,M_{x_1}), (x_2,M_{x_2})} h_s(\P_s) & =  D_{(x_1,M_{x_1})} \xi_s((x_2,M_{x_2}), \P_s\cup\{(x_2,M_{x_2})\})\\
& \quad +  D_{(x_2,M_{x_2})} \xi_s((x_1,M_{x_1}), \P_s\cup\{(x_1,M_{x_1})\})\\
& \quad + \sum_{z \in \P_s} D_{(x_1,M_{x_1}), (x_2,M_{x_2})}^2 \xi_s(z, \P_s)
\end{align*}
so that the Slivnyak-Mecke formula leads to
\begin{align*}
& \Prob( D^2_{(x_1,M_{x_1}),(x_2,M_{x_2})}h_s(\P_s) \neq 0)\\ & \leq \Prob( D_{(x_1,M_{x_1})} \xi_s((x_2,M_{x_2}), \P_s\cup\{(x_2,M_{x_2})\}) \neq 0)\\
& \quad + \Prob( D_{(x_2,M_{x_2})} \xi_s((x_1,M_{x_1}), \P_s\cup\{(x_1,M_{x_1})\}) \neq 0)\\
& \quad + \underbrace{ s \int_{\XX} \Prob(D_{(x_1,M_{x_1}), (x_2,M_{x_2})}^2 \xi_s((z,M_z), \P_s\cup\{(z,M_z)\}) \neq 0) \, \Q(\dint z)}_{=:T_{x_1,x_2,s}}.
\end{align*}
Here, we use part (a) of Lemma \ref{L2.5} to bound each of the first two terms on the right-hand side.
We may bound the first term by
$$
\Prob( D_{(x_1,M_{x_1})}\xi_s((x_2,M_{x_2}), \P_s\cup\{(x_2,M_{x_2})\}) \neq 0) \leq 2C \exp\left( -  c \max \{  \d_s(x_1, x_2), \d_s(x_2, K) \}^\alpha \right).
$$
Observing that $\d_s(x_1, K) \leq 2 \max\{ \d_s(x_2, K), \d_s(x_1, x_2)\}$
we obtain
\begin{align*}
& \Prob( D_{(x_1,M_{x_1})}\xi_s((x_2,M_{x_2}), \P_s\cup\{(x_2,M_{x_2})\}) \neq 0)\\
& \leq 2C \exp\left( -  c \max \{  \d_s(x_1, x_2), \d_s(x_1, K), \d_s(x_2, K)\}^\alpha/2^\alpha \right).
\end{align*}
We bound $\Prob( D_{(x_2,M_{x_2})}\xi_s((x_1,M_{x_1}), \P_s\cup\{(x_1,M_{x_1})\}) \neq 0)$ in the same way. It follows from part (b) of Lemma \ref{L2.5} that
$$
T_{x_1,x_2,s} \leq 4C s \int_{\XX} \exp( - c \max \{\d_s(x_1,z), \d_s(x_2, z), \d_s(z, K)\}^\alpha ) \, \Q(\dint z).
$$
Assume that $\d_{s}(x_{1},K)\geq \d_{s}(x_{2},K)$ (the reasoning is similar if $\d_{s}(x_{2},K)\geq \d_{s}(x_{1},K)$) and let $r=\max\{\d (x_{1},K),\d (x_1,x_2)\}/2$. For any $z\in B(x_1,r)$ the triangle inequality leads to $\max\{\d(z,x_{2}), \d(z, K)\}\geq r$. This implies that
\begin{align*}
T_{x_1,x_2,s} &\leq 4C s \int_{B(x_1,r)} \exp( - c\underbrace{\max\{\d_{s}(z,x_{2}), \d_s(z, K)\}^\alpha}_{\geq (s^{1/\gamma}r)^\alpha } ) \, \Q(\dint z)\\
& \quad + 4C s \int_{\XX\setminus B(x_1,r)} \exp( - c \d_s(x_1 ,z)^\alpha) \, \Q(\dint z).
\end{align*}
Recalling from \eqref{eqn:UpperBoundBall} that $\Q(B(x_1,r)) \leq \kappa r^{\gamma} $, we have
$$
4C s \int_{B(x_1,r)}\exp( - c  (s^{1/\gamma }r)^\alpha ) \, \Q(\dint z) \leq 4C s \kappa r^{\gamma} \exp(-c (s^{1/\gamma} r)^\alpha) \leq C_1 \exp(-c (s^{1/\gamma} r)^\alpha/2)
$$
with a constant $C_1\in(0,\infty)$ only depending on $C$, $c$, $\gamma$ and $\alpha$. On the other hand, \eqref{eqn:IntegrationFormula} yields
$$
4C s \int_{\XX\setminus B(x_1,r)} \exp( - c \d_s(x_1,z)^\alpha) \, \Q(\dint z) \leq C_2 \exp(-c (s^{1/\gamma} r)^\alpha/2)
$$
with a constant $C_2\in(0,\infty)$ only depending on $C$, $c$, $\gamma$ and $\alpha$. Hence, we obtain
$$
T_{x_1,x_2,s} \leq (C_1+C_2) \exp(-c\max\{\d_s(x_1,K), \d_s(x_2,K), \d_s(x_1,x_2)\}^\alpha/2^{\alpha+1})
$$
and
$$
\Prob( D^2_{(x_1,M_{x_1}),(x_2,M_{x_2})}h_s(\P_s) \neq 0)  \leq C_3 \exp(-c\max\{\d_s(x_1,K), \d_s(x_2,K), \d_s(x_1,x_2)\}^\alpha/2^{\alpha+1})
$$
with $C_3:=C_1+C_2+4C$. Consequently, we have
\begin{align*}
& s\int_\XX \Prob( D^2_{(x_1,M_{x_1}),(x_2,M_{x_2})}h_s(\P_s) \neq 0)^\beta  \, \Q(\dint x_2)\\
& \leq C_3^\beta s\int_{B(x_1,\d(x_1,K)/2)} \exp(-c\beta \d_s(x_2,K)^\alpha/2^{\alpha+1}) \, \Q(\dint x_2)\\
& \quad  + C_3^\beta s\int_{\XX \setminus B(x_1,\d(x_1,K)/2)} \exp(-c\beta \d_s(x_2,x_1)^\alpha/2^{\alpha+1}) \, \Q(\dint x_2).
\end{align*}
Using the same arguments as above, the right-hand side can be bounded by
$$
C_\beta \exp(-c\beta \d_s(K,x_1)^\alpha/4^{\alpha+1})
$$
with a constant $C_\beta\in(0,\infty)$ only depending on $\beta$, $C_3$, $c$, $\gamma$ and $\alpha$. This completes the proof of (a).

Similar arguments  show for the binomial case that
\begin{align*}
& \sup_{\A\subset\XX, |\A|\leq 1} \Prob(D_{(x_1,M_{x_1}), (x_2,M_{x_2})}^2 h_n(\X_{n-2-|\A|}\cup (\A,M_\A)) \neq 0) \\
& \leq \sup_{\A\subset\XX, |\A|\leq 1}  \Prob( D_{(x_1,M_{x_1})} \xi_n((x_2,M_{x_2}),\X_{n-2-|\A|}\cup\{(x_2,M_{x_2})\}\cup(\A,M_\A)) \neq 0)\\
& \quad + \sup_{\A\subset\XX, |\A|\leq 1} \Prob( D_{(x_2,M_{x_2})} \xi_n((x_1,M_{x_1}),\X_{n-2-|\A|}\cup\{(x_1,M_{x_1})\}\cup(\A,M_\A)) \neq 0)\\
& \quad + n \int_{\XX} \sup_{\A\subset\XX, |\A|\leq 1}  \Prob(D_{(x_1,M_{x_1}), (x_2,M_{x_2})}^2 \xi_n((z,M_z),\X_{n-3-|\A|}\cup\{(z,M_z)\}\cup(\A,M_\A)) \neq 0) \, \Q(\dint z)\\
& \quad + \sup_{z\in\XX} \Prob(D_{(x_1,M_{x_1}), (x_2,M_{x_2})}^2 \xi_n((z,M_z),\X_{n-3}\cup\{(z,M_z)\}) \neq 0).
\end{align*}
Now similar computations as for the Poisson case conclude the proof of part (b).
\end{proof}

For $\alpha,\tau\geq 0$ put
$$
I_{K,s}(\alpha,\tau):= s \int_{\XX} \exp(-\tau \d_s(x,K)^{\alpha}) \, \Q(\dint x), \quad s\geq 1.
$$

\begin{lemm} \label{L2.7} Let $\beta \in (0, \infty)$ be fixed. There is a constant $\tilde{C}_\beta\in(0,\infty)$ such that for all $s \geq 1$ we have
\begin{equation} \label{eqn:BoundIntegral1}
s \int_{\XX} \left( s \int_{\XX} \Prob(D_{(x_1,M_{x_1}), (x_2,M_{x_2})}^2 h_s(\P_s) \neq 0)^\beta \, \Q(\dint x_2) \right)^2  \, \Q(\dint x_1) \leq \tilde{C}_\beta I_{K,s}(\alpha,c\beta/2^{2\alpha+1}),
\end{equation}
\begin{equation} \label{eqn:BoundIntegral2}
s^2 \int_{\XX^2} \Prob(D_{(x_1,M_{x_1}), (x_2,M_{x_2})}^2 h_s(\P_s) \neq 0)^\beta \, \Q^2(\dint (x_1,x_2)) \leq \tilde{C}_\beta I_{K,s}(\alpha,c\beta/4^{\alpha+1})
\end{equation}
and
\begin{equation} \label{eqn:BoundIntegral3}
s \int_{\XX} \Prob(D_{(x,M_x)} h_s(\P_s) \neq 0)^\beta \, \Q(\dint x)\leq \tilde{C}_\beta I_{K,s}(\alpha,c\beta/2^{\alpha+1}).
\end{equation}
\end{lemm}

\begin{proof}
By Lemma \ref{L2.6}(a) the integrals in \eqref{eqn:BoundIntegral1} and \eqref{eqn:BoundIntegral2} are bounded by
$$
C_\beta^2 s \int_{\XX} \exp(  - c\beta \d_s(x_1, K)^\alpha/2^{2\alpha+1}) \, \Q(\dint x_1) = C_\beta^2 I_{k,s}(\alpha,c\beta/2^{2\alpha+1})
$$
and
$$
C_\beta s \int_{\XX} \exp( -c\beta\d_s(x_1, K)^\alpha/4^{\alpha+1}) \, \Q(\dint x_1) = C_\beta I_{k,s}(\alpha,c\beta/4^{\alpha+1}),
$$
respectively. In order to derive the bound in \eqref{eqn:BoundIntegral3}, we compute $\Prob(D_{(x,M_x)} h_s(\P_s) \neq 0)$ as follows. By Lemma \ref{lem:DH} and the Slivnyack-Mecke formula we obtain that
\begin{align*}
&\Prob(D_{(x,M_x)} h_s(\P_s) \neq 0)\\
& \leq \Prob( \xi_s((x,M_x), \P_s\cup\{(x,M_x)\}) \neq 0)  + \E \sum_{z\in \P_s} \mathbf{1}_{\{ D_{(x,M_x)} \xi_s(z, \P_s) \neq 0\}} \allowdisplaybreaks\\
& = \Prob( \xi_s((x,M_x), \P_s\cup\{(x,M_x)\}) \neq 0)\\
& \quad + s \int_{\XX} \Prob( D_{(x,M_x)} \xi_s((z,M_z), \P_s\cup\{(z,M_z)\}) \neq 0) \, \Q(\dint z).
\end{align*}
By \eqref{eqn:expfastPoisson} and Lemma \ref{L2.5}(a) we obtain that for all $x \in \XX$ and $s\geq 1$,
\begin{align*}
\Prob(D_{(x,M_x)} h_s(\P_s) \neq 0) & \leq  C \exp( -c\d_s(x, K)^\alpha)\\
& \quad + 2C s \int_{\XX } \exp( -c\max \{ \d_s(x,z), \d_s(z, K) \}^\alpha) \, \Q(\dint z).
\end{align*}
Letting $r := \d(x, K)/2$,  partitioning $\XX$ into the union of $\XX \setminus B(x,r)$  and $B(x,r)$,  and following the
discussion in the proof of Lemma \ref{L2.6}, we obtain
$$
2C s \int_{\XX } \exp( -c\max \{ \d_s(x,z), \d_s(z, K) \}^\alpha) \, \Q(\dint z)
 \leq C_1 \exp( - c \d_s(x, K)^\alpha/2^{\alpha+1} )
$$
with a constant $C_1\in(0,\infty)$. Consequently, for all $x \in \XX$ and $s\geq 1$ we have
$$
\Prob(D_{(x,M_x)} h_s(\P_s) \neq 0) \leq   (C+C_1) \exp( - c \d_s(x, K)^\alpha/2^{\alpha+1} )
$$
and
$$
s \int_{\XX} \Prob(D_{(x,M_x)} h_s(\P_s)  \neq 0)^\beta \, \Q(\dint x) \leq (C+C_1) I_{K,s}(\alpha, c\beta/2^{\alpha+1}),
$$
which was to be shown.
\end{proof}

\begin{lemm} \label{lemma:BoundsIntegralsBinomial}
Let $\beta \in (0, \infty)$ be fixed. There is a constant $\tilde{C}_\beta\in(0,\infty)$ such that for all $n\geq 9$ we have
\begin{equation} \label{eqn:BoundIntegral1Binomial}
\begin{split}
& n\int_{\XX} \left( n \int_{\XX} \sup_{\A\subset\XX, |\A|\leq 1}\Prob(D_{(x_1,M_{x_1}), (x_2,M_{x_2})}^2 h_n(\X_{n-2-|\A|}\cup(\A,M_\A)) \neq 0)^\beta \, \Q(\dint x_2) \right)^2  \, \Q(\dint x_1)\\
& \leq \tilde{C}_\beta I_{K,n}(\alpha,c\beta/2^{2\alpha+1}),
\end{split}
\end{equation}
\begin{equation} \label{eqn:BoundIntegral2Binomial}
\begin{split}
& n^2 \int_{\XX^2}  \sup_{\A\subset \XX, |\A|\leq 1}\Prob(D_{(x_1,M_{x_1}), (x_2,M_{x_2})}^2 h_n( \X_{n-2-|\A|}\cup(\A,M_\A))) \neq 0)^\beta \, \Q^2(\dint(x_1, x_2)\\
& \leq \tilde{C}_\beta I_{K,n}(\alpha, c\beta/4^{\alpha+1})
\end{split}
\end{equation}
and
\begin{equation} \label{eqn:BoundIntegral3Binomial}
 n \int_{\XX} \Prob(D_{(x,M_x)} h_n(\X_{n-1}) \neq 0)^\beta \, \Q(\dint x)\leq \tilde{C}_\beta I_{K,n}(\alpha, c\beta/2^{\alpha+1}).
\end{equation}
\end{lemm}

\begin{proof}
The bounds in \eqref{eqn:BoundIntegral1Binomial} and \eqref{eqn:BoundIntegral2Binomial} follow immediately from Lemma \ref{L2.6}(b) and the definition of $I_{K,n}(\alpha,\tau)$.  By Lemma \ref{lem:DH} we obtain that, for $x\in\XX$,
\begin{align*}
& \Prob(D_{(x,M_x)} h_n(\X_{n-1}) \neq 0)\\
& \leq \Prob( \xi_n((x,M_x),\X_{n-1}\cup\{(x,M_x)\}) \neq 0) + \E \sum_{z\in \X_{n-1}} \mathbf{1}_{\{D_{(x,M_x)} \xi_n(z,\X_{n-1}\cup\{z\}) \neq 0\}} \allowdisplaybreaks\\
& \leq \Prob( \xi_n((x,M_x),\X_{n-1}\cup\{(x,M_x)\}) \neq 0)\\
& \quad + n \int_{\XX} \Prob( D_{(x,M_x)} \xi_n((z,M_z), \X_{n-2}\cup\{(z,M_z)\}) \neq 0) \, \Q(\dint z).
\end{align*}
Combining the bound from Lemma \ref{lem:PDxineq0Binomial}(a) with the computations from the proof of Lemma \ref{L2.7} and \eqref{eqn:expfastBinomial}, we see that there is a constant $C_1\in(0,\infty)$ such that for all $x \in \XX$ and $s\geq 1$ we have
$$
\Prob(D_{(x,M_x)} h_n(\X_{n-1}) \neq 0) \leq  C_1 \exp(  -c \d_s( x, K)^\alpha/2^{\alpha+1}).
$$
Now \eqref{eqn:BoundIntegral3Binomial} follows from the definition of $I_{K,n}(\alpha,\tau)$.
\end{proof}

\begin{proof}[Proof of Theorem \ref{theoLSY}]
We start with the proof of the Poisson case \eqref{eqn:KolmogorovPoisson}. It follows from Lemma \ref{lem:BoundsDPoisson} that the condition \eqref{eqn:ConditionLPS} with the exponent $4+p/2$ in Theorem \ref{thm:LPS} is satisfied for all $s\geq 1$ with the constant $C_{p/2}$. In the following we use the abbreviation
$$
I_{K,s}=I_{K,s}(\alpha,c p/(36\cdot 4^{\alpha+1})).
$$
Together with Lemma \ref{L2.7} (with $\beta=p/36$) it follows from Theorem \ref{thm:LPS} that there is a constant $\tilde{C}\in(0,\infty)$ depending on $\tilde{C}_{p/36}$,$C_{p/2}$ and $p$ such that
$$
d_K\bigg(\frac{H_s-\E H_s}{\sqrt{\Var H_s}}, N\bigg) \leq \tilde{C} \bigg(\frac{\sqrt{I_{K,s}}}{\Var H_s} +\frac{I_{K,s}}{(\Var H_s)^{3/2}}+\frac{I_{K,s}^{5/4}+I_{K,s}^{3/2}}{(\Var H_s)^2}\bigg),
$$
which completes the proof of the Poisson case.

For the binomial case \eqref{eqn:KolmogorovBinomial} it follows from Lemma \ref{lem:BoundDBinomial} that the condition \eqref{eqn:Conditiongeneral-binomial} in Theorem \ref{thm:general-binomial} is satisfied with the exponent $4+p/2$ for all $n \geq 9$ with the same constant $C_{p/2}\geq 1$. Using the abbreviation
$$
I_{K,n}=I_{K,n}(\alpha,c p/(18\cdot 4^{\alpha+1})),
$$
we obtain from Lemma \ref{lemma:BoundsIntegralsBinomial} (with $\beta=p/18$) and Theorem \ref{thm:general-binomial} that there is a constant $\tilde{C}\in(0,\infty)$ depending on $\tilde{C}_{p/18}$, $C_{p/2}$ and $p$ such that
$$
d_K\bigg(\frac{H'_n-\E H'_n}{\sqrt{\Var H'_n}}, N\bigg) \leq \tilde{C} \bigg(\frac{\sqrt{I_{K,n}}}{\Var H'_n} +\frac{I_{K,n}}{(\Var H'_n)^{3/2}}+\frac{I_{K,n}+(I_{K,n})^{3/2}}{(\Var H_n')^2}\bigg),
$$
which completes the proof.
\end{proof}

\vskip.5cm

Before giving the proof of Theorem \ref{2coroLSY} we require a lemma. We thank Steffen Winter for discussions concerning the proof. For $K\subset\R^d$, recall that $K_r := \{ y \in \R^d: \ \ \d(y, K) \leq r \}$ and that $\overline{\cal M}^{d-1}(K)$ is  defined at \eqref{eqn:UpperMinkowski}.

\begin{lemm}  \label{SWlemma}
If $K$ is either a full-dimensional compact subset of $\R^d$ with
 $\overline{\cal M}^{d-1}( \partial K) < \infty$ or a $(d-1)$-dimensional compact subset of $\R^d$ with $\overline{\cal M}^{d-1}(K) < \infty$,
then  there exists a constant $C$ such that
 \be  \label{Mcon} {\cal H}^{d-1}( \partial K_r) \leq C(1  + r^{d-1}), \ r > 0. \ee
\end{lemm}

\begin{proof}
By Corollary 3.6 in \cite{RW}, the hypotheses yield $r_0, C_1 \in (0, \infty)$ such that
 $$
{\cal H}^{d-1}( \partial K_r) \leq C_1, \ r \in (0, r_0).
 $$
Combining Lemma 4.1 of \cite{RW} along with Corollaries 2.4 and 2.5 and Equation (2.1) of \cite{RW}, we conclude for almost all $\tilde{r}>0$ that
$$
{\cal H}^{d-1}( \partial K_r) \leq (r/\tilde{r})^{d-1} {\cal H}^{d-1}( \partial K_{\tilde{r}})
$$
for all $r>\tilde{r}$. Choosing such a $\tilde{r}>0$ from $(0,r_0)$ we see that
$$
{\cal H}^{d-1}( \partial K_r) \leq C_1 (1/\tilde{r})^{d-1} r^{d-1}, \ r \in [r_0, \infty),
$$
which completes the proof.
\end{proof}

\vskip.3cm

\begin{proof}[Proof of Theorem \ref{2coroLSY}] Note that we have the same situation as described in Example 1 in Section \ref{sec:MainResults}. In the following we evaluate $I_{K,s}$, which allows us to apply Corollary \ref{coroLSY}. It suffices to show that if $K$ is a full $d$-dimensional subset
of $\XX$, then $I_{K,s} = O(s)$, while $I_{K,s}=O(s^{1-1/d})$ for lower dimensional $K$. Indeed, put
$c :=  \min\{c_{stab},c_K\}p/(36 \cdot 4^{\alpha+1})$, so that
\begin{align*}
I_{K,s} & = s \int_{\XX} \exp( -c \d_s(x, K)^{\alpha} ) g(x) \, \dint x \\
& \leq \|g\|_\infty s \int_{K}  \exp( -c s^{\alpha/d} \d(x, K)^{\alpha} ) \, \dint x + \|g\|_\infty s \int_{\XX \setminus K}  \exp( -c s^{\alpha/d} \d(x, K)^{\alpha} ) \, \dint x \allowdisplaybreaks\\
& = \|g\|_\infty \Vol_d(K) s + \|g\|_\infty s\int_0^{\infty} \int_{ \partial K_r } \exp(- c s^{\alpha/d} r^\alpha) \, \H^{d-1} (\dint y) \, \dint r \allowdisplaybreaks\\
& \leq \|g\|_\infty \Vol_d(K) s + C \|g\|_\infty s\int_0^{\infty} \exp(- c s^{\alpha/d} r^\alpha) \, (1  + r^{d-1} ) \, \dint r \\
& \leq \|g\|_\infty \Vol_d(K) s + C \|g\|_\infty s^{(d-1)/d} \int_0^{\infty} \exp(- c u^\alpha) \, (1  +  s^{-(d-1)/d} u^{d-1}) \, \dint u,
\end{align*}
where the second equality follows by the co-area formula and where the second inequality follows by Lemma \ref{SWlemma}. If $K$ is a full $d$-dimensional subset of $\XX$, then the first integral dominates and is $O(s)$. Otherwise, $\Vol_d(K)$ vanishes and the second integral is $O(s^{(d-1)/d})$.
\end{proof}

\section{Applications}\label{sec:Applications}

By appropriately choosing the measure space $(\XX, {\cal F}, \Q)$, the scores $(\xi_s)_{s \geq 1}$ and $(\xi_n)_{n\in\N}$, and the set $K \subset \XX$, we may use  the general results of Theorem \ref{theoLSY}, Corollary \ref{coroLSY} and Theorem \ref{2coroLSY} to deduce presumably optimal  rates of normal
convergence for statistics in geometric probability. For example, in the setting  $\XX = \R^d$, we expect that all of the statistics $H_s$ and $H'_n$ described in \cite{BY05, Pe07, PY1, PY4, PY5} consist of sums of scores  $\xi_s$ and $\xi_n$ satisfying the conditions of Theorem \ref{2coroLSY}, showing that the statistics in these papers enjoy rates of normal convergence (in the Kolmogorov distance) given by the reciprocal of the standard deviation of $H_s$ and $H_n'$, respectively. Previously, the rates in these papers either contained extraneous logarithmic factors, as in the case of Poisson input, or the rates were sometimes non-existent, as in the case of binomial input. In the following we do this in detail for some prominent statistics featuring in the stochastic geometry literature, including the $k$-face and intrinsic volume  functionals of convex hulls of random samples. Our selection of statistics is illustrative rather than exhaustive and is intended to demonstrate the wide applicability of Theorem \ref{theoLSY} and the relative simplicity of Corollary \ref{coroLSY} and Theorem \ref{2coroLSY}. In some instances the rates of convergence are subject to variance lower bounds, a separate problem not addressed here.

We believe that one could use our approach to also deduce presumably optimal rates of normal convergence for statistics of random sequential packing problems as in \cite{SPY}, set approximation via Delaunay triangulations as in \cite{JY}, generalized spacings as in \cite{BPY}, and general proximity graphs as in \cite{GJR}.

\vskip.3cm

\subsection{Nearest neighbors graphs and statistics of high-dimensional data sets} \label{NNGdatasets}

\noindent a.  {\em Total edge length of nearest neighbors graphs.}
Let $(\XX, {\cal F}, \Q)$ be equipped with a semi-metric $\d$ such that \eqref{eqn:SurfaceBall} is satisfied for some
$\gamma$ and $\kappa$. We equip $\XX$ with a fixed linear order, which is
possible by the well-ordering principle.
Given $\X\in\mathbf{N}$, $k \in \N$, and $x \in \X$, let $V_k(x, \X)$ be the set of $k$ nearest neighbors of $x$, i.e., the $k$ closest points of $x$ in $\X\setminus\{x\}$. In case that that these $k$ points are not unique, we  break the tie via the fixed linear order on $\XX$. The (undirected) nearest neighbor graph $NG_1(\X)$ is the graph with vertex set $\X$ obtained by including  an edge $\{x,y\}$  if $y\in V_1(x,\X)$ and/or $x\in V_1(y,\X)$. More generally, the (undirected) $k$-nearest neighbors graph $NG_k(\X)$ is the graph with vertex set $\X$ obtained by including an edge $\{x,y\}$  if $y\in V_k(x,\X)$ and/or $x\in V_k(y,\X)$. For all $q \geq 0$ define
\be \label{defnn}
\xi^{(q)}(x, \X) := \sum_{y \in V_k(x, \X)} \rho^{(q)}(x, y),
\ee
where $\rho^{(q)}(x,y) := \d(x,y)^q/2$ if $x$ and $y$ are mutual $k$-nearest neighbors, i.e., $x\in V_k(y,\X)$ and $y\in V_k(x,\X)$, and
otherwise $\rho^{(q)}(x,y) := \d(x,y)^q$.
The total edge length of the undirected $k$-nearest neighbors graph on $\X$ with $q$th power weighted edges is
$$
L_{NG_k}^{(q)}(\X) = \sum_{x \in \X} \xi^{(q)}(x, \X).
$$
As usual $\P_s$ is a Poisson point process on $\XX$ with intensity measure $s \Q$ and $\X_n$ is a binomial point process of $n$ points in $\XX$ distributed according to $\Q$.
We assume in the following that $(\XX, \cal F, \Q)$ satisfies, beside \eqref{eqn:SurfaceBall},
\begin{equation}\label{eqn:LowerBoundBall}
\inf_{x \in \XX} \Q (B(x,r))  \geq c r^{\gamma}, \quad r\in[0,\diam(\XX)],
\end{equation}
where $\gamma$ is the constant from \eqref{eqn:SurfaceBall}, $\diam(\XX)$ stands for the diameter of $\XX$ and $c>0$.

\begin{theo}  \label{NNG} If $q \geq 0$ and $\Var (L_{NG_k}^{(q)}(\P_s)) = \Omega(s^{1-2q /\gamma})$,
then there is a $\tilde{C}\in(0,\infty)$ such that
\be \label{nnPoisson}
d_K \left( \frac{ L_{NG_k}^{(q)}( \P_s) - \E L_{NG_k}^{(q)}( \P_s) } { \sqrt{ \Var L_{NG_k}^{(q)}( \P_s) } } , N \right)
\leq  \frac{ \tilde{C} } { \sqrt{s} },   \quad s \geq 1,
\ee
whereas if $\Var ( L_{NG_k}^{(q)}(\X_n)) = \Omega(n^{1-2q /\gamma})$, then
\be \label{nnbinomial}
d_K \left( \frac{ L_{NG_k}^{(q)}(\X_n) - \E L_{NG_k}^{(q)}( \X_n) } {\sqrt{ \Var L_{NG_k}^{(q)}( \X_n) } } , N \right)
\leq \frac{ \tilde{C} } { \sqrt{n} },  \quad n\geq 9.
\ee
\end{theo}

\noindent{\em Remarks.}
(i) {\em Comparison with previous work.} Research has focused  on central limit theorems for $L_{NG_k}^{(q)}(\P_s), s \to \infty,$ and $L_{NG_k}^{(q)}(\X_n), n \to \infty,$ when
$\XX$ is a full-dimensional subset of $\R^d$ and where $\d$ is the usual Euclidean distance.  This includes the seminal work \cite{BB}, the paper \cite{AB} and the more recent works
\cite{PenroseRosoman,PY1, PY5}.  When $\XX$ is a sub-manifold of $\R^d$ equipped with the Euclidean metric on $\R^d$, the paper
\cite{PY6} develops the limit theory for $L_{NG_k}^{(q)}(\P_s), s \to \infty,$ and $L_{NG_k}^{(q)}(\X_n), n \to \infty.$  When $\XX$ is a compact convex subset of $\R^d$, the paper \cite{LPS} establishes the presumably optimal $O(s^{-1/2})$ rate of normal convergence for $L_{NG_k}^{(q)}(\P_s)$.
However these papers neither provide the presumably optimal $O(n^{-1/2})$ rate of normal convergence for  $L_{NG_k}^{(q)}(\X_n)$ in the $d_K$ distance, nor do they consider
input on  arbitrary metric spaces.  Theorem  \ref{NNG} rectifies this.

\vskip.1cm
\noindent (ii) {\em Binomial input.} The rate   for binomial input \eqref{nnbinomial} improves upon the
rate of convergence in the Wasserstein distance $d_W$ given by
\be \label{DWrate}
d_W \left( \frac{ L_{NG_k}^{(q)}(\X_n) - \E L_{NG_k}^{(q)}( \X_n) } {\sqrt{ \Var L_{NG_k}^{(q)}( \X_n) } } , N \right)
= O\left(  \frac{k^4 \tilde{\gamma}_p^{2/p} } {n^{ (p -8)/2p} } + \frac{k^3 \tilde{\gamma}_p^{3/p} } {n^{ (p -6)/2p} } \right),
\ee
as in  Theorem 3.4 of \cite{Chat}
as well as the same rate in the Kolmogorov distance as in Section 6.3 of  \cite{LRP}. Here
$\tilde{\gamma}_p := \E | n^{q /\gamma} \xi^{(q)}(X_1, \X_n) |^p$ and $p > 8$.  For all $\varepsilon > 0$ we have
$\Prob( n^{q /\gamma} \xi^{(q)}(X_1, \X_n) > \varepsilon) = (1 - C \varepsilon^{\gamma}/n)^n$ and it follows that
$\tilde{\gamma}_p^{1/p} \uparrow \infty$  as
$p \to \infty$.  Thus by letting $p \to \infty$, we do not recover the $O(n^{-1/2})$ rate in \eqref{DWrate}, but only achieve the rate $O(n^{-1/2} (\log n)^\tau)$ with some $\tau>0$.

 \vskip.1cm
\noindent (iii) {\em Variance bounds.} When $\XX$ is a full-dimensional  compact convex subset of $\R^d$,  then $\gamma = d$,  $\Var (L_{NG_k}^{(q)}(\P_s)) = \Theta(s^{1-2q/ \gamma})$, and $\Var (L_{NG_k}( \X_n)) = \Theta(n^{1-2q /\gamma})$, which follows from Theorem 2.1 and Lemma 6.3 of \cite{PY1}
(these results treat the case $q = 1$ but the proofs easily extend to arbitrary $q \in (0, \infty)$).
Thus we obtain the required  variance lower bounds of Theorem \ref{NNG}.
If $\Var (L_{NG_k}^{(q)}(\P_s)) = \Omega(s^{1-2q /\gamma})$  does not hold, then the convergence rate in \eqref{nnPoisson}
is replaced by \eqref{eqn:KolmogorovPoisson} with $I_{K,s}$ set to $s$,
with a similar statement if $\Var (L_{NG_k}( \X_n)) = \Omega(n^{1-2q /\gamma})$ does not hold.

\vskip.1cm

\noindent (iv) {\em Extension of Theorem \ref{NNG}.} The {\em directed} $k$-nearest neighbors graph, denoted $NG'_k(\X)$, is the directed graph with vertex set $\X$ obtained by including
 a directed edge from each point to each of its  $k$ nearest neighbors.
 The total edge length of the directed $k$-nearest neighbors graph on $\X$ with $q$th power-weighted edges is
$$
L_{NG'_k}^{(q)}(\X) = \sum_{x \in \X} \tilde{\xi}^{(q)}(x, \X)
$$
where
$$
\tilde{\xi}^{(q)}(x, \X) := \sum_{y \in V_k(x, \X)} \d(x,y)^q.
$$
The proof of Theorem \ref{NNG} given below shows that the analogs of \eqref{nnPoisson} and \eqref{nnbinomial} hold for $L_{NG'_k}^{(q)}(\P_s)$ and
$L_{NG'_k}^{(q)}(\X_n)$ as well.

\begin{proof}
In the following we prove \eqref{nnPoisson}.  We deduce this from Corollary  \ref{coroLSY} with $\xi_s(x, \P_s)$ set to $s^{q /\gamma} \xi^{(q)}(x, \P_s)$, with $\xi^{(q)}$ as at \eqref{defnn} and with
$K$ set to $\XX$. Recalling the terminology of Corollary  \ref{coroLSY}, we have $H_s := s^{q/\gamma} L_{NG_k}^{(q)}( \P_s) = \sum_{x \in \P_s} \xi_s(x, \P_s)$, with $\Var H_s = \Var( s^{q /\gamma} \sum_{x \in \P_s} \xi^{(q)}(x, \P_s)) = \Omega(s)$, by assumption. Recall from \eqref{eqn:IXlinear} that $I_{K,s} = \Theta(s)$.
We claim that $R_s(x, \X\cup\{x\}):= 3 \d(x, x_{kNN}(x,\X\cup\{x\}))$ is a radius of stabilization for $\xi_s(x, \X\cup\{x\})$, where $x_{kNN}(x,\X\cup\{x\})$ is the point of $V_k(x,\X\cup\{x\})$ with the maximal distance to $x$.  Indeed, if a point $y$ is a $k$-nearest neighbor
of $x$, then all of its $k$-nearest neighbors must belong to $B(y, 2\d(x, x_{kNN}(x,\X\cup\{x\})))$, since this ball contains with $x$ and its $k-1$ nearest neighbors
enough potential $k$-nearest neighbors for $y$.

We now show that $R_s(x, \P_s\cup\{x\})$ satisfies exponential stabilization \eqref{eqn:expstabPoisson}. Since $R_s(x,\X\cup\{x\})$ is decreasing in $\X$, we do not need to add a deterministic point set $\A$. Notice that
$$
\Prob(R_s(x, \P_s\cup\{x\}) > r) = \Prob( \P_s(B(x,r/3) < k )), \quad r \geq 0.
$$
The number of points from $\P_s$ in $B(x,r/3)$ follows a Poisson distribution with parameter $s \Q(B(x,r/3))$.  By \eqref{eqn:LowerBoundBall} this exceeds $cs (r/3)^{\gamma}$ if $r\in[0,3\,\diam(\XX)]$. By a Chernoff bound for the Poisson distribution
(e.g. Lemma 1.2 of \cite{Pbook}), there is another constant $\tilde{c}\in(0,\infty)$ such that
$$
\Prob(R_s(x, \P_s\cup\{x\}) > r) \leq k \exp(-\tilde{c}sr^{\gamma} ), \quad r \in [0, 3\,\diam(\XX)].
$$
This also holds for $r > 3\,\diam(\XX)$, since $\Prob(R_s(x, \P_s\cup\{x\}) \geq r) = 0$ in this case.
This gives \eqref{eqn:expstabPoisson}, with $\alpha_{stab} = \gamma$, $c_{stab} = \tilde{c}$, and $C_{stab} = k$.  We may modify this
argument to obtain exponential stabilization with respect to binomial input as at \eqref{eqn:expstabBinomial}.

For all $q \in [0, \infty)$, the scores $(\xi_s)_{s \geq 1}$ also satisfy the  $(4+p)$th moment
condition \eqref{eqn:momPoisson} for all $p \in [0, \infty)$  since
$$
\xi_s(x, \P_s\cup\{x\}\cup\A) \leq k s^{q/\gamma}\d(x, x_{kNN}(x,\P_s\cup\{x\}))^q
$$
for all $\A\subset\XX$ with $|\A|\leq 7$, and the above computation shows that
$s^{q/\gamma}\d(x, x_{kNN}(x,\P_s\cup\{x\}))^q$ has an exponentially decaying tail. The bound \eqref{nnPoisson} follows by Corollary  \ref{coroLSY}. The proof of \eqref{nnbinomial} is similar.
This completes the proof of Theorem \ref{NNG}.
\end{proof}

\noindent b. {\em Statistics of high-dimensional data sets.}  In the case that $\XX$ is an $m$-dimensional $C^1$-submanifold of $\R^d$, with ${\rm {d}}$ the Euclidean distance
in $\R^d$, the {\em directed} nearest neighbors graph version of Theorem \ref{NNG} (cf. Remark (iii) above) may be refined to give rates of normal convergence for  statistics of high-dimensional non-linear data sets.  This goes as follows.  Recall that high-dimensional non-linear data sets are typically modeled as the realization of $\X_n:= \{X_1,...,X_n \},$ with $X_i, 1 \leq i \leq n$, i.i.d.\ copies of a random variable $X$ having support on an unknown (non-linear) manifold $\XX$ embedded in $\R^d$. Typically the coordinate representation
of $X_i$ is unknown, but the interpoint distances are known.  Given this information, the goal is to establish estimators of global characteristics
of $\XX$, including intrinsic dimension, as well as global properties of the distribution of $X$, such as  R\'enyi entropy.
Recall that if the distribution of the random variable $X$ has a Radon-Nikodym derivative $f_X$ with respect to the uniform measure on $\XX$, then given $\rho \in (0, \infty), \rho \neq 1$, the R\'enyi $\rho$-entropy of $X$ is
$$
H_{\rho}(f_X) := (1 - \rho)^{-1} \log \int_{\XX} f_X(x)^{\rho} \, \dint x.
$$

Let $\XX$ be an $m$-dimensional subset of $\R^d$, $m \leq  d$, equipped with the Euclidean metric ${\rm {d}}$ on $\R^d$. Henceforth, assume $\XX$ is an {\em $m$-dimensional $C^1$-submanifold-with-boundary} (see Section 2.1 of \cite{PY6} for details and precise definitions).
Let $\Q$ be a measure on $\XX$ with a bounded density $f_X$ with respect to the uniform surface measure on $\XX$ such that condition \eqref{eqn:SurfaceBall} is satisfied with $\gamma := m$. Note that Example 2 (Section 2) provides conditions which guarantee that \eqref{eqn:SurfaceBall} holds. Assume $f_X$ is bounded away from zero and infinity, and
$$
\inf_x \Q (B(x,r))  \geq c r^{m}, \ \ r  \in [0, \text{diam}(\XX)],
$$
with some constant $c\in(0,\infty)$. The latter condition is called the `locally conic' condition in \cite{PY6} (cf. (2.3)  in  \cite{PY6}).

Under the above conditions and given Poisson input $\P_s$ with intensity $sf_X$,
the main results of  \cite{PY6} establish rates of normal convergence for estimators of intrinsic dimension,
estimators of  R\'enyi entropy,  and for Vietoris-Rips clique counts
(see Section 2 of \cite{PY6} for precise statements).  However these rates contain extraneous logarithmic factors
and \cite{PY6} also stops short of establishing rates of normal convergence
when Poisson input is replaced by binomial input.
In what follows we rectify this for estimators of R\'enyi entropy.  The methods potentially apply to yield rates of normal convergence for estimators of Shannon entropy and intrinsic dimension, but this lies  beyond the scope of this paper.

When $f_X$ satisfies the assumptions stated above and is also continuous on $\XX$, then $n^{q/m-1} L^{(q)}_{NG'_1} (\X_n)$
is a consistent estimator of a multiple of
$\int_{\XX} f_X(x)^{1 - q/m} \, \dint x$, as shown in Theorem 2.2 of \cite{PY6}.  The following result establishes a rate of normal convergence for
$ L^{(q)}_{NG'_k} (\X_n)$ and, in particular, for the estimator  $n^{q/m-1} L^{(q)}_{NG'_1} (\X_n)$.

\begin{theo}  \label{entropy}
If $k\in\N$ and $q \in (0, \infty)$, then there is a constant $c\in(0,\infty)$ such that
\be \label{entropy1}
d_K \left( \frac{L^{(q)}_{NG'_k} (\X_n) - \E L^{(q)}_{NG'_k} (\X_n) } {\sqrt{ \Var L^{(q)}_{NG'_k} (\X_n) } } , N \right)
\leq \frac{ c } {  \sqrt{n} }, \quad n \geq 9.
\ee
A similar result holds if the binomial input $\X_n$ is replaced by Poisson input.
\end{theo}

\noindent{\em Remarks.}
\noindent{(i)} We have to exclude the case $q=0$  since $L^{(0)}_{NG'_1} (\X_n)=kn$ if $n>k$. For the Poisson case a central limit theorem still holds, but becomes trivial since we have $L^{(0)}_{NG'_1} (\P_s)=k |\P_s|$ if $|\P_s|\geq k+1$.

\noindent{(ii)} In the same vein as  Remark (ii) following Theorem \ref{NNG}, Theorem 3.4 of \cite{Chat} yields
a rate of normal convergence for  $L^{(q)}_{NG'_1} (\X_n)$ in the Wasserstein distance  $d_W$ given by the right-hand side of
\eqref{DWrate}.  However, the bound \eqref{entropy1} is superior and is moreover expressed in the Kolmogorov distance $d_K$. When the input $\X_n$ is replaced by Poisson input $\P_s$, we obtain the rate of normal convergence  $O(s^{-1/2})$, improving upon the rates of \cite{PY5, PY6}.

\begin{proof}
Appealing to the method
of proof in Theorem 2.3 of \cite{PY6} and the variance lower bounds of Theorem
6.1 of \cite{PY1}, we see that $\Var L_{NG'_k}^{(q)}(\X_n)  = \Theta(n^{1-2q/m})$
and $\Var L_{NG'_k}^{(q)}(\P_s)  = \Theta(s^{1-2q/m})$. The  proof follows now  the proof of Theorem \ref{NNG}.
\end{proof}

\subsection{Maximal points} \label{maxsection}

Consider the cone  ${\rm{Co}} =(\R^+)^d$ with apex at
the origin of $\R^d$, $d\geq 2$. Given $\X \in\mathbf{N}$, $x \in
\X$ is called  maximal if $({\rm{Co}} \oplus x) \cap
\X = \{x\}$. In other words, a point $x= (x_1,...,x_d) \in
\X$ is maximal if there is no other point $(z_1,...,z_d) \in \X$
with $z_i \geq x_i$ for all $1 \leq i \leq d$. The  maximal layer
$m_{\rm{Co}}(\X)$ is the collection of maximal points in $\X$.  Let
$M_{\rm{Co}}(\X):= \text{card} (m_{\rm{Co}}(\X))$.
Maximal points are of broad interest in computational geometry and economics; see the
books  \cite{PS},  \cite{Ch}, and the survey \cite{Sh}.

Put
$$
\XX := \{  x \in [0, \infty)^d: \ F(x) \leq 1 \}
$$
where $F: [0, \infty)^d \to \R^+$ is a strictly increasing function of each coordinate variable, satisfies {$F(0)<1$}, is continuously differentiable, and
has continuous partials $F_i$, $1 \leq i \leq d$, bounded away from zero and infinity.
Let $\Q$ be a measure on $\XX$ with Radon-Nikodym derivative $g$ with respect to Lebesgue measure on
$\XX$, with $g$  bounded away from zero and infinity. As usual, $\P_s$ is the Poisson point process with intensity $s \Q$ and $\X_n$ is a binomial point process of $n$ i.i.d. points distributed according to $\Q$.

\begin{theo}  \label{maxpts}
There is a constant $c\in(0,\infty)$ such that
\be \label{max2}
d_K\left(\frac{M_{\rm{Co}}(\P_s) - \E M_{\rm{Co}}(\P_s) } {\sqrt{ \Var M_{\rm{Co}}(\P_s)} } , N \right) \leq c s^{-\frac{1} {2} +  \frac{1} {2d}  }, \quad s \geq 1.
\ee
Assuming $\Var  M_{\rm{Co}}(\X_n) = \Omega(n^{(d-1)/d})$, the binomial counterpart to  \eqref{max2} holds, with $\P_s$ replaced by $\X_n$.
\end{theo}

\noindent{\em Remarks.}
(i) {\em Existing results.} The rates of normal convergence given by Theorem \ref{maxpts} improve upon those given in \cite{BX} for Poisson and binomial input for the bounded Wasserstein distance and in \cite{BX1} and \cite{Yu} for Poisson input for the Kolmogorov distance.  While these findings are also proved via the Stein method, the local dependency methods employed  there all incorporate extraneous logarithmic factors. Likewise, when $d = 2$, the paper \cite{BHT} provides rates of normal convergence in the Kolmogorov distance for binomial input, but aside from the special case that $F$ is linear, the rates incorporate extraneous logarithmic factors. The precise approximation bounds of
Theorem \ref{maxpts} remove the logarithmic  factors in \cite{BHT, BX, BX1, Yu}.

\vskip.3cm
\noindent (ii) We have taken   $\rm{Co} = (\R^+)^d$ to simplify the presentation,
 but the results extend to  general cones which are subsets of $(\R^+)^d$ and which have apex at the origin.

\begin{proof}[Proof of Theorem \ref{maxpts}]
We deduce this theorem from Theorem \ref{2coroLSY}(b) and consider score functions
$$
\zeta(x, \X)  := \begin{cases} 1  & \text{ if }   (({\rm{Co}} \oplus x) \cap \XX ) \cap \X = \{x\}
\\
0  & \text{ otherwise}.\end{cases}
$$
Notice that $M_{\rm{Co}}(\X)= \sum_{x \in \P_s} \zeta(x,\X)$. Put $K:= \{ x \in [0, \infty)^d: \ F(x) = 1 \}$. The assumptions on $F$ imply $\overline{\cal M}^{d-1}(K)< \infty$. In the following, we only prove \eqref{max2}  for Poisson input, as the proof for binomial input is similar. Thus we only need to show that the scores $\zeta_s \equiv \zeta$ satisfy the conditions of Theorem \ref{2coroLSY}(b).  First, $\zeta$ is bounded and so satisfies
the $(4 + p)$th moment condition \eqref{eqn:momPoisson}
for all $p \in [0, \infty)$.
As shown in \cite{Yu} (see proof of Theorem 2.5 in Section 6),  $\zeta_s$ also  satisfy
exponential stabilization \eqref{eqn:expstabPoisson} with $\alpha_{stab} = d$ with respect to Poisson input $\P_s$.
Also, we assert that the scores decay  exponentially fast with the distance to $K$ with $\alpha_K = d$.
To see this, let $r(x):= \d(x, K)$ be the distance between $x$ and $K$ and note that $({\rm{Co}} \oplus x) \cap \XX$ contains
the set  $S(x):= B(x, r(x)) \cap  ({\rm{Co}} \oplus x)$.
It follows that
\begin{align*}
\Prob(\zeta(x, \P_s\cup \{x\}) \neq 0)   & =  \Prob( ({\rm{Co}} \oplus x) \cap \XX ) \cap \P_s = \{x\}) = \exp\left( - s \int_{ ({\rm{Co}} \oplus x) \cap \XX} d \Q \right) \\
& \leq \exp( - s \Q(S(x))) \leq \exp( -\bar{c} \d_s(x,K)^d)
\end{align*}
with some constant $\bar{c}:= \bar{c}(\Q) \in(0,\infty)$, and thus \eqref{eqn:expfastPoisson} holds with $\alpha_K = d$.

We now show $\Var  M_{\rm{Co}}(\P_s) = \Theta(s^{(d-1)/d})$.
The hypotheses on $F$ imply that there are $M = \Theta(s^{(d-1)/d})$ disjoint sets
 $S_i:= ({\rm{Co}} \oplus x_i) \cap \XX, i = 1,...,M$, with $x_i \in \XX$,  such that $Q_i:= [0,s^{-1/d}]^d \oplus x_i \subset S_i$ and $x_i+s^{-1/d}e\in K$, where $e=(1,\hdots,1)\in\R^d$. Given $x_i$, for $1 \leq j \leq d$ define $d$ sub-cubes of $Q_i$
 $$
 Q_{ij}:=  \Big( \frac{2}{3} s^{-1/d}, s^{-1/d}\Big]^{j-1}\times \Big[0, \frac{1}{3} s^{-1/d}\Big)\times \Big( \frac{2}{3} s^{-1/d}, s^{-1/d}\Big]^{d-j}  \oplus x_i,
 $$
 as well as the central cube $\tilde{Q}_i := \Pi_{j=1}^d [ \frac{1}{3} s^{-1/d}, \frac{2}{3} s^{-1/d}] \oplus x_i$.  All cubes thus constructed are disjoint.  Say that $S_i, 1 \leq i \leq M,$ is admissible if there are points
 $$
 p_{ij} \in \P_s \cap Q_{ij}, \  1 \leq j \leq d,
 $$
 which are maximal and $S_i \setminus \tilde{Q}_i$ contains no other points in $\P_s$.  Given that $S_i$ is admissible,
 we assert that  the maximality status of points
 in $\P_s \cap \tilde{Q}_i^c$ is unaffected by the (possibly empty) configuration of $\P_s$ inside $\tilde{Q}_i$.  Indeed, if
 $x \in \P_s \cap \tilde{Q}_i^c \cap Q_i$, then $x \in \{p_{ij}\}_{j=1}^d$ and so $({\rm{Co}} \oplus x) \cap \tilde{Q}_i = \emptyset$, showing the assertion in this case.  On the other hand, if  $x \in \P_s \cap \tilde{Q}_i^c\cap Q_i^c$ and if  $({\rm{Co}} \oplus x) \cap \tilde{Q}_i \neq \emptyset$, then
 ${\rm{Co}} \oplus x$  must contain at least one of the cubes $Q_{ij}$,  thus   ${\rm{Co}} \oplus x$ contains at least one  of the points $\{p_{ij}\}_{j=1}^d$ and hence $\zeta(x, \P_s)$ vanishes.  Let $I$ be the indices $i \in \{1,...,M\}$ such that $S_i$ is admissible.

 Let ${\cal F}_s$ be the sigma algebra generated by $I$ and $\P_s \cap (\XX \setminus \cup_{i \in I}\tilde{Q}_i)$, including the maximal points $\{ \{p_{ij} \}_{j=1}^d \}_{i \in I}$.
 Conditional on ${\cal F}_s$, note that $\zeta(x, \P_s)$ is deterministic for $x \in \P_s \cap (\XX \setminus \cup_{i \in I}\tilde{Q}_i)$.
 The conditional variance formula gives
 \begin{align*}
\Var M_{\rm{Co}}(\P_s)   & \geq \E \Var[  \sum_{x \in \P_s \cap \cup_{i \in I} \tilde{Q}_i} \zeta(x, \P_s)  + \sum_{x \in \P_s \cap (\XX \setminus \cup_{i \in I} \tilde{Q}_i)) } \zeta(x, \P_s) | \ {\cal F}_s ]
\\
& = \E \Var[ \sum_{i \in I}  \sum_{x \in \P_s \cap  \tilde{Q}_i} \zeta(x, \P_s) | \ {\cal F}_s ] 
 = \E \sum_{i \in I} \Var[  \sum_{x \in \P_s \cap \tilde{Q}_i} \zeta(x, \P_s) | \ {\cal F}_s ]
\end{align*}
where the last equality follows by independence of  $\sum_{x \in \P_s \cap  \tilde{Q}_i} \zeta(x, \P_s), i \in I$. For $i\in I$ the number of maximal points in $\tilde{Q}_i$ only depends on the restriction of $\P_s$ to $\tilde{Q}_i$ and, thus, exhibits non-zero variability. Together with the bounds on $g$, we obtain that $\Var[  \sum_{x \in \P_s \cap \tilde{Q}_i} \zeta(x, \P_s) | \ {\cal F}_s ] \geq c_1 > 0$ uniformly in $i\in I$. Since $\E {\rm{card}}(I)\geq c_2 s^{(d-1)/d}$ with $c_2\in(0,\infty)$ the asserted variance lower bound follows. Theorem \ref{2coroLSY}(b) gives \eqref{max2}.

The proof method in \cite{Yu} is for Poisson input $\P_s$, but it may be easily extended to show that the $(\zeta_n)_{n \geq 1}$ are exponentially stabilizing with respect to binomial input and that $(\zeta_n)_{n \geq 1}$ decay exponentially fast with the distance to $K$.
Thus the conditions of Theorem \ref{2coroLSY}(b) are satisfied,
and so \eqref{max2} follows from \eqref{eqn:KolmogorovPoissoncoro},
concluding the proof of Theorem \ref{maxpts}.
\end{proof}

\subsection{Set approximation via Voronoi tessellations} \label{Voronoisection}

Throughout this subsection let $\XX :=[-1/2,1/2]^d, d \geq 2,$ and let $A\subset {\text{int}}(\XX)$ be a full-dimensional subset of $\R^d$. Let $\Q$ be the uniform measure on $\XX$. For $\X\in\mathbf{N}$ and $x\in\X$ the Voronoi cell $C(x, \X)$ is the set of all $z\in \XX$ such that the distance between $z$ and $x$ is at most equal to the distance between $z$ and any other point of $\X$. The collection of all $C(x,\X)$ with $x\in\X$  is called the Voronoi tessellation  of $\XX$. The Voronoi approximation of $A$ with respect to $\X$ is the union of all Voronoi cells $C(x,\X), x\in\X,$ with $x\in A$, i.e.,
$$
A(\X):=\bigcup_{x\in \X \cap A}C(x,\X).
$$
In the following we let $\X$ be either a Poisson point process  $\P_s$, $s\geq1$, with intensity measure $s\Q$ or a binomial point process  $\X_n$ of $n\in\N$ points distributed according to $\Q$. We are now interested in the behavior of the random approximations
$$
A_s:=A(\P_s), \quad s\geq 1, \quad \text{and} \quad A'_{n}:=A(\X_n), \quad n\in\N,
$$
of $A$. Note that $A_s$ is also called the Poisson-Voronoi approximation.

 Typically $A$ is an unknown set having unknown  geometric characteristics such as volume and surface area. Notice that $A_s$ and $A'_n$ are random polyhedral approximations of $A$, with volumes closely approximating that of $A$ as $s$ and $n$ become large.  There is a large literature devoted to quantifying this approximation and we refer to \cite{LRP, Yu} for further discussion and references. One might also expect that  $\H^{d-1}(\partial A_s)$ closely approximates
 a scalar multiple of $\H^{d-1}(\partial A)$, provided the latter quantity exists and is finite.  This has been shown in \cite{Yu}.
Using Theorem \ref{2coroLSY}(b) we deduce rates of normal convergence for the volume and surface area statistics of $A_s$ and $A'_n$ as well as
$\Vol(A_s \Delta A)$ and $\Vol(A'_n \Delta A)$. Here and elsewhere in this section we abbreviate $\Vol_d$ by $\Vol$. The symmetric difference $U\Delta V$ of two sets $U,V\subset\R^d$ is given by $U\Delta V:= (U\setminus V)\cup (V\setminus U)$.

\begin{theo}  \label{Vor}
\begin{itemize} \item []
\item [(a)] Let $A\subset (-1/2,1/2)^d$ be closed and such that $\partial A$ satisfies $\overline{\cal M}^{d-1}(\partial A)< \infty$ and contains a $(d-1)$-dimensional $C^2$-submanifold and let $F\in\{\Vol,\Vol(\cdot\Delta A),\H^{d-1}(\partial \cdot) \}$. Then there is a constant $\tilde{C}\in(0,\infty)$ such that
\be \label{VorboundGeneral}
d_K \left( \frac{F(A_s) -  \E F(A_s)}{ \sqrt{ \Var F(A_s)} },  N \right) \leq  \tilde{C} s^{ -\frac{(d - 1)} {2d} },  \quad s \geq 1,
 \ee
and
\be \label{VorboundbinomGeneral}
d_K \left( \frac{ F(A'_{n}) -  \E F(A'_{n})}{ \sqrt{ \Var F(A'_{n})} },  N \right) \leq  \tilde{C} n^{ -\frac{(d - 1)} {2d} },  \quad n \geq 9,
\ee
as well as
\be \label{Vorbound}
d_K \left( \frac{ \Vol(A_s) -  \Vol(A)}{ \sqrt{ \Var \Vol(A_s)} },  N \right) \leq  \tilde{C} s^{ -\frac{(d - 1)} {2d} },  \quad s \geq 1,
 \ee
and
\be \label{Vorboundbinom}
d_K \left( \frac{ \Vol(A'_{n}) -  \Vol(A)}{ \sqrt{ \Var \Vol(A'_{n})} },  N \right) \leq  \tilde{C} n^{ -\frac{(d - 1)} {2d} },  \quad n \geq 9.
\ee
\item [(b)] If $ F = \Vol$ and $A\subset(-1/2,1/2)^d$ is compact and convex, then all of the above inequalities are in force.
\end{itemize}
\end{theo}

\noindent{\em Remarks.}
\noindent (i) The bound \eqref{VorboundGeneral}  provides a rate of convergence for the main result of \cite{Schulte} (see Theorem 1.1 there), which establishes asymptotic normality for
 $\Vol(A_s)$, $A$ convex. The bound \eqref{VorboundGeneral} also improves  upon Corollary 2.1 of \cite{Yu} which shows
 $$
 d_K \left( \frac{ \Vol(A_s) - \E \Vol(A_s)}{\sqrt{ \Var
 \Vol(A_s)} } ,  N \right)  = O \left(  (\log s)^{3d + 1} s^{-\frac{(d - 1)} {2d} } \right).
 $$
Recall that the normal convergence of $\H^{d-1}(\partial A_s)$ is given in Remark (i) after Theorem 2.4 of  \cite{Yu} and the bound \eqref{VorboundGeneral} for $F=\mathcal{H}^{d-1}(\partial \cdot) $ provides a rate for this normal convergence.

 \noindent (ii)  The bound  \eqref{Vorboundbinom}  improves upon the bound of Theorem 6.1 of \cite{LRP},  which contains extra logarithmic factors, and, thus, addresses an open problem raised in Remark 6.9  of \cite{LRP}.

\noindent (iii) We may likewise deduce identical rates of normal convergence for other geometric statistics of $A_s$, including
the total number of $k$-dimensional faces of $A_s$, $k \in \{0,1,...,d-1\}$, as well
as the $k$-dimensional Hausdorff measure of the union  of the $k$-dimensional faces of $A_s$ (thus when $k = d-1$, this gives $\H^{d-1}(\partial A_s)$).
Second order asymptotics, including the requisite variance lower bounds for these statistics, are established in \cite{TY}. In the case of geometric statistics of $A'_{n}$, we expect similar variance lower bounds and central limit theorems.

\noindent (iv) Lower bounds for $\Var F(A_s)$ and $\Var F(A'_{n})$ are essential to
showing \eqref{VorboundGeneral}-\eqref{Vorboundbinom}.  We expect the order of these
bounds to be unchanged if  $\Q$ has a density bounded away from zero and
infinity. We thus expect Theorem \ref{convexhull} to remain valid in this context because all other arguments in our proof hold for such $\Q$.

\begin{proof}
We first prove \eqref{VorboundGeneral} for $F=\Vol$ and $F=\Vol(\cdot\Delta A)$. The proof method extends easily to the case when Poisson input is replaced by binomial input and we sketch the details as needed. To deduce \eqref{VorboundGeneral} from Theorem \ref{2coroLSY}(b), we need to (i) express $sF(A_s)$
as a sum of stabilizing score functions and (ii) define $K \subset \XX$ and show that the scores decay exponentially fast
with respect to $K$.

\noindent(i) {\em Definition of scores}. As in \cite{Yu}, for $\X\in\mathbf{N}$, $x \in \X$, and a fixed subset $A$ of $\XX$,  define the scores
\begin{equation}\label{defp}
\nu^{ \pm  }(x, \X)  :=\begin{cases}
\Vol (C(x, \X) \cap A^c)  & \text{ if } x \in A
\\
\pm \Vol (C(x, \X) \cap A)  & \text{ if } x \in A^c.
\end{cases}
\end{equation}

\noindent Define  $\nu^{\pm}_s(x, \X):= s \nu^{\pm}(x, \X)$.  By the
definition of $\nu^{\pm}$ at \eqref{defp} we have
$$
s\Vol(A_s) = \sum_{x \in \P_s} \nu^-_{s}(x,\P_s)+s\Vol(A) \quad \text{ and } \quad s\Vol(A \Delta  A_s) = \sum_{x \in \P_s} \nu^+_{s}(x, \P_s).
$$
The arguments of Section 5.1 of \cite{Pe} show that the scores $\nu^{ \pm  }_s$  have a radius of stabilization $R_s(x, \P_s\cup\{x\})$
with respect to $\P_s$ which satisfies \eqref{eqn:expstabPoisson} with $\gamma= d$ and $\alpha_{stab} = d$. The scores $\nu^{\pm}_s$ also satisfy the $(4 + p)$th moment condition \eqref{eqn:momBinomial} for all $p \in [0, \infty)$.

As remarked in \cite{Yu} and as shown in Lemma 5.1 of  \cite{Pe}, the scores $\nu^{ \pm  }_n$
have a radius of stabilization $R_n(x, \X_{n - 8}\cup\{x\})$ with respect to binomial input $\X_n$
which satisfies  \eqref{eqn:expstabBinomial}  with $\gamma= d$ and $\alpha_{stab} = d$.

\noindent (ii) {\em Definition of $K$}.  We set $K$ to be $\partial A$.
As noted in the proof of Theorem 2.1 of \cite{Yu},  we assert that the scores $\nu^{ \pm  }_s$ decay exponentially fast with
their distance to $\partial A$, i.e. they satisfy \eqref{eqn:expfastPoisson} and \eqref{eqn:expfastBinomial} when $K$ is set to $\partial A$
and with $\alpha_{K} = d$. To see this for Poisson input, note that
$$
\Prob(\nu^{\pm}_s(x, \P_s\cup\{x\}) \neq 0) \leq \Prob( \text{diam} ( C(x, \P_s\cup\{x\})) \geq \d(x, K))
$$
for $x\in[-1/2,1/2]^d$. Since $\text{diam} ( C(x, \P_s\cup\{x\})) \leq 2 R_s(x, \P_s\cup\{x\})$ and since $R_s(x, \P_s\cup\{x\})$ has exponentially decaying tails, the assertion follows.

We deduce \eqref{VorboundGeneral} from the
bound \eqref{aratessurface} of Theorem \ref{2coroLSY}(b) as follows.
If either $\partial A$ contains a $(d-1)$-dimensional $C^2$-submanifold or $A$ is
compact and convex, then  $s^2 \Var \Vol( A_s) = \Omega(s^{(d-1)/d})$; see Theorem 1.2 of \cite{Schulte}, Theorem 1.1 of \cite{TY} and Theorem 2.2 of \cite{Yu}.
All conditions of Theorem \ref{2coroLSY} are satisfied and so \eqref{VorboundGeneral} follows for $F=\Vol$. Replacing $\Vol(A_s)$ with
$\Vol( A \Delta A_s)$, \eqref{VorboundGeneral} holds if  $\partial A$ contains a $(d-1)$-dimensional $C^2$-submanifold. This assertion follows since  the stated conditions imply $s^2\Var \Vol( A \Delta A_s) = \Omega(s^{(d-1)/d})$, as shown in Theorem 2.2
of \cite{Yu}. We may similarly deduce \eqref{VorboundbinomGeneral} from the
bound \eqref{aratessurface} of Theorem \ref{2coroLSY}(b).  If either $\partial A$ contains  a $(d-1)$-dimensional $C^2$-submanifold or $A$ is
compact and convex, then $n^2 \Var \Vol( A'_{n}) = \Omega(n^{(d-1)/d})$ as shown in
Theorem 2.3 of \cite{Yu}. Thus \eqref{VorboundbinomGeneral} follows for $F=\Vol$. Considering now $F=\Vol(\cdot\Delta A)$,
and appealing to the variance lower bounds of Theorem 2.3 of \cite{Yu}, we see that when
$\partial A$ contains a $(d-1)$-dimensional $C^2$-submanifold,
all conditions of Theorem \ref{2coroLSY}(b) are satisfied in the context of  binomial input,
and so the bound \eqref{VorboundbinomGeneral} follows  for $F=\Vol(\cdot\Delta A)$.

To deduce \eqref{Vorbound} from \eqref{VorboundGeneral}, we need to replace $\E \Vol( A_{s})$ with $\Vol(A)$. As shown in \cite[Theorem 2]{LRVega},  if the random input consists of $n$ i.i.d. uniformly distributed random variables then $\left|
\E \Vol(A'_{n})-\Vol(A)
\right|\leq c^{n}$
for some $c\in (0,1).$ A similar statement holds for Poisson input $\P_s$: If $|\P_s|$ is the cardinality of $\P_s$, then
\begin{align*}
\left|
\E\Vol(A_{s})-\Vol(A)
\right|= \sum_{n\in \mathbb{N}} \Prob(|\P_s|=n)\left|
\E\Vol(A'_{n})-\Vol(A)
\right|\leq \exp(s(c-1)).
\end{align*}
This exponential bias allows one to replace $\E\Vol(A_{s})$ by $\Vol(A)$ in \eqref{VorboundGeneral} and similarly for  $\E\Vol(A'_{n})$.
This gives \eqref{Vorbound} and \eqref{Vorboundbinom}.

We now show \eqref{VorboundGeneral} for $F=\mathcal{H}^{d-1}(\partial \cdot) $ and that it also holds when Poisson input is replaced by binomial input.
Given $\X\in\mathbf{N}$, define for $x \in \X \cap A$
the score $\alpha(x, \X)$ to be the $\H^{d-1}$
measure of the $(d-1)$-dimensional faces of $C(x,\X)$ belonging to the
boundary of $\bigcup_{w \in \X \cap A} C(w, \X)$;  if there are no such faces or
if $x \notin \X \cap A$,
then set $\alpha(x, \X)$ to be zero.

Put $\alpha_s(x, \X):= s^{(d-1)/d} \alpha(x, \X)$. Recalling the notation in \eqref{Poissonstat} and \eqref{Binomialstat}, the
 surface area of $A_s$ and $A'_n$ is then given by
$$
s^{(d-1)/d} \H^{d-1}(\partial A_s)= h_s(\P_s) = \sum_{x \in \P_s} \alpha_s(x, \P_s)
$$
and
$$
n^{(d-1)/d}\H^{d-1}( \partial A'_{n})= h_n(\X_n) = \sum_{x \in \X_n} \alpha_n(x, \X_n),
$$
respectively.
We want to deduce \eqref{VorboundGeneral} and  \eqref{VorboundbinomGeneral} for $F=\mathcal{H}^{d-1}(\partial\cdot)$ from Theorem \ref{2coroLSY}(b) with $K$ set to $\partial A$. As shown in the proof of Theorem 2.5 of \cite{Yu}, the scores $\alpha_s$  are exponentially stabilizing with respect to Poisson and
binomial input. In other words they satisfy \eqref{eqn:expstabPoisson}  and \eqref{eqn:expstabBinomial}   with $\gamma= d$ and $\alpha_{stab} = d$.
They also satisfy the $(4 + p)$th moment conditions \eqref{eqn:momPoisson} and \eqref{eqn:momBinomial}  for all $p \in [0, \infty)$.
As noted in the proof of Theorem 2.5 of \cite{Yu},  the scores $\alpha_s$ decay exponentially fast with
their distance to $\partial A$, i.e. they satisfy \eqref{eqn:expfastPoisson} and \eqref{eqn:expfastBinomial} when $K$ is set to $\partial A$.
 We note that
\be \label{feb26}
 \Var \H^{d-1}( \partial A_s) =
\Theta(s^{-(d-1)/d}),
\ee
as shown in  Theorem 1.1 of \cite{TY}. We assert that
$$
\Var \H^{d-1}( \partial A'_{n}) = \Theta(n^{-(d-1)/d}).
$$
This may be proved by mimicking the methods to prove  \eqref{feb26} or, alternatively,
with $Z(n)$ denoting an independent Poisson random variable with mean $n$, we could use Lemma 6.1 of \cite{Yu}
to show $|\Var h_n(\X_{Z(n)}) - \Var h_n(\X_n)  | = o(n^{(d-1)/d})$. Hence, all conditions of Theorem \ref{2coroLSY}(b) are satisfied for Poisson and for binomial input. This gives \eqref{VorboundGeneral} and \eqref{VorboundbinomGeneral} for $F=\mathcal{H}^{d-1}(\partial\cdot)$, as desired.
\end{proof}

 \subsection{Statistics of convex hulls of random point samples}  \label{conhull}

In the following let $A$ be a compact convex subset of $\R^d$ with non-empty interior, $C^2$-boundary and positive Gaussian curvature. By $\Q$ we denote the uniform measure on $A$. Let $\P_s$, $s\geq 1$, be a Poisson point process with intensity measure $s \Q$ and let $\X_n$, $n\in\N$, be a binomial point process of $n$ independent points distributed according to $\Q$. From now on $\operatorname{Conv}(\X)$ stands for the convex hull of a set $\X \subset \R^d$. The aim of this subsection is to establish rates of normal convergence for statistics of the random polytopes $\operatorname{Conv}(\P_s)$ and $\operatorname{Conv}(\X_n)$. We denote the number of $k$-faces of a polytope $P$ by $f_k(P)$, $k\in\{0,\hdots,d-1\}$, and its intrinsic volumes by $V_i(P)$, $i\in\{1,\hdots,d\}$.

\begin{theo}  \label{convexhull}\
For any $h\in \{f_0,\hdots,f_{d-1},V_1,\hdots,V_d\}$, there is a constant $C_h\in(0,\infty)$ also depending on $A$ such that
\be \label{CLTPo}
d_K\bigg( \frac{h(\operatorname{Conv}(\P_s))-\E h(\operatorname{Conv}(\P_s))}{\sqrt{\Var h(\operatorname{Conv}(\P_s))}},N\bigg)\leq C_h s^{- \frac{d-1} {2(d+1)}}, \quad s\geq 1,
\ee
and
\be \label{CLTBi}
d_K\bigg( \frac{h(\operatorname{Conv}(\X_n))-\E h(\operatorname{Conv}(\X_n))}{\sqrt{\Var h(\operatorname{Conv}(\X_n))}},N\bigg)\leq C_h n^{- \frac{d-1} {2(d+1)}}, \quad n \geq \max\{9,d+2\}.
\ee
\end{theo}

\noindent{\em Remarks.}  (i) {\em Previous work}. The asymptotic study of the statistics $h(\operatorname{Conv}(\P_s))$ and $h(\operatorname{Conv}(\X_n))$, $h\in \{f_0,\hdots,f_{d-1},V_1,\hdots,V_d\}$, has a long and rich history, starting with the seminal works \cite{RS,RS2}. The  breakthrough paper \cite{Re}, which relies on dependency graph methods and Voronoi cells, establishes rates of normal convergence for
Poisson input and  $h\in \{f_0,\hdots,f_{d-1}, V_d\}$ of the order
$s^{-\frac{d-1}{2(d+1)}}$  times some power of  $\ln(s)$ (see Theorems 1
and 2). Still in the setting  $h\in \{f_0,\hdots,f_{d-1}, V_d\}$, but with binomial input Theorem 1.2 and Theorem 1.3 of \cite{Vu}  provide the rates of convergence $n^{-1/(d + 1) + o(1) }$ for $d\geq 3$  and $n^{-1/6+o(1)}$ for $d=2$, which improved previous bounds in \cite{Re} for the binomial case, but is still weaker than  \eqref{CLTBi}. When $h\in \{f_0,\hdots,f_{d-1},V_1,\hdots,V_d\}$ and $A$ is the unit ball, Theorem 7.1 of \cite{CSY} gives a central limit theorem for $h(\operatorname{Conv}(\P_s))$, with  convergence rates involving extra logarithmic factors. We are unaware of central limit theorem results for intrinsic volume functionals over binomial input.

\noindent (ii) {\em Extensions. } Lower bounds for $\Var h(\operatorname{Conv}(\P_s))$ and $\Var h(\operatorname{Conv}(\X_n))$ are essential to
showing \eqref{CLTPo} and \eqref{CLTBi}.  We expect the order of these
bounds to be unchanged if  $\Q$ has a density bounded away from zero and
infinity. Consequently we anticipate that  Theorem \ref{convexhull}
remains valid in this context because all other arguments in our proof below also work for such a density.

\vskip.3cm

In the following we may assume without loss of generality that $\0$ is in the interior of $A$. The proof of Theorem \ref{convexhull} is divided into several lemmas and we prepare it by recalling some geometric facts and introducing some notation.

For a boundary point $z\in\partial A$ we denote by $T_z$ the tangent space parametrized by $\R^{d-1}$
in such a way  that $z$ is the origin. The boundary of $A$ in a neighborhood of $z$ may be identified with the graph of a function $f_z: T_z\to\R$. It may be deduced from \cite[Section 5]{Re} that there are constants $\underline{c}\in(0,1)$, $\overline{c}\in(1,\infty)$ and $r_0\in(0,\infty)$ such that uniformly for all $z\in \partial A$,
\begin{equation}\label{eqn:SandwichBoundary}
\underline{c}^2\|v\|^2 \leq f_z(v) \leq \overline{c}^2 \|v\|^2, \quad v\in T_z\cap B^{d-1}(\0,r_0),
\end{equation}
where we denote by $B^m(x,r)$ the closed ball with center $x\in\R^m$ and radius $r>0$ in $\R^m$, $m\in\N$.

For $u>0$ we define
$$
A_{-u}:=\{y\in A: \d(y,A^c)\leq u\},
$$
where $A^c:= \R^d \setminus A$. It follows from \eqref{eqn:SandwichBoundary} that there is a $\varrho>0$ such that all points $x\in A_{-3\varrho}$ have a unique projection $\Pi_{\partial A}(x)$ to $\partial A$. For $3\varrho\geq \overline{r}\geq r \geq \underline{r}\geq 0$ it also holds that
\begin{equation}\label{eqn:InclusionInnerBoundary}
\partial A_{-\overline{r}} \subset (\partial A_{-r} \oplus (\overline{r}-r) B^d(0,1)) \quad \text{ and } \quad \partial A_{-\underline{r}} \subset (\partial A_{-r}\oplus (r-\underline{r}) B^d(0,1)).
\end{equation}

We denote by $\d_{max}$ the metric
$$
\d_{max}(x,y):=\max\{\|x-y\|,\sqrt{|\d(x,A^c)-\d(y,A^c)|}\}, \quad x,y\in A,
$$
and define for $x\in A$ and $r>0$,
$$
B_{\d_{max}}(x,r):=\{y\in A: \d_{max}(x,y)\leq r\}.
$$
 The following lemma ensures that the space $(A,\mathcal{B}(A),\Q)$ and the metric $\d_{max}$ satisfy condition \eqref{eqn:SurfaceBall} for $x\in A_{- \varrho}$,  with $\gamma = d + 1$.

\begin{lemm}\label{lem:ConditionSpaceRandomPolytopes}
There is a constant $\kappa>0$ such that for all $x\in A_{-\varrho}$ and $r>0$
\begin{equation}\label{eqn:ConditionSpaceRandomPolytopes}
\limsup_{\varepsilon\to\infty} \frac{\Q(B_{\d_{max}}(x,r+\varepsilon))-\Q(B_{\d_{max}}(x,r))}{\varepsilon} \leq \kappa (d+1) r^d.
\end{equation}
\end{lemm}

\begin{proof}
Recall that, for $w>0$ and $D\subset \R^d$, the outer $w$-parallel set of $D$ is $\{x\in D^c: \d(x,D)\leq w\}$. For $u,v\in [0,\operatorname{diam}(A)/2]$, $A_{-(u+v)}\setminus A_{-u}=(A\setminus A_{-u})_{-v}$ and $A\setminus A_{-u}$ is convex. Consequently, $\Q((A\setminus A_{-u})_{-v})$ can be bounded by the volume of the outer $v$-parallel set of $A\setminus A_{-u}$, which can be bounded by the volume of the outer $v$-parallel set of $A$ so that
$$
\Q(A_{-(u+v)}\setminus A_{-u})\leq C_A v
$$
with some universal constant $C_A$ only depending on $A$. Since a similar inequality holds for $\Q(B^d(x,u+v)\setminus B^d(x,u))$, we see that the $\limsup$ in \eqref{eqn:ConditionSpaceRandomPolytopes} is bounded. For this reason it is sufficient to establish \eqref{eqn:ConditionSpaceRandomPolytopes} for small $r$.

We define for  $x\in A_{- \varrho}$ and $r \in (0,\varrho)$,
$$
U_{x,r}:=B_{\d_{max}}(x,r)\cap \{y\in A: \|x-y\|=r\}
$$
and
$$
V_{x,r}:=B_{\d_{max}}(x,r)\cap \{y\in A: |\d(x,A^{c})-\d(y,A^{c})|=r^2\}.
$$
It follows from \eqref{eqn:InclusionInnerBoundary} that
\begin{align*}
& \limsup_{\varepsilon\to\infty} \frac{\Q(B_{\d_{max}}(x,r+\varepsilon))-\Q(B_{\d_{max}}(x,r))}{\varepsilon} \\
& \leq \limsup_{\varepsilon\to\infty} \frac{\Q(U_{x,r} \oplus |\varepsilon| B^d(0,1))+\Q(V_{x,r} \oplus (2r|\varepsilon|+\varepsilon^2) B^d(0,1))}{|\varepsilon|}\\
& \leq 2 \mathcal{H}^{d-1}(U_{x,r})+4r\mathcal{H}^{d-1}(V_{x,r}).
\end{align*}
For $r$ sufficiently small, we obtain sub- and supersets for $A_{-(\d(x,A^c)-r^2)}\cap B^d(x,r)$ and $A_{-(\d(x,A^c)+r^2)}\cap B^d(x,r)$ by taking the inner parallel sets with respect to the paraboloids given in \eqref{eqn:SandwichBoundary}. Consequently, $U_{x,r}$ is contained in a strip whose Euclidean thickness is of the order $r^2$. This implies that $\mathcal{H}^{d-1}(U_{x,r})\leq c_A r^d$ for all $r>0$ with some constant $c_A \in(0,\infty)$ only depending on $A$. Since $V_{x,r}$ is the union of the intersection of the boundaries of the convex sets $A_{-(\d(x,A^c)+r^2)}$ and $A_{-(\d(x,A^c)-r^2)}$ with $B^d(x,r)$, we have that $\mathcal{H}^{d-1}(V_{x,r})\leq 2 d\kappa_d r^{d-1}$, which completes the proof.
\end{proof}

We let $u_x:=(\Pi_{\partial A}(x)-x)/\|\Pi_{\partial A}(x)-x\|$ for $x\in \operatorname{int}(A)$, whereas for $x\in \partial A$ we let $u_x$ be the outer unit normal at $x$.
For $x\in A$ and $r>0$ we define the hyperplanes
$$
H_x:=\{y\in \R^d: \langle u_x, y\rangle =\langle u_x, x\rangle \}
$$
and the parametrized family of sets
\begin{align*}
 A_{x,r} :=  \begin{cases}
\operatorname{Conv}((H_x\cap B^d(x,r/\overline{c}))\cup\{x+ r^2 u_x\}) & \text{ if } r\leq \sqrt{\d(x,A^c)}\\
A\setminus \operatorname{Conv}(( A \setminus B^d(x,r/\underline{c}))\cup\{x\}) & \text{ if } r>\sqrt{\d(x,A^c)}.
 \end{cases}
\end{align*}
When $x\in A_{-\varrho}$ and $r > \sqrt{\d(x,A^c)}$ is sufficiently small, we note that $x$ is an extreme point of $A\setminus A_{x,r}$. The sets $A_{x,r}$ have the following important properties.

\begin{lemm}\label{lem:PropertiesAxr}
\begin{itemize}
\item [(a)] There are constants $C_\Q,c_\Q\in (0,\infty)$ such that
$$
C_\Q r^{d+1} \geq \Q(A_{x,r})\geq c_\Q r^{d+1}, \quad x\in A_{-\varrho}, r\in[0,1].
$$
\item [(b)] There is a constant $c_{max}\in(0,\infty)$ such that $A_{x,r}\subset B_{\d_{max}}(x,c_{max}r)$ for any $r>0$ and $x\in A_{-\tilde{\varrho}}$ with $\tilde{\varrho}:=\min\{1/(4\overline{c}^2),\varrho\}$.
\end{itemize}
\end{lemm}

\begin{proof}  We denote the epigraphs of $v\mapsto\underline{c}^2 \|v\|^2$ and $v\mapsto\overline{c}^2 \|v\|^2$ by $\underline{P}_z$ and $\overline{P}_z$.
For $r\leq\sqrt{\d(x,A^c)}$ we have $\Q(A_{x,r})=\kappa_{d-1} r^{d+1}/(d\overline{c}^{d-1})$.  For $x\in A$ and $r>\sqrt{\d(x,A^c)}$ let $z:=\Pi _{\partial A}(x)$. Since
$$
\operatorname{Conv}( (A  \setminus B^d(x,r/\underline{c})) \cup \{x\} ) \subset \operatorname{Conv}(  (A  \setminus B^d(z,r/\underline{c})) \cup \{z\} ),
$$
it follows that $A_{x,r} \supset A_{z,r}$.  Additionally
$$
A_{z,r}\supset \overline{P}_z \setminus\operatorname{Conv}(\{z\}\cup (\underline{P}_{z}\setminus B^d(z,r/\underline{c}))).
$$
A longer computation shows that the volume of the set on the right-hand side can be bounded  below by a non-negative scalar multiple of  $r^{d+1}$. The upper bound in part (a) can be proven similarly.

To prove part (b) it suffices  to consider only the situation $r\in[0,1]$. It follows immediately from the definition of $A_{x,r}$ that $A_{x,r}\subset B^d(x,r/\underline{c})$ for $r\in[0,1]$. For $x\in A_{-\varrho}$ with $\d(x,A^c)\leq 1/(4\overline{c}^2)$, $r\leq \sqrt{\d(x,A^c)}$ and $y\in A_{x,r}$, we obtain by a direct but longer computation that
$$
\d(x,A^c) \geq \d(y,A^c) \geq \d(x,A^c)-4\overline{c}^2 r^2.
$$
On the other hand, for $r>\sqrt{\d(x,A^c)}$ and $y\in A_{x,r}$, we have with $z=\Pi_{\partial A}(x)$ that
\begin{align*}
\d(y,A^c) & \leq \sup_{v\in\partial A\cap B^d(x,r/\underline{c})} \d(v,H_z) \leq \sup_{v\in\overline{P}_z\cap B^d(z,r/\underline{c}+\d(x,A^c))} \d(v,H_z) \\
& \leq \sup_{v\in\overline{P}_z\cap B^d(z,(1/\underline{c}+1)r)} \d(v,H_z) \leq \overline{c}^2 (1/\underline{c}+1)^2 r^2.
\end{align*}
This implies that $A_{x,r}\subset \{y\in B^d(0,1): \sqrt{|\d(x,A^c)-\d(y,A^c) |}\leq \overline{c} (1/\underline{c}+2)r\}$, which completes the proof of part (b).
\end{proof}

For $k\in\{0,\hdots,d-1\}$ and $\X\in\mathbf{N}$ let $\mathcal{F}_k(\operatorname{Conv}(\X))$ be the set of $k$-dimensional faces of $\operatorname{Conv}(\X)$. To cast $f_k(\operatorname{Conv}(\X))$ in the form of \eqref{Poissonstat} and \eqref{Binomialstat}, we define
$$
\xi_k(x,\X):= \frac{1}{k+1} \sum_{F\in\mathcal{F}_k(\operatorname{Conv}(\X))} \mathbf{1}_{\{x\in F\}}, \quad x\in\X.
$$
Note that $f_k(\operatorname{Conv}(\X)) = \sum_{x \in \X} \xi_k(x, \X)$.

To cast the intrinsic volumes $V_j(\operatorname{Conv}(\X))$, $j \in\{1,\hdots,d-1\}$, in the form of \eqref{Poissonstat} and \eqref{Binomialstat}, we need some more  notation. Given the convex set $A$ and a linear subspace  $E$, denote by $A | E$ the orthogonal projection of $A$ onto $E$. For $x\in \mathbb{R}^{d}\setminus \{0\}$, let $L(x)$ the line spanned by $x$. Given a line $N\subset \mathbb{R}^{d}$ through the origin, and for $1\leq j\leq d,$ let $G(N,j)$ be the set of $j$-dimensional linear subspaces of $\mathbb{R}^{d}$ containing $N$. Let then $\nu_j^N(\cdot )$ be the Haar probability measure on $G(N,j)$.
Let $M\subset A$ be convex.  For  $j\in \{0,\dots ,d-1\}$,  $x\in \mathbb{R}^{d}
\setminus \{\0\}$, and $L\in G(L(x),j)$ define
\begin{align*}
f^{L}(x):=\mathbf{1}_{\{x\in (A | L)\setminus (M | L)\}}
\end{align*}
and, as in \cite{CSY}, define the \emph{projection avoidance function}  $\theta^{A,M}_{j}: \ \mathbb{R}^{d}\setminus \{\0\} \mapsto [0,1]$ by
\begin{align*}
\theta^{A,M}_{j}(x):= \int_{G(L(x),j)} f^{L}(x) \, \nu^{L(x)}_{j}(\dint L).
\end{align*}
The following result generalizes \cite[(2.7)]{CSY} to non-spherical compact sets, with arguments similar to Lemma A1 from \cite{GT}.  The proof is in the appendix.

\begin{lemm} \label{lm:intrinsic-volumes-representation}
Let $M\subset A$ be a convex subset of $\R^d$.  For all $j\in \{0,\dots ,d-1\}$
there is a constant $\kappa _{d,j}$ depending on $d,j$ such that
\begin{align}  \label{eq:bound-intrinsic}
V_{j}(A)-V_{j}(M)=\kappa _{d,j}\int_{A\setminus M}\theta _{j}^{A,M}(x)\|x\|^{-(d-j)} \, \dint x .
\end{align}
 \end{lemm}

For $\X\in\mathbf{N}$ and $F\in \mathcal{F}_{d-1}(\operatorname{Conv}(\X))$ put ${\text{cone}}(F):= \{ry: \ y \in F, r > 0 \}$. Define for $j \in \{1,...,d-1 \}$
\begin{align*}
\xi _{j,s}(x,\X)= \frac{s\kappa_{d,j}}{d}\sum_{F\in \mathcal{F}_{d-1}(\operatorname{Conv}(\X))} \mathbf{1}_{\{x\in F\}}\int_{\text{\rm{Cone}}(F )\cap (A\setminus  \operatorname{Conv}(\X))}\|x\|^{-(d-j)}\theta _{j}^{A,\operatorname{Conv}(\X)}(x) \, \dint x
\end{align*}
for $x\in\X$, $s\geq 1$. Lemma \ref{lm:intrinsic-volumes-representation} yields
\be \label{IVscaling}
s (V_{j}(A)-V_j(\operatorname{Conv}(\X))) =  \sum_{x \in \X} \xi_{j,s}(x, \X)
\ee
if $\0$ is in the interior of $\operatorname{Conv}(\X)$ and if all points of $\X$ are in general position.
For $x\in\X$ and $s\geq 1$ define
$$
\xi_{d,s}(x,\X):= \frac{s}{d}\sum_{F\in \mathcal{F}_{d-1}(\operatorname{Conv}(\X))} \mathbf{1}_{\{x\in F\}}\int_{ \text{\rm{Cone}}(F)\cap (A\setminus  \operatorname{Conv}(\X) ) } \dint x.
$$
If $\0$ is in the interior of $\operatorname{Conv}(\X)$ and all points of $\X$ are in general position, we have as well
$$
s V_d(A \setminus \operatorname{Conv}(\X)) =  \sum_{x \in \X} \xi_{d,s}(x, \X).
$$
The definitions of the scores and $\|\theta_j^{A,\operatorname{Conv}(\X)}\|_\infty\leq 1$ show that for $\X\in\mathbf{N}$, $x\in\X$, $s\geq 1$ and $j \in \{0,..., d-1\}$
\begin{align}
\label{eq:score-inequality}
\xi_{j,s}(x,\X)\leq \kappa_{d,j} r(\operatorname{Conv}(\X))^{-(d-j)}\xi_{d,s}(x,\X),
\end{align}
where $r(\operatorname{Conv}(\X))$ is the radius of the largest ball centered at $\0$ and contained in $\operatorname{Conv}(\X)$.

Since $\0\in \text{\rm{int}}(A)$, we can choose $\rho _{0}\in(0,\tilde{\varrho})$ such that $B(\0,2\rho _{0})\subset A$. For a score $\xi$ we denote by $\tilde{\xi}$ the modified score
$$
\tilde{\xi}(x,\X):=\mathbf{1}_{\{x\in A_{-\rho_0}\}} \xi(x,(\X\cap A_{-\rho_0})\cup\{\0\})
$$
for $\X\in\mathbf{N}$ and $x\in\X$. Our strategy of proof for Theorem \ref{convexhull} is to apply in a first step Corollary \ref{coroLSY} in connection with Remark (v) after Theorem \ref{2coroLSY} to these modified scores, putting $\XX:=A$ and $\tilde{\XX} := A_{-\rho_0}$ and $K$ set to $\partial A$. Thereafter we show that the result remains true without truncating and without adding the origin as an additional point.

For a score $\xi$ and $\X\in\mathbf{N}$ we define
$$
S_{\xi}(\X):=\sum_{x\in\X} \xi(x,\X).
$$

\begin{lemm}\label{lem:ApproximationConvexHull}
For any $\xi_s\in\{\xi_0,\hdots,\xi_{d-1},\xi_{1,s},\hdots,\xi_{d,s}\}$ there are constants $C_0,c_0\in(0,\infty)$ such that
 \begin{align*}
& \max\{ \Prob(S_{\xi_s}(\P_s) \neq S_{\tilde{\xi}_s}(\P_s)),\Prob(B^{d}(\0,\rho _{0})\not\subset \operatorname{Conv}(\P _{s})), \\
&\hspace{4cm}|\E S_{\xi_s}(\P_s) - \E S_{\tilde{\xi}_s}(\P_s)|, |\Var S_{\xi_s}(\P_s) - \Var S_{\tilde{\xi}_s}(\P_s)|\} \\
& \leq C_0 \exp(-c_0 s)
\end{align*}
for $s\geq 1$ and
\begin{align*}
 & \max\{ \Prob(S_{\xi_n}(\X_n) \neq S_{\tilde{\xi}_n}(\X_n)),\Prob (B^{d}(\0,\rho _{0})\not\subset \operatorname{Conv}(\X _{n})),\\
 &\hspace{4cm} |\E S_{\xi_n}(\X_n) - \E S_{\tilde{\xi}_n}(\X_n)|, |\Var S_{\xi_n}(\X_n) - \Var S_{\tilde{\xi}_n}(\X_n)|\}\\
 & \leq C_0 \exp(-c_0 n)
\end{align*}
for $n\geq 1$.
\end{lemm}

\begin{proof}
One can choose sets $A_1,\hdots,A_m\subset \{x\in A: \d(x,A^c)\leq \rho _{0}\}$ with non-empty interior such that, for any $\X\in\mathbf{N}$ with $A_i\cap\X\neq\emptyset$, $i\in\{1,\hdots,m\}$,
$$
\operatorname{conv}(\X)\supset\{x\in A: \d(x,A^c)> \rho _{0}\}.
$$
Using $B(\0,2\rho _{0})\subset A$, this inclusion yields $B(\0,\rho _{0})\subset \operatorname{Conv}(\X )$.
The event $S_{\tilde{\xi}_s}(\X)\neq S_{\xi_s}(\X)$ is also a subset of the event $A_i\cap\X=\emptyset$ for some $i\in\{1,\hdots,m\}$. These observations prove the probability bounds.

The generous upper bounds
$$
\max_{ k \in \{0,...,d-1 \}}  f_k(\conv(\X)) \leq |\X|^{d-1} \quad \text{and} \quad \max_{i\in\{1,\hdots,d\}}V_i(\conv(\X))\leq \max_{i\in\{1,\hdots,d\}} V_i(A)
$$
lead to $|S_{\tilde{\xi}_s}(\X)-S_{\xi_s}(\X)|\leq C_{d} s |\X|^{d}$ for some universal constant $C_{d}\in(0,\infty)$. Together with H\"older's inequality and the above probability bounds this yields the asserted expectation and variance bounds.
\end{proof}

The results of \cite{Re} show that for $\xi_s\in\{\xi_0,\hdots,\xi_{d-1},\xi_{d,s}\}$ one has
\begin{equation}\label{eqn:VarianceLowerBoundConvexHull}
\Var S_{\xi_s}(\P_s) = \Theta(s^{\frac{d-1}{d+1}}) \quad \text{and} \quad \Var S_{\xi_n}(\X_n) = \Theta(n^{\frac{d-1}{d+1}}).
\end{equation}
For $\xi_s\in\{\xi_{1,s},\hdots,\xi_{d-1,s}\}$ and taking into account scaling \eqref{IVscaling}, we know from Corollary 7.1 of \cite{CSY} and from Theorem 2 of \cite{BFV} that
\begin{equation}\label{eqn:VarianceLowerBoundConvexHulla}
\Var S_{\xi_s}(\P_s) = \Theta(s^{ \frac{d -1 }{d+1}}) \quad \text{and} \quad \Var S_{\xi_n}(\X_n) = \Theta(n^{\frac{d -1}{d+1}}).
\end{equation}
Hence, Lemma \ref{lem:ApproximationConvexHull} implies that for $\xi_s\in\{\xi_0,\hdots,\xi_{d-1},\xi_{1,s},\hdots,\xi_{d,s}\}$
\begin{equation}\label{eqn:VarianceLowerBoundConvexHullTilde}
\Var S_{\tilde{\xi}_s}(\P_s) =  \Theta(s^{\frac{d-1}{d+1}}) \quad \text{and} \quad \Var S_{\tilde{\xi}_n}(\X_n) = \Theta(n^{\frac{d-1}{d+1}}).
\end{equation}

For a point $x\in A$ let $\tilde{H}_{x,1},\hdots,\tilde{H}_{x,2^{d-1}}$ be a decomposition of $H_x$ into solid orthants having $x$ in common and let $H_{x,i}:=\tilde{H}_{x,i}+\operatorname{Span}(x)$ for $i\in\{1,\hdots,2^{d-1}\}$.

\begin{lemm}\label{lem:Axr}
Let  $\tilde{\xi}_s\in\{\tilde{\xi}_0,\hdots,\tilde{\xi}_{d-1},\tilde{\xi}_{1,s},\hdots,\tilde{\xi}_{d,s}\}$ and let $x\in A$, $r>0$ and $\X\in\mathbf{N}$ be such that $x\in\X$ and $\X\cap A_{-\rho_0}\cap A_{x,r}\cap H_{x,i}\neq \emptyset$ for $i\in\{1,\hdots,2^{d-1}\}$. Then, $\tilde{\xi}_s(x,\X)$ is completely determined by $\mathcal{X}\cap A_{-\rho_0}\cap A_{x,r}$, i.e., thus by $\X\cap A_{-\rho_0}\cap B_{\d_{max}}(x,c_{max}r)$ with $c_{max}$ as in Lemma \ref{lem:PropertiesAxr}.
\end{lemm}
\begin{proof}
Let $\d(x,A^c)\leq \rho_0 \leq 1/(8\overline{c}^2)$ since, otherwise, the assertion is trivial. By assumption there are $y_1,\hdots,y_{2^{d-1}}$ such that $y_i\in \X\cap A_{-\rho_0}\cap A_{x,r}\cap H_{x,i}$ for $i\in\{1,\hdots,2^{d-1}\}$. Let $C_{x,y_1,\hdots,y_{2^{d-1}}}$ be the cone with apex $x$ generated by the points $\0,y_1,\hdots,y_{2^{d-1}}$. If $C_{x,y_1,\hdots,y_{2^{d-1}}}=\R^d$, we have $x\in\operatorname{Conv}(\{\0,y_1,\hdots,y_{2^{d-1}}\})$, whence $\tilde{\xi}_s(x,\X)=0$. If $C_{x,y_1,\hdots,y_{2^{d-1}}}\neq\R^d$ (this implies that $r>\sqrt{\d(x,A^c)}$), no point in the interior of $C_{x,y_1,\hdots,y_{2^{d-1}}}$ can be connected with $x$ by an edge. Since $\operatorname{Conv}((A\setminus B^d(x,r/\underline{c}))\cup\{x\})\subset C_{x,y_1,\hdots,y_{2^{d-1}}}$, all points in $A\setminus A_{x,r}$ are irrelevant for the facial structure at $x$. Consequently  the scores
$\tilde{\xi}_s$ are completely determined by $\X\cap A_{-\rho_0}\cap A_{x,r}$. In view of  Lemma \ref{lem:PropertiesAxr}(b) we have $A_{x,r}\subset B_{\d_{max}}(x,c_{max}r)$ so the same is true for $\X\cap A_{-\rho_0}\cap B_{\d_{max}}(x,c_{max}r)$.
\end{proof}

We define the map $R: A \times \mathbf{N} \to \R$ which sends $(x, \X)$ to
\begin{align*}
& R(x,\mathcal{X}\cup\{x\}) :=\\
& \begin{cases} c_{max} \inf\{r\geq 0: \X\cap A_{-\rho_0}\cap A_{x,r}\cap H_{x,i}\neq \emptyset \text{ for } i\in\{1,\hdots,2^{d-1}\} \} & \text{ if } x\in A_{-\rho_0} \\
0 & \text{ if } x\notin A_{-\rho_0}. \end{cases}
\end{align*}
The next lemma shows that all $\tilde{\xi}_s \in \{\tilde{\xi}_0,\hdots,\tilde{\xi}_{d-1},\tilde{\xi}_{1,s},\hdots,\tilde{\xi}_{d,s}\}$ satisfy \eqref{eqn:expstabPoisson} and \eqref{eqn:expstabBinomial} with $\alpha_{stab} = d + 1$.

\begin{lemm}\label{lem:ExstabRandomPolytopes}
$R$ is a radius of stabilization for any $\tilde{\xi}_s\in\{\tilde{\xi}_0,\hdots,\tilde{\xi}_{d-1},\tilde{\xi}_{1,s},\hdots,\tilde{\xi}_{d,s}\}$  and there are constants $C,c\in (0,\infty)$ such that for $r\geq 0$, $x\in A$
$$
\Prob(R(x,\P_s\cup\{x\})\geq r) \leq C \exp(-c s r^{d+1}), \quad s\geq 1,
$$
whereas
$$
\Prob(R(x,\X_{n-8}\cup\{x\})\geq r) \leq C \exp(-c n r^{d+1}), \quad n \geq 9.
$$
\end{lemm}

\begin{proof}
It follows from Lemma \ref{lem:Axr} that $R$ is a radius of stabilization. It is sufficient to establish the desired inequalities for $x\in A_{-\rho_0}$ and $r\in [0,r_0]$ for some $r_0>0$. We see that
$$
\Prob(R(x,\P_s\cup\{x\})\geq r) \leq \Prob(\exists i\in\{1,\hdots,2^{d-1}\}: \P_s\cap A_{-\rho_0}\cap A_{x,r/c_{max}}\cap H_{x,i} = \emptyset).
$$
Choosing $r_0$ small enough so that $A_{x,r/c_{max}}\cap A_{-\rho_0}=A_{x,r/c_{max}}$ and noting that by the same arguments as in the proof of Lemma \ref{lem:PropertiesAxr}(a) $\Q(A_{x,r/c_{max}}\cap H_i)\geq c_{\Q} r^{d+1}/(2^{d-1}c_{max}^{d+1})$, we obtain that
$$
\Prob(R(x,\P_s\cup\{x\})\geq r) \leq 2^{d-1} \exp(-s c_{\Q} r^{d+1}/(2^{d-1}c_{max}^{d+1})).
$$
The proof for the binomial case goes similarly.
\end{proof}

The next lemma shows that all $\tilde{\xi}_s \in \{\tilde{\xi}_0,\hdots,\tilde{\xi}_{d-1},\tilde{\xi}_{1,s},\hdots,\tilde{\xi}_{d,s}\}$ satisfy
\eqref{eqn:expfastPoisson} and \eqref{eqn:expfastBinomial} with $\alpha_{\partial A} =d+1$.

\begin{lemm}\label{lem:BoundaryRandomPolytopes}
For any $\tilde{\xi}_s\in\{\tilde{\xi}_0,\hdots,\tilde{\xi}_{d-1},\tilde{\xi}_{1,s},\hdots,\tilde{\xi}_{d,s}\}$ there are constants $C_{b},c_b\in(0,\infty)$ such that
for $x\in A$, $\A\subset A$ with $|\A|\leq 7$
$$
\Prob(\tilde{\xi}_s(x,\P_s\cup\{x\}\cup \A)\neq 0) \leq C_b \exp(-c_b s \d_{max}(x,A^c)^{d+1}), \quad s \geq 1,
$$
whereas
$$
\Prob(\tilde{\xi}_n(x,\X_{n-8}\cup\{x\}\cup A)\neq 0) \leq  C_b \exp(-c_b n \d_{max}(x,A^c)^{d+1} ), \quad n \geq 9.
$$
\end{lemm}
\begin{proof}
For $x\in A_{-\rho_0}$, $\X\in\mathbf{N}$ and $\A\subset A$ with $|\A|\leq 7$ we have that $\tilde{\xi}_{s}(x,\X\cup\{x\}\cup\A)=0$ if $R(x,\X\cup\{x\})\leq \sqrt{\d(x,A^c)}=\d_{max}(x,A^c)$. Thus, the assertions follow from Lemma \ref{lem:ExstabRandomPolytopes}.
\end{proof}

\begin{lemm}\label{lem:MomentsRandomPolytopes}
For any $q\geq 1$ and $\tilde{\xi}_s\in\{\tilde{\xi}_0,\hdots,\tilde{\xi}_{d-1},\tilde{\xi}_{1,s},\hdots,\tilde{\xi}_{d,s}\}$ there is a constant $C_q\in(0,\infty)$ such that for all $\A\subset A$ with $|\A|\leq 7$,
$$
\sup_{s\geq 1}\sup_{x\in A} \E |\tilde{\xi}_s(x,\P_s\cup\{x\}\cup\A)|^q \leq C_q
\quad
\text{and}
\quad
\sup_{n\in\N, n\geq 9}\sup_{x\in A} \E |\tilde{\xi}_n(x,\X_{n-8}\cup\{x\}\cup\A)|^q \leq C_q.
$$
\end{lemm}

\begin{proof}
The assertion for $\tilde{\xi}_0,\hdots,\tilde{\xi}_{d-1}$ can be shown similarly as in Lemma 7.1 of \cite{CSY}. Similar considerations as in the proof of Lemma \ref{lem:Axr} show that
$$
\tilde{\xi}_{d,s}(x,\P_s\cup\{x\}\cup\A) \leq s \Q(A_{x,R(x,\P_s\cup\{x\})}).
$$
Combining this with Lemma \ref{lem:PropertiesAxr} and Lemma \ref{lem:ExstabRandomPolytopes} leads to the inequality for $\tilde{\xi}_{d,s}$ in the Poisson case, which can be proven similarly in the binomial case. For the intrinsic volumes $\tilde{\xi}_{j,s}, j \in \{0,...,d - 1\},$ the bound  \eqref{eq:score-inequality} shows that the $q$-th moment of $\tilde \xi _{j,s}$ is bounded by  a constant multiple of the $q$-th moment of $\tilde \xi _{d,s}$ plus  $s^{q} \Prob(B(\0,\rho _{0})\not\subset \operatorname{Conv}(\X_{s}))$, which by  Lemma \ref{lem:ApproximationConvexHull} is bounded by $s^{q} C_{0}\exp(-c_{0}s)$. This completes the proof.
\end{proof}

\begin{lemm}\label{lem:CLTApproximationConvexHull}
For any $\tilde{\xi}_s\in\{\tilde{\xi}_0,\hdots,\tilde{\xi}_{d-1},\tilde{\xi}_{1,s},\hdots,\tilde{\xi}_{d,s}\}$ there is a constant $\tilde{C}\in(0,\infty)$ such that
$$
d_K\bigg(\frac{S_{\tilde{\xi}_s}(\P_s)-\E S_{\tilde{\xi}_s}(\P_s)}{\sqrt{\Var S_{\tilde{\xi}_s}(\P_s)}},N\bigg) \leq \tilde{C} s^{-\frac{d-1}{2(d+1)}}, \quad s\geq 1,
$$
and
$$
d_K\bigg(\frac{S_{\tilde{\xi}_n}(\X_n)-\E S_{\tilde{\xi}_n}(\X_n)}{\sqrt{\Var S_{\tilde{\xi}_n}(\X_n)}},N\bigg) \leq \tilde{C} n^{-\frac{d-1}{2(d+1)}}, \quad n\geq 9.
$$
\end{lemm}
\begin{proof}
By Lemmas \ref{lem:ConditionSpaceRandomPolytopes},  \ref{lem:ExstabRandomPolytopes},  \ref{lem:BoundaryRandomPolytopes}, and \ref{lem:MomentsRandomPolytopes} all conditions of Corollary \ref{coroLSY} in connection with Remark (v) after Theorem \ref{2coroLSY} are satisfied with $\XX:=A$, $\tilde{\XX}:=A_{-\rho_0}$ and $K:= \partial A$.  Note that
 $I_{\partial A, s} = O(s^{(d-1)/(d + 1)})$, which completes the proof.
\end{proof}

\begin{proof}[Proof of Theorem \ref{convexhull}]
For any pair $(X, \tilde{X})$ of square integrable random variables satisfying $\Var X,\Var \tilde{X}>0$, a straightforward computation shows that
\begin{align*}
& d_K\bigg(\frac{X-\E X}{\sqrt{\Var X}},N\bigg) \\
& \leq d_K\bigg(\frac{\tilde{X}-\E X }{\sqrt{\Var X}},N\bigg) + \Prob(X\neq \tilde{X}) \allowdisplaybreaks\\
& = d_K\bigg(\frac{\tilde{X}-\E\tilde{X}}{\sqrt{\Var \tilde{X}}},N\bigg(  \frac{\E X - \E \tilde{X} }{\sqrt{\Var \tilde{X} }}, \frac{{\Var X }}  {{\Var \tilde{X} }} \bigg) \bigg)   + \Prob(X\neq \tilde{X}) \allowdisplaybreaks\\
& \leq d_K\bigg(\frac{\tilde{X}-\E\tilde{X}}{\sqrt{\Var \tilde{X}}},N\bigg) + d_K\bigg( N, N\bigg(  \frac{\E X - \E \tilde{X} }{\sqrt{\Var \tilde{X} }}, \frac{{\Var X }}  {{\Var \tilde{X} }} \bigg) \bigg) +
 \Prob(X\neq \tilde{X}) \allowdisplaybreaks\\
& \leq d_K\bigg(\frac{\tilde{X}-\E\tilde{X}}{\sqrt{\Var \tilde{X}}},N\bigg) +  \frac{|\E X - \E\tilde{X}|}{\sqrt{\Var \tilde{X}}} +  C {\bigg| \frac{{\Var X}}{{\Var \tilde{X}}}-1\bigg|}  +
 \Prob(X\neq \tilde{X}),
\end{align*}
where $N(\mu,\sigma^2)$ stands for a Gaussian random variable with mean $\mu$ and variance $\sigma^2$ and $C\in(0,\infty)$ is some universal constant.
Applying this to the pairs $(X, \tilde{X}) := (S_{\xi_s}(\P_s), S_{\tilde{\xi}_s}(\P_s))$ and $(X, \tilde{X}) := (S_{\xi_n}(\X_n), S_{\tilde{\xi}_n}(\X_n))$, respectively, together with Lemma \ref{lem:ApproximationConvexHull}, Lemma \ref{lem:CLTApproximationConvexHull}, \eqref{eqn:VarianceLowerBoundConvexHull},
\eqref{eqn:VarianceLowerBoundConvexHulla}, and \eqref{eqn:VarianceLowerBoundConvexHullTilde} completes the proof.
\end{proof}

\subsection{Clique counts in generalized random geometric graphs} \label{NNGsection}

Let $(\XX, {\cal F}, \Q)$ be equipped with a semi-metric $\d$ such that \eqref{eqn:SurfaceBall} is satisfied for some
$\gamma$ and $\kappa$. Moreover, let $\MM = [0, \infty)$ be equipped with the Borel sigma algebra $\mathcal{F}_{\MM}:= {\cal B}([0,\infty))$ and a probability measure $\Q_{\MM}$ on $([0, \infty), {\cal B}([0,\infty)))$. By $\widehat{\Q}$ we denote the the product measure of $\Q$ and $\Q_{\MM}$. In the following let $\P_s$ be a marked Poisson point process with intensity measure ${s\widehat{\Q}}$, $s\geq 1$, and let $\X_n$ be a marked binomial point process of $n\in\N$ points distributed according to $\widehat{\Q}$.

Given $\X \in\mathbf{N}$, recall that $\mathbf{N}$ is the set of point configurations in $\widehat{\XX}$, and a scale parameter $\beta \in (0, \infty)$, consider the graph $G(\X, \beta)$ on $\X$ with $(x_1, m_{x_1})\in\X$ and $(x_2, m_{x_2})\in\X$ joined with an edge iff $\d(x_1, x_2) \leq \beta \min( m_{x_1}, m_{x_2} )$.  When $m_x = 1$ for all $x \in \X$, we obtain the familiar geometric graph with parameter $\beta$. Alternatively, we could use the connection rule that $(x_1, m_{x_1})$ and $(x_2, m_{x_2})$ are joined with an edge iff $\d(x_1, x_2) \leq \beta \max( m_{x_1}, m_{x_2} )$. A scale-free random graph based on this connection rule with an underlying marked Poisson point process is studied in \cite{Hi}. The number of cliques of order $k + 1$ in  $G(\X, \beta)$, here denoted ${\cal C}_k(\X, \beta),$
is a well-studied statistic in random geometric graphs. Recall that $k+1$ vertices of a graph form a clique of order $k+1$ if each pair of them is connected by an edge.

The clique count ${\cal C}_k(\X, \beta)$ is also a central statistic  in topological data analysis.
Consider the simplicial complex ${\cal R}^\beta(\X)$  whose $k$-simplices correspond to unordered $(k + 1)$-tuples of points of $\X$
such that any constituent pair  of points $(x_1, m_{x_1})$ and $(x_2, m_{x_2})$ satisfies  $\d(x_1, x_2) \leq \beta \min( m_{x_1}, m_{x_2} )$.
When $m_x = 1$ for all $x \in \X$ then  ${\cal R}^\beta(\X)$
coincides with the  Vietoris-Rips complex with scale parameter $\beta$ and ${\cal C}_k(\X, \beta)$ counts the
number of $k$-simplices in ${\cal R}^\beta(\X)$.

When $\Q$ is the uniform measure on a compact set $\XX \subset \R^d$ with $\Vol(\XX)>0$ and $\gamma = d$, the ungainly quantity ${\cal C}_k(\P_s, \beta s^{-1/\gamma})$ studied below is equivalent to the more natural clique count ${\cal C}_k(\widetilde{\P}_1\cap s^{1/d} \XX , \beta)$, where $\widetilde{\P}_1$ is a rate one stationary Poisson point process in $\R^d$ and $\widetilde{\P}_1\cap s^{1/d} \XX$ is its restriction to $s^{1/d} \XX$.

\begin{theo}  \label{cliquecount} Let $k \in \N$ and $\beta \in (0, \infty)$ and assume there are
constants $c_1 \in (0, \infty)$ and $c_2 \in (0, \infty)$ such that
\be
\label{exptail}
\Prob(M_x \geq r) \leq c_1 \exp( - \frac{r^{c_2}} {c_1} ),  \quad x \in \XX, \ r \in (0, \infty).
\ee
If $\inf_{s\geq 1} \Var {\cal C}_k( \P_s, \beta s^{-1/\gamma} )/s > 0$, then there is a constant $\tilde{C}\in(0,\infty)$ such that
\be \label{ccPoisson}
d_K \left( \frac{ {\cal C}_k( \P_s, \beta s^{-1/\gamma} ) - \E {\cal C}_k( \P_s, \beta s^{-1/\gamma} ) } {\sqrt{ \Var {\cal C}_k( \P_s, \beta s^{-1/\gamma} ) } } , N \right)
 \leq   \frac{ \tilde{C} } { \sqrt{s} },  \quad s \geq 1.
\ee
Likewise if $\inf_{n\geq 9}\Var {\cal C}_k( \X_n, \beta n^{-1/\gamma} )/n> 0$, then there is a constant $\tilde{C}\in(0,\infty)$ such that
\be \label{ccBinomial}
d_K \left( \frac{ {\cal C}_k( \X_n, \beta n^{-1/\gamma} ) - \E {\cal C}_k( \X_n, \beta n^{-1/\gamma} ) } {\sqrt{ \Var{\cal C}_k( \X_n, \beta n^{-1/\gamma} ) } } , N \right)
\leq   \frac{ \tilde{C} } { \sqrt{n} },  \quad n \geq 9.
\ee
\end{theo}

\noindent{\em Remarks.}
(i) When $\XX$ is a full-dimensional subset of $\R^d$  and when $M_x \equiv 1$ for all $x \in \XX$, i.e., $\Q_\MM$ is the Dirac measure concentrated at one, a central limit theorem for the Poisson case is shown in \cite[Theorem 3.10]{Pbook}. Although the result in \cite{Pbook} is non-quantitative, the method of proof should yield a rate of convergence for the Kolmogorov distance.
Rates of normal convergence with respect to the Wasserstein distance $d_W$ are given in \cite{DFR}.

\noindent (ii) The contributions of this theorem are three-fold. First,
$\XX$ may be an arbitrary metric space, not necessarily a subset of $\R^d$.  Second, the graphs $G(\P_s, \beta s^{-1/\gamma})$ and $G(\X_n, \beta n^{-1/\gamma})$ are more general than the standard random geometric graph, as they consist of edges having arbitrary (exponentially decaying) lengths. Third, by applying our general findings we obtain presumably optimal rates of convergence for the Poisson and the binomial case at the same time.

\noindent (iii) The random variable ${\cal C}_k( \P_s, \beta s^{-1/\gamma} )$ is a so-called Poisson U-statistic. Bounds for the normal approximation of such random variables were deduced, for example, in \cite{ReSc} and \cite{LRP} for the Wasserstein distance and in \cite{SchulteKolmogorov} and \cite{EicTha} for the Kolmogorov distance. These results should also yield bounds similar to those in \eqref{ccPoisson}.

\noindent (iv) The assumption $\inf_{s\geq 1} \Var {\cal C}_k( \P_s, \beta s^{-1/\gamma} )/s > 0$ is satisfied if $\XX \subset \R^d$
is a full $d$-dimensional set and $g$ is a bounded probability density, as noted in the proof of Theorem 2.5 in Section 6 of \cite{PY6}.
If this assumption is not satisfied then we would have instead
\begin{align*}
& d_K \left( \frac{ {\cal C}_k( \P_s, \beta s^{-1/\gamma} ) - \E {\cal C}_k( \P_s, \beta s^{-1/\gamma} ) } {\sqrt{ \Var {\cal C}_k( \P_s, \beta s^{-1/\gamma} ) } } , N \right)\\
& \leq  \tilde{C}  \bigg(\frac{\sqrt{ s }}{\Var {\cal C}_k( \P_s, \beta s^{-1/\gamma} )}+\frac{ s }{(\Var {\cal C}_k( \P_s, \beta s^{-1/\gamma} ))^{3/2}}+\frac{s^{3/2}}{(\Var {\cal C}_k( \P_s, \beta s^{-1/\gamma} ))^{2}}\bigg), \quad s \geq 1.
\end{align*}
A similar comment applies for an underlying binomial point process in the situation where $\inf_{n\geq 9}\Var {\cal C}_k( \X_n, \beta n^{-1/\gamma} )/n> 0$ does not hold.

\begin{proof} To deduce Theorem \ref{cliquecount} from Corollary \ref{coroLSY}, we express ${\cal C}_k( \X, \beta s^{-1/\gamma} )$ as a sum of stabilizing score functions, which goes as follows.
Fix $\gamma, s, \beta \in (0, \infty)$. For $\X\in\mathbf{N}$ and $x\in\X$ let $\phi_{k,s}^{(\beta)}(x,\X)$ be the number of $(k+1)$-cliques containing $x$ in $G(\X, \beta s^{-1/\gamma})$ and such that $x$ is the point with the largest mark.
This gives the desired identification
$$
{\cal C}_k( \X, \beta s^{-1/\gamma} ) = \sum_{x \in \X} \phi_{k,s}^{(\beta)}(x,\X).
$$
Now we are ready to deduce \eqref{ccPoisson} and \eqref{ccBinomial} from Corollary \ref{coroLSY} with
the scores $\xi_s$ and $\xi_n$ set to $\phi_{k,s}^{(\beta)}$ and $\phi_{k,n}^{(\beta)}$, respectively, and with $K$ set to $\XX.$
Notice that $I_{K,s} = \Theta(s)$, as noted in \eqref{eqn:IXlinear}.
It is enough to show that $\phi_{k,s}^{(\beta)}$ and $\phi_{k,n}^{(\beta)}$ satisfy
all conditions of Corollary  \ref{coroLSY}.  Stabilization \eqref{eqn:expstabPoisson} is satisfied with
$\alpha_{stab} = a$, with  the radius of stabilization
$$
R_s((x,M_x), \P_s \cup \{(x,M_x)\}) = \beta s^{-1/\gamma} M_x,
$$
because $M_x$ has exponentially decaying tails as in
\eqref{exptail}.
For any $p>0$ we have
\begin{align*}
& \E  | \phi_{k,s}^{(\beta)}((x,M_x),\P_s\cup\{(x,M_x)\cup(\A,M_\A)\})|^{4+p}\\
 & \leq \E | \text{card} \{\P_s \cap B(x, \beta s^{-1/\gamma} M_x) \} +7|^{(4+p)k} \leq C(\beta, p, \gamma) < \infty
\end{align*}
for all $x\in\XX$, $s\geq 1$ and $\A\subset\XX$ with $|\A|\leq 7$
and so the $(4 + p)$th moment condition \eqref{eqn:momPoisson} holds for $p \in (0, \infty)$. The conclusion \eqref{ccPoisson} follows from \eqref{eqn:KolmogorovPoissoncoro}. The proof of \eqref{ccBinomial} is similar.
\end{proof}

\section{Appendix} Here we provide the proof of Lemma \ref{lm:intrinsic-volumes-representation}.

\begin{proof}
We need some additional notation. Throughout the proof, $\kappa $ is a constant depending on $d,j$, whose value may change from line to line. For $L$ some linear space, let $\ell^L$ the Lebesgue measure on $L$, $G(L,q),q<\text{\rm{dim}}(L)$ its space of $q$-dimensional subspaces, and $\nu_{q}^{L}$ the Haar probability measure on $G(L,q)$. Note $G(\mathbb{R}^{d},j)=G(d,j)$ and
 $\nu _{j}=\nu _{j}^{\mathbb{R}^{d}}$.
 Theorem  6.2.2 from \cite{SW} yields
\begin{align*}
V_{j}(A)-V_{j}(M)&=\kappa \int_{G(d,j)}(V_{j}(A | L )-V_{j}(M | L)) \, \nu _{j}(\dint L)\\
&=\kappa \int_{G(d,j)}\int_{  L}f^{L}(x) \, \ell^{L}(\dint x) \, \nu _{j}(\dint L).
\end{align*}
The Blaschke-Petkantschin formula (Theorem 7.2.1 in \cite{SW}) over the $\ell^{L}$ integral shows that the right-hand side equals
\begin{align*}
\kappa \int_{G(d,j)}\int_{G(L,1)}\int_{N}f^{L}(x)\|x\|^{j-1} \, \ell^{N}(\dint x) \, \nu _{1}^{L}(\dint N) \, \nu _{j}(\dint L).
\end{align*}
Fubini's theorem and Theorem 7.1.1 in \cite{SW} yield that the last expression is
\begin{align}
\notag&\kappa \int_{G(d,1)}\int_{G(N,j)}\int_{N}f^{L}(x)\|x\|^{j-1} \, \ell^{N}(\dint x) \, \nu_j^N(\dint L) \, \nu_{1}(\dint N)\\
\notag=&\kappa \int_{G(d,1)}\int_{N}\int_{G(N,j)}f^{L}(x)\|x\|^{j-1} \, \nu_j^N(\dint L) \, \ell^{N}(\dint x) \, \nu_{1}(\dint N)\\
\notag=&\kappa \int_{G(d,1)}\int_{N}\|x\|^{j-1}\int_{G(L(x),j)}f^{L}(x) \, \nu_j^{L(x)}(\dint L) \, \ell^{N}(\dint x) \, \nu_{1}(\dint N)\\
\label{eq:proof-BP}=&\kappa \int_{G(d,1)}\int_{N}f(x) \, \ell^{N}(\dint x) \, \nu_{1}(\dint N)\end{align}
with $f(x)=\|x\|^{j-1}\int_{G(L(x),j)}f^{L}(x) \, \nu_j^{L(x)}(\dint L)
$
because $N=L(x)$ in the second line. An independent application of the Blaschke-Petkantschin formula with $g(x)=f(x)\|x\|^{-(d-1)}  $ for each $L$ yields
\begin{align*}
\int_{\mathbb{R}^{d}}g(x) \, \ell^{d}(\dint x) & =\int_{G(d,1)}\int_{N} g(x)\|x\|^{d-1} \, \ell^{N}(\dint x) \, \nu_{1}(\dint N)\\
&=\int_{G(d,1)}\int_{N} f(x) \, \ell^{N}(\dint x) \, \nu_{1}(\dint N)
\end{align*}
whence \eqref{eq:proof-BP}  is equal to
$\int_{\mathbb{R}^{d}}\int_{G(L(x),j)}f^{L}(x)\|x\|^{(j-1)-(d-1)} \, \nu_j^{L(x)}(\dint L) \, \ell^{d}(\dint x),$
which completes the proof.
 \end{proof}

\end{document}